\documentclass[11pt]{amsart}

\usepackage[latin1]{inputenc}
\usepackage{amsmath,amsthm}
\usepackage{amssymb,amsfonts}
\usepackage{hyperref}
\usepackage{tikz}

\usepackage[shortlabels]{enumitem}
\usepackage{bm}
\usepackage{cite}
\usepackage{graphicx}
\usepackage{verbatim}
\usepackage{setspace}
\usepackage{color}
\usepackage{microtype}

\usepackage[mode=multiuser,status=draft,lang=english]{fixme}
\fxsetup{theme=colorsig}

\FXRegisterAuthor{M}{aM}{MV}

\usepackage[a4paper]{geometry}

\usepackage{stmaryrd}
\DeclareSymbolFont{bbold}{U}{bbold}{m}{n}
\DeclareSymbolFontAlphabet{\mathbbm}{bbold}

\title[]{Absolute model companionship, forcibility,
 and the continuum problem}
\author{Matteo Viale}
\thanks{
The author acknowledges support from INDAM through GNSAGA and from the project:
\emph{PRIN 2017-2017NWTM8R
Mathematical Logic: models, sets, computability.}
}

\theoremstyle{plain}
	\newtheorem{Theorem}{Theorem}[section]

	\newtheorem{Lemma}[Theorem]{Lemma}
	\newtheorem{Notation}[Theorem]{Notation}
	\newtheorem{Remark}[Theorem]{Remark}
	
	\newtheorem{theorem}{Theorem}[subsection]
	\newtheorem{proposition}[theorem]{Proposition}
	\newtheorem{lemma}[theorem]{Lemma}
	\newtheorem{corollary}[theorem]{Corollary}
	\newtheorem{fact}[theorem]{Fact}
	
	\newtheorem{conjecture}[theorem]{Conjecture}
	
	\newtheorem{claim}{Claim}
	
\theoremstyle{definition}
	\newtheorem{Definition}[Theorem]{Definition}

	\newtheorem{definition}[theorem]{Definition}
	\newtheorem{notation}[theorem]{Notation}
	\newtheorem{example}[theorem]{Example}

\theoremstyle{remark}
	\newtheorem{remark}[theorem]{Remark}

\newcommand{\Ord}{\ensuremath{\mathrm{Ord}}}

\newcommand{\ZFC}{\ensuremath{\mathsf{ZFC}}}
\newcommand{\ZF}{\ensuremath{\mathsf{ZF}}}

\newcommand{\WFE}{\ensuremath{\mathsf{WFE}}}

\DeclareMathOperator{\dom}{dom}
\DeclareMathOperator{\ran}{ran}
\DeclareMathOperator{\crit}{crit}
\DeclareMathOperator{\otp}{otp}

\DeclareMathOperator{\Ult}{Ult}

\DeclareMathOperator{\Coll}{Coll}

\newcommand{\maxUB}{\ensuremath{
{\mathbf{MAX}(\mathsf{UB})}}}
\newcommand{\maxA}{\ensuremath{
{\mathbf{MAX}(\mathsf{\mathcal{A}})}}}
\newcommand{\Pmax}{\ensuremath{\mathbb{P}_{\mathrm{max}}}}
\newcommand{\NS}{\ensuremath{\mathbf{NS}}} 
\newcommand{\stUB}{\ensuremath{(*)\text{-}\mathsf{UB}}}
\newcommand{\stA}{\ensuremath{(*)\text{-}\mathcal{A}}}

\newcommand{\bool}[1]{\mathsf{#1}}
\newcommand{\tow}[1]{\mathcal{#1}}

\newcommand{\SpecAMC}[1]{\mathfrak{spec}_{\mathsf{AMC}}\left(#1\right)}
\newcommand{\SpecMC}[1]{\mathfrak{spec}_{\mathsf{MC}}\left(#1\right)}
\newcommand{\pow}[1]{\mathcal{P}\left(#1\right)}

\newcommand{\qp}[1]{\left[ #1 \right]}
\newcommand{\Qp}[1]{\left\llbracket #1 \right\rrbracket}
\newcommand{\ap}[1]{\langle #1 \rangle}
\newcommand{\bp}[1]{\left\lbrace #1 \right\rbrace}

\newcommand{\Cod}{\ensuremath{\text{{\rm Cod}}}}

\newcommand{\UB}{\ensuremath{\text{{\sf UB}}}}

\newcommand{\MM}{\ensuremath{\text{{\sf MM}}}} 
\newcommand{\AX}{\ensuremath{\text{{\sf AX}}}} 
\newcommand{\CH}{\ensuremath{\text{{\sf CH}}}} 

\newcommand{\SSP}{\ensuremath{\text{{\sf SSP}}}}

\begin{document}

\begin{abstract}
Absolute model companionship (AMC) is a strict strengthening of model companionship defined as follows:
For a theory $T$, $T_{\exists\vee\forall}$ denotes the logical consequences of $T$ which are boolean combinations of universal sentences.
$S$ is the AMC of $T$ if it is model complete and $T_{\exists\vee\forall}=S_{\exists\vee\forall}$.

We use AMC to study the continuum problem and to gauge the expressive power of forcing.
We show that (a definable version of) $2^{\aleph_0}=\aleph_2$ is the unique solution to the continuum problem which can be in the AMC of a \emph{partial Morleyization} of the $\in$-theory $\ZFC+$\emph{there are class many supercompact cardinals}.

We also show that (assuming large cardinals) 
forcibility overlaps with the apparently weaker notion of consistency for any mathematical problem $\psi$ expressible as a $\Pi_2$-sentence of a (very large fragment of) third order arithmetic
($\CH$, the Suslin hypothesis, the Whitehead conjecture for free groups are a small sample of such problems $\psi$).

Finally we characterize a strong form of Woodin's axiom $(*)$ as the assertion that the first order theory of 
$H_{\aleph_2}$ as formalized in a certain natural signature is model complete.

 \end{abstract}

\maketitle

This paper is divided in two parts: the first part introduces two model theoretic concepts which are then used in the second part to analyze and gauge the complexity of the axiomatization of set theory given  by
$\ZFC$ (eventually enriched with large cardinal axioms). The key theme of the present work is to combine the ideas of Robinson around the notions of model companionship, model completeness, and existentially closed models for a first order theory with those arising in set theory from the analysis of the forcing method. Specifically we will show that a natural strengthening of the notion of model companionship is particularly fit to analyze 
$\ZFC$: on the one hand it can be used to infer that forcibility overlaps with consistency at least when dealing with a large family of interesting mathematical problems, on the other hand it gives a viable model theoretic tool to tackle the continuum problem and provides an argument to assert that $2^{\aleph_0}=\aleph_2$. Furthermore it provides the means to give an elegant model theoretic formulation of a natural strong form of Woodin's axiom $(*)$.

We now briefly outline the model theoretic content and the set theoretic content of the paper  trying to avoid technicalities\footnote{The model theory part of this paper is independent from its set theory part and can be read by anyone familiar with the basic facts about model companionship and model completeness. The set theory part depends on the model theory part and has two types of arguments: there are basic results whose proofs leverage on classical theorems covered in any master course on the subject (essentially Levy absoluteness, no knowledge of forcing required); there are also advanced results whose proofs require a strong background in forcing axioms and Woodin's work on axiom 
$(*)$ and take advantage of Asper\'o and Schindler's recent breakthrough \cite{ASPSCH(*)}. }.

\subsection*{Model theory}

\subsubsection*{Absolute model companionship}
 A first new model theoretic concept in this paper is that of absolute model companionship (AMC), which to our knowledge hasn't been explicitly stated yet.
Given a first order theory $T$ in a signature $\tau$, a $\tau$-structure $\mathcal{M}$ is $T$-existentially closed ($T$-ec) if and only if it is a substructure of a $\tau$-model of $T$ and it is a $\Sigma_1$-elementary substructure of any $\tau$-superstructure which models $T$. A standard example is given by $\bool{Fields}$ (the $\bp{+,\cdot,0,1}$-theory of fields), with the $\bool{Fields}$-ec models being exactly the algebraically closed fields.

The way algebraically closed fields sits inside the $\bp{+,\cdot,0,1}$-structures which are fields is described by Robinson's notion of model companionship: a $\tau$-theory $S$ is the model companion of a $\tau$-theory $T$ if the elementary class given by the $\tau$-models of $S$ consists exactly of the $T$-ec models.
The $\bp{+,\cdot,0,1}$-theory of algebraically closed fields $\bool{ACF}$
is the model companion of the $\bp{+,\cdot,0,1}$-theory $\bool{Fields}$; 
however for an arbitrary theory $T$ the $T$-ec models may not form an elementary class (e.g. the existentially closed models for the $\bp{\cdot,1}$-theory of groups are not an elementary class, hence this axiomatization of groups has no model companion). A very special case occurs when $T$ is its own model companion (e.g. the models of $T$ are exactly the $T$-ec models), in which case $T$ is \emph{model complete}.

A basic observation is that a $\tau$-structure $\mathcal{M}$ is $T$-ec if and only if it is $T_\forall$-ec (where $T_\forall$ consists of the universal consequences of $T$ in signature $\tau$) if and only if it is $T_{\forall\vee\exists}$ (the latter being the boolean combinations of universal sentences which follow from $T$).
Another non-trivial fact is that for a \emph{complete} theory $T$ 
a $T$-ec structure $\mathcal{M}$ realizes any $\Pi_2$-sentence
which holds true in some model of $T_{\forall\vee\exists}$.
A third non-trivial remark is that if $S$ is the model companion of $T$, $S$ is axiomatized by its $\Pi_2$-consequences.

Combining these three observations one is led to the speculation that the model companion $S$ of a $\tau$-theory $T$ (if it exists) could be axiomatized by the family of $\Pi_2$-sentences which holds in some model of $T_{\forall\vee\exists}$.
This is an assertion which is slightly too bold and holds true in case $T$ is a \emph{complete}, model companionable theory. However it can fail for non-complete model companionable theories; the standard counterexample being $\bool{ACF}$ versus $\bool{Fields}$: $\forall x \neg(x^2+1=0)$ is a $\Pi_2$-sentence (it is actually $\Pi_1$) which is not an axiom of $\bool{ACF}$ and holds in the field $\mathbb{Q}$, being therefore consistent with the universal and existential fragment of the theory of fields.

This brings us to introduce the notion of absolute model companionship (AMC): a $\tau$-theory $T$ has an AMC if the class of $T$-ec models is axiomatized by the $\Pi_2$-sentences which are consistent with the universal and existential fragments of any completion of $T$ (see Def. \ref{def:AMC}, Thm. \ref{Thm:AMCchar}). Complete first order theories are
model companionable if and only if they admit an AMC, but there are non-complete theories admitting a model companion but not an AMC (e.g. the $\bp{+,\cdot,0,1}$-theory of fields).

\subsubsection*{The AMC-spectrum of a first order theory}

A second model theoretic concept we introduce in this paper is that of
AMC-spectrum (and model companionship-spectrum) of a mathematical theory $T$.
Model theory is extremely successful in classifying the complexity of a mathematical theory according to its ``structural properties'' and has produced a variety of dividing lines to separate the so called ``tame'' mathematical theories from the others: typically a mathematical theory is considered ``wild'' or ``unclassifiable'' if it can code in itself first order arithmetic, hence the $\in$-theory $\ZFC$ is considered unclassifiable. On the other hand several tools have been developed to ``classify'' mathematical theories, for example stability, simplicity, NIP are structural properties of ``tame'' mathematical theories.

It is a matter of fact that most mathematical theories admit many different first order axiomatizations in many distinct signatures.  A common characteristic of ``tameness'' properties such as stability, simplicity, NIP is that they are \emph{signature invariant}: more precisely if we take a $\tau$-theory $T$ and we consider its Morleyization $T^*$ in the signature $\tau^*$ which adds predicate symbols and axioms granting that all $\tau$-formulae are equivalent to atomic $\tau^*$-formulae (see Notation \ref{not:keynotation0}), $T$ is 
stable (NIP, simple) if and only if so is $T^*$. In contrast Robinson's notion of model companionship
is a useful property of a first order theory, but is not signature invariant: for example  $\bool{ACF}$ is the model companion of $\bool{Fields}$ in signature 
$\tau=\bp{+,\cdot,0,1}$, but if we consider their Morleyizations $\bool{ACF}^*$ and $\bool{Fields}^*$ in signature $\tau^*$, it no longer holds true that  $\bool{ACF}^*$ is the model companion of $\bool{Fields}^*$. Conversely a $\tau$-theory $R$ may not have a model companion (e.g. the 
$\sigma=\bp{\cdot,1}$-theory of groups) but its Morleyization $R^*$ in signature $\sigma^*$ is its own model companion.

In this paper we consider the dependence of the existence of a model companion from the signature by itself a useful structural property of a mathematical theory. Specifically we introduce the notion of \emph{partial Morleyzation} of a $\tau$-theory $T$:
given a set $A$ of $\tau$-formulae, we consider the signature $\tau_A$ obtained by expanding $\tau$ exactly with predicate symbols for the formulae in $A$ and the $\tau_A$-theory $T_{\tau,A}$ with axioms asserting that every formula in $A$ is logically equivalent to the corresponding predicate of $\tau_A\setminus\tau$ (see again\footnote{Actually we give in this introduction a simpliefied version of the Morleyization procedures we consider; the full details can be found in Notation \ref{not:keynotation0}.} Notation \ref{not:keynotation0}). The AMC-spectrum of a $\tau$-theory $T$ (Def. \ref{Def:AMCSpec}) is given
by those sets $A$ of $\tau$-formulae for which $T+T_{\tau,A}$ admits an absolute model companion (mutatis mutandis we can define the model companionship spectrum of $T$).

A category theoretic perspective on this approach can be enlightening: given a $\tau$-theory $T$, consider the category $\mathcal{C}_T$ given by the elementary class of its $\tau$-models as objects and the $\tau$-morphisms (e.g. maps which preserve the \emph{atomic} $\tau$-formulae) between them as arrows. By taking $A$ a subset of the $\tau$-formulae we can pass to the category $\mathcal{C}_{T+T_{\tau,A}}$ whose objects are
 $\tau_A$-models of $T+T_{\tau,A}$ and whose arrows are the $\tau_A$-morphisms (in accordance with Notation \ref{not:keynotation0}). There is a natural identification of the objects of $\mathcal{C}_T$ and those of $\mathcal{C}_{T+T_{\tau,A}}$, but the arrows of $\mathcal{C}_{T+T_{\tau,A}}$ are now a possibly much narrower subfamily of the arrows of $\mathcal{C}_{T}$. This paper (by considering the case of $T$ being set theory) shows that important structural information on a $\tau$-theory $T$ is given (at least from a categorial point of view) by classifying which sets of $\tau$-formulae $A$ produce an elementary class of 
 $\tau_A$-models for which the $T+T_{\tau,A}$-ec models constitute themselves an elementary class.
 The model companionship spectrum of $T$ gives exactly this structural information on the arrows of $\mathcal{C}_T$.
 Similar considerations apply for the AMC-spectrum of a theory. On the other hand simplicity, stability, NIP
 provide fundamental structural informations on the class of objects of $\mathcal{C}_T$, but it is not transparent whether they also convey information on its class of arrows.

\subsection*{Set theory}

\subsubsection*{What is the right signature for set theory?}

The $\in$-signature is certainly sufficient to give by means of $\ZFC$ a first order axiomatization of set theory (with eventually 
other extra hypothesis such as large cardinal axioms), but we can see rightaway that it is not efficient to formalize many 
basic set theoretic concepts. Consider for example the notion of ordered pair: on the board we write $x=\ap{y,z}$ to mean that
\emph{$x$ is the ordered pair with first component $y$ and second component $z$}. In set theory this concept is formalized by means
of Kuratowski's trick stating that $x=\bp{\bp{y},\bp{y,z}}$. However the ``natural'' $\in$-formula in free variables $x,y,z$ formalizing the above is:
\[
\exists t\exists u\;[\forall w\,(w\in x\leftrightarrow w=t\vee w=u)
\wedge\forall v\,(v\in t\leftrightarrow v=y)
\wedge\forall v\,(v\in u\leftrightarrow v=y\vee v=z)].
\]
It is clear that the meaning of this $\in$-formula is hardly decodable with a rapid glance (unlike $x=\ap{y,z}$), moreover just 
from the point of view of its syntactic complexity it is already $\Sigma_2$.
On the other hand we do not regard the notion of ordered pair as a complex or doubtful concept (as is the case for the notion of uncountability, or many of the properties of the continuum such as its correct place in the hierarchy of uncountable cardinals, etc...). Similarly other very basic notions such as: being a function, a binary relation, the domain or the range of a function, etc.. are formalized by rather complicated $\in$-formulae, both from the point of view of readability for human beings, and from the mere computation of their syntactic complexity according to the Levy hierarchy.

The standard solution adopted by set theorists (e.g. \cite[Chapter IV, Def. 3.5]{KUNEN}) is to regard as elementary all 
those properties which can be formalized using $\in$-formulae all of whose quantifiers are bounded to range 
over the elements of some set, i.e. the so called $\Delta_0$-formulae. 
We adopt this point of view in stating our main set theoretic results, and we maintain it for the remainder of this paper,  considering $\in_{\Delta_0}$ the correct basic signature where to formalize set theory, where the latter is the signature obtained with the partial Morleyization induced by the $\Delta_0$-predicates (and slightly more, see Section \ref{subsec:AMCspecST} and Notation \ref{not:basicsettheorynot} for details).

\subsubsection*{Existentially closed fragments of the set theoretic universe versus AMC}

We now give a non-exhaustive list outlining on the one hand that for certain regular cardinals $\kappa$ 
     $H_\kappa$ is not that far from being an existentially closed structure for set theory as formalized in the appropriate signature
     (e.g. Levy absoluteness, Shoenfield's absoluteness, $\mathsf{BMM}$, $\mathsf{BMM}^{++}$) or 
     from having a model complete theory\footnote{Recall that a $\tau$-theory $T$ is model complete 
     if and only if the substructure relation between its models overlaps with
       the elementary substructure relation. In particular the mentioned results are weak form of 
        ``model completeness'' for the theory of $H_{\aleph_i}$ for $i=1,2$.} 
        (e.g. Woodin's absoluteness, $\MM^{+++}$). The reader needs not be familiar with these results, as they only serve as motivation for what we aim to do in the sequel.

Below we denote a $\tau$-structure $(M,R^M:R\in\tau)$ by $(M,\tau^M)$, $\sqsubseteq$ denotes the substructure relation, $\prec$ the elementary substructure relation, $\prec_1$ the $\Sigma_1$-elementary substructure relation (see Notation \ref{not:keynotation} for details).
Recall the following results:
\begin{description}
\item[Levy absoluteness] (see Lemma \ref{lem:levabsgen} below)
Whenever $\kappa$ is a regular uncountable cardinal,
\[
(H_\kappa,\in_{\Delta_0}^V)\prec_1 (V,\in_{\Delta_0}^V).
\] 
 \item[Shoenfield's absoluteness]  (see \cite[Lemma 1.2]{VIAMMREV} for the apparently weaker formulation we give here)
  Whenever $G$ is $V$-generic for some forcing notion in $V$,
 \[
 (H_{\omega_1},\in_{\Delta_0}^V)\prec_1 (V[G],\in_{\Delta_0}^{V[G]}).
 \]
  \item[Woodin's absoluteness]  (see \cite[Lemma 3.2]{VIAMMREV} for the weak form of Woodin's result we give here) 
  Whenever $G$ is $V$-generic for some forcing notion in $V$ (and there are class many Woodin cardinals in $V$),
  \[
  (H_{\omega_1}^V,\in_{\Delta_0}^V)\prec (H_{\omega_1}^{V[G]},\in_{\Delta_0}^{V[G]}).
  \]
   \item[Bounded Martin's Maximum ($\mathsf{BMM}$)] (see \cite{BAG00})
   Whenever $G$ is $V$-generic for some stationary set preserving forcing notion in $V$,
   \[
   (H_{\omega_2},\in_{\Delta_0}^V)\prec_1 (V[G],\in_{\Delta_0}^{V[G]}).
   \]
     \item[$\mathsf{BMM}^{++}$]  (see \cite[Def. 10.91]{WOOBOOK})
     Whenever $G$ is $V$-generic for some stationary set preserving forcing notion in $V$,
      \[
      (H_{\omega_2},\in_{\Delta_0}^V,\NS_{\omega_1}^V)
      \prec_1 (V[G],\in_{\Delta_0}^{V[G]},\NS_{\omega_1}^{V[G]}),
      \]
       where $\NS_{\omega_1}$ is a unary predicate symbol interpreted by the 
       non-stationary ideal on $\omega_1$.
      \item[Bounded category forcing axioms, $\MM^{+++}$, $\bool{RA}_\omega(\SSP)$]  (see\footnote{We omit a definition of these axioms since this demands a detour from the main focus of the present paper.} \cite{VIAASP,VIAAUD14,VIAMM+++})
       Whenever 
      $V$ and $V[G]$ are models of $\MM^{+++}$ ($\bool{RA}_\omega(\SSP)$, $\bool{BCFA}(\SSP)$) 
      and $G$ is $V$-generic for some stationary
	set preserving forcing notion in $V$,
      \[
      (H_{\omega_2}^V,\in_{\Delta_0}^V)\prec (H_{\omega_2}^{V[G]},\in_{\Delta_0}^{V[G]}).
      \]
     \end{description}

      We also want to mention:
     \begin{description}
     \item[Absoluteness between transitive models] (see \cite[Lemma 14.21]{JECHST})
whenever $G$ is $V$-generic for some forcing notion in $V$,
 \[
 (V,\in_{\Delta_0}^V)\sqsubseteq (V[G],\in_{\Delta_0}^{V[G]}).
 \]
     \end{description}
      This latter property entails that forcing 
     preserves the meaning of the basic  concepts of set theory\footnote{We ignored for the moment
	any consideration regarding universally Baire sets. 
	These sets will play an essential role in the proof of Thm. \ref{Thm: mainthmforcibility}.}, e.g. 
	those formalized by atomic $\in_{\Delta_0}$-formulae; furthermore Shoenfield's and Levy's absoluteness grant that $\Pi_1$-sentences for 
	$\in_{\Delta_0}$ do not change truth value in forcing extensions.

\subsection*{Main set theoretic results}
The set theoretic results of the present paper will be outlined in details in Section \ref{sec:mainresults} and systematize the above
 considerations. They can be informally summarized as follows:
\begin{itemize}
\item The theories of the various $H_\lambda$ for $\lambda$ an uncountable regular cardinal provide the prototypes of model companions for $\ZFC$ in some signature $\bp{\in}_A$ with $A$ in the model companionship spectrum of\footnote{As a side remark we note that Hirschfeld \cite{Hir} has proved that $\ZF$ has an AMC for the $\in$-signature; however he himself acknowledges that his result is not that informative on the properties of set theory, since the $\in$-AMC of $\ZF$ is given by a theory which is a small perturbation of the theory of dense linear orders (it interprets the $\in$-relation as an irreflexive transitive relation which defines a dense and  strict preorder). Hirschfeld does not argue that $\ZF$ has an AMC but he shows that it decides any universal $\in$-sentence and has a model companion; under these assumptions  the model companion of $\ZF$ is its AMC.
We believe that a signature in the AMC-spectrum of set theory is meaningful if it allows to prove
\ref{Thm:AMCsettheory+Repl-1} of Thm. \ref{Thm:AMCsettheory+Repl}, e.g. if it is able to show that the Goedel operations are well defined and that a large amount of replacement holds in the AMC of set theory according to the signature. This brings almost automatically to consider signatures which are more or less expansions of $\in_{\Delta_0}$: they must be able to express the concept of function, relation, domain, codomain, etc by means of terms or of quantifier free formulae.} $\ZFC$. More precisely:
for any definable cardinal $\kappa$, there is a set of formulae $A_\kappa$ in the AMC-spectrum of $\ZFC$ such that the theory of $H_{\kappa^+}$ is the model companion of set theory for $\bp{\in}_{A_\kappa}$ (see the second item of Thm. \ref{Thm:AMCsettheory+Repl} and the first part of Thm. \ref{Thm: mainthmforcibility}); note that for each such $\kappa$ $\bp{\in}_{A_\kappa}$ extends $\in_{\Delta_0}$. Furthermore any existentially closed structure for set theory
in some signature extending $\in_{\Delta_0}$ looks  like an $H_\lambda$ for some regular uncountable $\lambda$ (see the first item of Thm. \ref{Thm:AMCsettheory+Repl}).

\item
Forcing suffices to produce the interesting models of a very large fragment of the mathematical universe: 
for statements of second or third order arithmetic formalizable by $\Pi_2$-sentences of $\in_{\Delta_0}$ (among which for example the negation of the continuum hypothesis) their consistency (if at all possible) can already be obtained by forcing (by the first part of Thm.  \ref{Thm: mainthmforcibility}).

\item
The above results provide an argument for $2^{\aleph_0}=\aleph_2$ by on the basis of the following assertions:
\begin{enumerate}
\item \label{slogan1}
Large cardinal axioms are part of the accepted axioms of set theory.
\item \label{slogan2}
The mathematical universe realizes the $\Pi_2$-sentences for third order arithmetic which can hold in some model of set theory enriched with large cardinal axioms.
\end{enumerate} 
Note that assertion (\ref{slogan2}) above is the standard argument used to motivate forcing axioms and reproduces in the set theoretic framework the creation process of new numbers bringing from the universal $\bp{+,\cdot,0,1}$-theory $T$ of semirings without zero divisors (holding for the natural numbers) to the theory of algebraically closed fields (holding for the algebraic numbers): the algebraic numbers are obtained in a $\bp{+,\cdot,0,1}$-structure which realizes $T$ and all the ``interesting'' $\Pi_2$-sentences which can be individually made consistent with $T$, e.g. for each $n$
 \[
 \forall a_0,\dots,a_n\exists x\,(\sum_{i=0}^na_i x^i=0).
 \] 

On the basis of assertion \ref{slogan1}, assertion \ref{slogan2} (as well as $2^{\aleph_0}=\aleph_2$) can be validated as follows:
\begin{itemize}
\item Model theory (specifically the notion of AMC) gives us the means to formulate elegantly in precise mathematical terms assertion \ref{slogan2} (by Thm. \ref{Thm:AMCchar}).
\item
$2^{\aleph_0}=\aleph_2$ is the unique solution of the continuum problem which falls in the model companion of set theory enriched with large cardinals for at least one $\bp{\in}_A$ 
with $A$ in the AMC-spectrum of set theory 
(by Theorems \ref{Thm: mainthmforcibility}, \ref{mainthm:CH*}, \ref{mainthm:2omegageqomega2*}).
\item 
Furthermore the $\Pi_2$-consequences for $H_{\aleph_2}$ of $\bool{BMM}^{++}$, Woodin's axiom $(*)$, etc as formalized in an appropriate signature for third order arithmetic of the form $\bp{\in}_B$ (for a natural set $B$ of $\in$-formulae) form the AMC of set theory enriched with large cardinal axioms for $\bp{\in}_B$; moreover this AMC exactly overlaps with the forcible $\Pi_2$-sentences for $H_{\aleph_2}$  for this signature (see Thm.  \ref{Thm: mainthmforcibility},  Thm. \ref{Thm:mainthm-1bis}) and  has  among its $\Pi_2$-axioms one which entails that $2^{\aleph_0}=\aleph_2$ is witnessed by a definable well order (for example by the results of \cite{CAIVEL06,MOO06,TOD02}). 
\item
Finally it can also be shown that the assertion that the theory of $H_{\aleph_2}$ in (a natural expansion of signature $\bp{\in}_B$) is model complete (equivalently it is the AMC of set theory as formalized in this signature) gives a characterization of the strong form (here denoted as $\stUB$) of Woodin's axiom $(*)$ which predicates the existence of a generic filter for $\Pmax$ meeting all the dense subsets which are universally Baire in the codes (cfr. Def. \ref{def:stA}, Thm. \ref{Thm:mainthm-1bis}).
 \end{itemize}
  \end{itemize}
 
These results outline the role of forcing axioms in detecting a ``natural'' AMC for set theory describing the theory of $H_{\aleph_2}$. Key to their proof is the breakthrough \cite{ASPSCH(*)} by Asper\'o and Schindler that $\MM^{++}$ implies $\stUB$.



\section{Main results} \label{sec:mainresults}
We now give a precise and brief list of the main results of this paper (the proofs are deferred to later sections). 
We introduce hastily all new concepts and terminology needed to phrase them.
Our results are of three types:
\begin{description}
\item[Basic model theoretic results] we introduce AMC (a strengthening of model companionship) and we analyze its main properties. We also introduce the notion of partial Morleyization and we show how to combine it with AMC in order to get a useful classification tool for first order theories\footnote{This is elementary material accessible to anyone familiar with the notion of model companionship.}.
\item[Basic set theoretic results] we outline the general properties of theories which can be the model companion of partial Morleyizations of
$\ZFC$ leveraging on classical set theoretic results\footnote{A basic knowledge of set theory at the level of \cite{JECHHRBACEKBOOK} suffices to get through this part of the paper.}.
\item[Advanced set theoretic result] we show that set theory enriched with large cardinal axioms has an AMC with respect to a natural signature for third order arithmetic\footnote{This signature extends $\in_{\Delta_0}$ with predicates for the lightface definable universally Baire sets, a constant for $\omega_1$, and a unary predicate for the nonstationary ideal on $\omega_1$.}. Furthermore this AMC is given by the $\Pi_2$-sentences
of a very large fragment of third order arithmetic which are forcible.
We also analyze why this result provides an argument in favour of $2^{\aleph_0}=\aleph_2$ and gives a model theoretic characterization of a strengthening of Woodin's axiom\footnote{This part requires familiarity with forcing axioms, stationary tower forcing, $\Pmax$-technology, the theory of universally Baire sets assuming large cardinals.} $(*)$.
\end{description}
The model theory part is completely independent of the other two parts, while each of the set theoretic parts requires only the model theory part as a background.

\subsection{Absolute model companionship, partial Morleyizations, and the AMC-spectrum of a theory}\label{subsec:mainres-AMC}
We introduce a strengthening of model companionship which to our knowledge hasn't been explicitly stated yet. We refer the reader to 
\cite{VIAMODCOMPNOTES} or \cite[Section 3.5]{CHAKEI90},
\cite[Sections 3.1-3.2]{TENZIE} for a detailed account of model companionship\footnote{The reader will find a detailed analysis of the concepts and results we introduce here in Section \ref{sec-AMC+spectrum} (which could be read rightaway after section \ref{subsec:mainres-AMC} if one has no interest in the set-theoretic content of this article).}
\begin{Notation}
Let $\tau$ be a signature and $T$ be a $\tau$-theory.
\begin{itemize}
\item
 $\tau_\forall$ is the set of universal sentences for $\tau$. 
$\tau_{\forall\vee\exists}$ is the set of boolean combinations of sentences in $\tau_{\forall}$.
\item
$T_\forall$ (respectively
$T_{\forall\vee\exists}$) is the set of sentences in $\tau_{\forall}$ (respectively
$\tau_{\forall\vee\exists}$) which are logical consequences of $T$.
\end{itemize}
\end{Notation}

\begin{example}
In signature $\bp{+,\cdot,0,1}$ let $\bool{ACF}$ be the theory of algebraically closed fields and $\bool{Fields}$ be the theory of fields; we have that 
$\bool{ACF}_\forall=\bool{Fields}_\forall$ while 
$\bool{ACF}_{\forall\vee\exists}\supsetneq \bool{Fields}_{\forall\vee\exists}$:
$\exists x\,(x^2+1=0)$ is in the former but not in the latter.
\end{example}

\subsubsection{Absolute model companionship}
It is well known that if $T,S$ are $\tau$-theories such that $T_\forall= S_\forall$, then any
model of $T$ is a substructure of a model of $S$ and conversely. However the following holds as well (and to our knowledge hasn't been explicitly stated):

\begin{Lemma}
Let $\tau$ be a signature and $T,S$ be $\tau$-theories. TFAE:
\begin{enumerate}
\item $T_{\forall\vee\exists}\supseteq S_{\forall\vee\exists}$;
\item Every $\tau$-model $\mathcal{M}$ of $T$ is a $\tau$-substructure of a $\tau$-model $\mathcal{N}$ of $S$ such that $\mathcal{M}$ and $\mathcal{N}$ satisfy exactly the same sentences in
$\tau_\forall$ (or equivalently in $\tau_{\forall\vee\exists}$).
\end{enumerate}
\end{Lemma}

Recall that a theory $T$ is \emph{model complete} if the substructure relation between its models overlaps with the elementary substructure relation (equivalently if $T$ is its own model companion); recall also that $T$ is the \emph{model companion} of $S$ if and only if $T_\forall=S_\forall$ and $T$ is model complete. We introduce the following strengthening of model companionship:
\begin{Definition} \label{def:AMC}
Let $\tau$ be a signature and $T,S$ be $\tau$-theories.
$T$ is the \emph{absolute model companion} (AMC) of $S$ if the following conditions are met:
\begin{enumerate}
\item $T_{\forall\vee\exists}= S_{\forall\vee\exists}$;
\item $T$ is model complete.
\end{enumerate}
\end{Definition}

Note that:
\begin{itemize}
\item
 $\bool{ACF}$ is the model companion of $\bool{Fields}$ (since $\bool{ACF}$ is model complete and $\bool{ACF}_\forall=\bool{Fields}_\forall$)  but not its absolute model companion (as $\exists x(x^2+1=0)\in \bool{ACF}_{\exists\vee\forall}\setminus \bool{Fields}_{\exists\vee\forall}$).
 \item
A model complete theory $T$ is the model companion of $T_\forall$ and the AMC of $T_{\forall\vee\exists}$. 
\item
The uniqueness of the model companion grants that a theory $S$ admitting a model companion $T$ with $T_{\forall\vee\exists}\neq S_{\forall\vee\exists}$ cannot have an AMC.
\end{itemize}
The following result is what brought our attention on the notion of AMC:

\begin{Theorem}\label{Thm:AMCchar}
Let $\tau$ be a signature and $T,S$ be $\tau$-theories. TFAE:
\begin{enumerate}
\item $T$ is the AMC of $S$;
\item $T$ is the model companion of $S$ and is axiomatized by the $\Pi_2$-sentences $\psi$ such that
$\psi+R_{\forall\vee\exists}$ is consistent for all $\tau$-theories $R\supseteq S$.
\end{enumerate}
\end{Theorem}

In case $S$ is complete, $T$ is the AMC of $S$ if and only if it is its model companion, and the second item of the above equivalence states that the model companion of $S$ is axiomatized by the 
$\Pi_2$-sentences consistent with $S_{\forall\vee\exists}$.

The following motivates our terminology for this strengthening of model companionship:
\begin{Lemma}
Assume $S,S'$ are $\tau$-theories such that $S'$ is the AMC of $S$.
Then for all $T\supseteq S$, $S'+T_{\forall}$ is the AMC of $T$.
\end{Lemma}

Note that this characterization does not hold for the weaker notion of model companionship: 
for the signature $\tau=\bp{+,\cdot,0,1}$ $\bool{ACF}$ is the model companion of 
$\bool{Fields}$, but for $T$ the theory of the rationals in signature $\tau$,
$\bool{ACF}+T_{\forall}$ is inconsistent, therefore it cannot be the model companion of $T$.

\subsubsection{Partial Morleyizations and the AMC-spectrum of a theory}

We now introduce the notation we  use to relate AMC to set theory.

\begin{Notation}\label{not:keynotation0}
Given a signature $\tau$, 
let $\phi(x_0,\dots,x_n)$ be a $\tau$-formula.
  
We let:
\begin{itemize} 
\item
$R_\phi$ be
a new
$n+1$-ary relation symbol,
\item
$f_\phi$ be a
new $n$-ary function symbol\footnote{As usual we confuse $0$-ary function symbols with constants.}, 
\item
$c_\tau$ be a new
constant symbol.
\end{itemize}
 We also let:
\[
\AX^0_\phi\text{ be } \forall\vec{x}[\phi(\vec{x})\leftrightarrow R_\phi(\vec{x})],
\]
\begin{align*}
\AX^1_\phi\text{ be } &\forall x_1,\dots,x_n\,\\
&[(\exists!  y\phi(y,x_1,\dots,x_n)\rightarrow \phi(f_\phi(x_1,\dots,x_n),x_1,\dots,x_n))\wedge\\
&\wedge(\neg\exists!  y\phi(y,x_1,\dots,x_n)\rightarrow f_\phi(x_1,\dots,x_n)=c_\tau)]
\end{align*}
for $\phi(x_0,\dots,x_n)$ having at least two free variables, and
\[
\AX^1_\phi\text{ be } \qp{(\exists!  y\phi(y))\rightarrow \phi(f_\phi)}\wedge
 \qp{(\neg\exists!  y\phi(y))\rightarrow c_\tau=f_\phi}.
\]
for $\phi(x)$ having exactly one free variable.

Let $\bool{Form}_\tau$ denotes the set of $\tau$-formulae.
For $A\subseteq \bool{Form}_\tau\times 2$
 \begin{itemize}
 \item 
 $\tau_A$ is the signature obtained by adding to $\tau$ relation symbols $R_\phi$ for the
 $(\phi,0)\in A$ and function symbols  $f_\phi$ for the
 $(\phi,1)\in A$ (together with the special symbol $c_\tau$ if at least one $(\phi,1)$ is in $A$).
 \item
  $T_{\tau,A}$ is the $\tau_A$-theory having as axioms
the sentences $\AX^i_\phi$ for $(\phi,i)\in A$.
\end{itemize}

\end{Notation}

Note the following:
\begin{itemize} 
\item
$\bool{Ax}^0_\phi$ makes the $\tau$-formula $\phi$ logically equivalent to an atomic formula.
For any $\tau$-theory $T$, if
$C=\bool{Form}_\tau\times \bp{0}$ and $\tau^*=\tau_C$, then
$T^*=T+T_{\tau,A}$ is a $\tau^*$-theory admitting quantifier elimination (the Morleyization of $T$). 
\item
$\bool{Ax}^1_\psi$ introduces a Skolem function  for $\psi$ without appealing to choice: the key point is that the interpretation in some $\tau$-structure $\mathcal{M}$ of the function symbol $f_\psi$ applied to the parameters $\vec{a}$ is determined by $\psi$ only if there is a unique solution $b\in\mathcal{M}$ of the formula $\exists y\,\psi(\vec{a},y)$, otherwise $f^{\mathcal{M}}_\psi(\vec{a})=c_\tau^{\mathcal{M}}$. This avoids the possibility of choosing arbitrarily $f_\psi(\vec{a})$  when $\exists y\,\psi(\vec{a},y)$
has more than one solution, and gives that for each $a\in\mathcal{M}$
any 
$\tau$-structure admits exactly one expansion to a
$\tau_A$-structure which is a model of $T_{\tau,A}$ and interprets $c_\tau$ as $a$. 
%
\end{itemize}

In the sequel of this paper we are interested to analyze what happens when the Morleyization process is performed on arbitrary subsets of $\bool{Form}_\tau\times 2$.

\begin{Definition} \label{Def:AMCSpec}

The AMC-spectrum of a $\tau$-theory $T$ ($\SpecAMC{T}$) is given by those $A\subseteq \bool{Form}_\tau\times 2$ such that $T+T_{\tau,A}$ has an AMC (which we denote by $\bool{AMC}(T,A)$).

The MC-spectrum of a $\tau$-theory $T$ ($\SpecMC{T}$)  is given by those $A\subseteq \bool{Form}_\tau\times 2$ such that $T+T_{\tau,A}$ has a model companion (which we denote by $\bool{MC}(T,A)$).
\end{Definition}

Note that $A=\bool{Form}_\tau\times \bp{0}$ is always in the model companionship spectrum of a theory $T$ (as $T+T_{\tau,A}$ admits quantifier elimination, hence is model complete and its own AMC in signature $\tau_A$).
Note also that $\emptyset $ is in the (A)MC-spectrum of $T$ if and only if $T$ has a model companion (an AMC).

We now show how we can use AMC and model companionship to extract interesting information on the continuum problem and on the relation between forcibility and consistency. Some of the results below holds also if we consider the \emph{model companionship spectrum} of set theory, but the most interesting ones work just in case we focus on the AMC-spectrum\footnote{Moreover as of now we are not even able to produce an example of an $\in$-theory $T\supseteq\ZFC$ which has some $A\subseteq\bool{Form}\times 2$ in its model companionship spectrum but not in its AMC-spectrum.}.

\subsection{The AMC-spectrum of set theory} \label{subsec:AMCspecST}

From now on  for any 
$A\subseteq \bool{Form}_{\bp{\in}}\times 2$ 
we write $\in_A$ rather than $\bp{\in}_A$, and we let $T_{\in,A}$ be the $\in_A$-theory
\[
T_{\bp{\in},A}+\forall x \,\qp{(\forall y\,y\notin x)\leftrightarrow c_{\bp{\in}}=x},
\]
where the theory $T_{\bp{\in},A}$ (according to Notation \ref{not:keynotation0} for $\bp{\in}$ and $A$) is reinforced  by  
an axiom asserting that the interpretation of the constant symbol $c_{\bp{\in}}$ is the empty set.

We will be interested only in sets $A\subseteq \bool{Form}_{\bp{\in}}\times 2$ so that $\in_A$ contains a basic signature $\in_{\Delta_0}$ where all the basic set theoretic results can be developed (e.g. those exposed in \cite[Chapter IV]{KUNEN} and/or in \cite[Chapter 13]{JECHST}).
The specific details on $\in_{\Delta_0}$ and the axioms $T_{\Delta_0}$ which give the expected interpretation to the symbols of $\in_{\Delta_0}$ are given in Notation \ref{not:basicsettheorynot}, we anticipate here that:
\begin{itemize}
\item
$\in_{\Delta_0}$ is of the form $\in_D$ for a specific recursive set 
$D\subseteq \bool{Form}_{\bp{\in}}\times 2$ so that $\in_{\Delta_0}$ includes
constant symbols for $\emptyset,\omega$, relation symbols for all $\Delta_0$-formulae, function symbols for all Goedel operations as defined in \cite[Def. 13.6]{JECHST}.
\item
$T_{\Delta_0}$ is a family of $\Pi_2$-axioms for $\in_{\Delta_0}$ so that
$\ZF^-+T_{\Delta_0}$ is equivalent to $\ZF^-+T_{\in,D}$.
\end{itemize}
$\ZF^-$ ($\ZFC^-$) denotes the $\in$-theory $\ZF$ ($\ZFC$) deprived of the powerset axiom, 
$\ZF^-_{\Delta_0}$ (respectively $\ZFC^-_{\Delta_0}$, $\ZF_{\Delta_0}$, $\ZFC_{\Delta_0}$) denotes the $\in_{\Delta_0}$-theory $\ZF^-+T_{\Delta_0}$ (respectively $\ZFC^-+T_{\Delta_0}$, $\ZF+T_{\Delta_0}$, $\ZFC+T_{\Delta_0}$).


%
%

\begin{Definition}
Let $T\supseteq\ZFC^-$ be an $\in$-theory. $\kappa$ is a $T$-definable cardinal if for some $\in$-formula $\phi_\kappa(x)$, $T$ proves:
\begin{itemize} 
\item
$\exists!x\,\phi_\kappa(x)$ and
\[
\forall x\,[\phi_\kappa(x)\rightarrow (x\emph{ is a cardinal})].
\]
\item
$\kappa$ is the constant $f_{\phi_\kappa}$ existing in the signature $\in_{\bp{(\phi_\kappa,1)}}$.
\end{itemize} 
\end{Definition}

The first result shows that the AMC spectrum of set theory isolates a rich set of theories which produce models of
$\ZFC^-$, e.g. structures which behave like an $H_\lambda$ for some regular $\lambda$.

\begin{Theorem}\label{Thm:AMCsettheory+Repl}
Let $R$ be an $\in$-theory extending $\ZFC$. 
\begin{enumerate}[(i)]
\item \label{Thm:AMCsettheory+Repl-1}
Assume $A\in \SpecMC{R}$ and $\in_A\supseteq \in_{\Delta_0}$. Then  any model of 
$\bool{MC}(R,A)$ is closed under Goedel operations and satisfies Extensionality, Foundation, Infinity, and Choice\footnote{It is also a model of a large fragment of the Collection Principle, however it requires a bit of technicalities to sort out this fragment.}
\item \label{Thm:AMCsettheory+Repl-2}
Assume $\kappa$ is an $R$-definable cardinal.
Then there exists $A_\kappa\in \SpecAMC{R}$ with $\in_A\supseteq \in_{\Delta_0}$ and such that 
$\bool{AMC}(R,A_\kappa)$ is given by the $\in_{A_\kappa}$-theory common to the structures\footnote{$H_{\kappa^+}^\mathcal{M}$ denotes the substructure of $\mathcal{M}$ whose extension is given by the formula defining $H_{\kappa^+}$ in the model (using the parameter $\kappa$).} $H_{\kappa^+}^\mathcal{M}$ as $\mathcal{M}$ ranges among the $\in_{A_\kappa}$-models of $R+T_{\in,A_\kappa}$.
\end{enumerate}
\end{Theorem}

\subsection{Forcibility versus absolute model companionship}
The following is the major result of the paper\footnote{The reader unaware of what is $\MM^{++}$ or a stationary set preserving forcing can skip the second and third items of the theorem.}:

\begin{Theorem}\label{Thm: mainthmforcibility}
Let $S$ be the $\in$-theory 
\[
\ZFC+\emph{there exists class many supercompact cardinals}.
\]
Then there is a set $B\in\SpecAMC{S}$ with $\in_B\supseteq\in_{\Delta_0}$ and such that for any 
$\Pi_2$-sentence $\psi$ for $\in_B$ and any $\in$-theory $R\supseteq S$ the following are equivalent:
\begin{enumerate}[(a)]
\item $\psi\in \bool{AMC}(R,B)$;
\item\label{Thm: mainthmforcibility-b}
 $(R+T_{\in,B})_{\forall\vee\exists}+S+\MM^{++}+T_{\in,B}$ proves\footnote{Here and elsewhere we write $\psi^N$ to denote the relativization of $\psi$ to a definable class (or set) $N$; see \cite[Def. IV.2.1]{KUNEN} for details.} $\psi^{H_{\omega_2}}$;
\item $R$ proves that  $\psi^{H_{\omega_2}}$ is forcible\footnote{Here and in the next item we mean that the $\in$-formula $\theta$ which is $T_{\in,B}$-equivalent to $\psi$ is such that $\theta^{H_{\omega_2}}$ is forcible by the appropriate forcing.} by a stationary set preserving forcing;
\item $R$ proves that  $\psi^{H_{\omega_2}}$ is forcible by some forcing;
\item For any $R'\supseteq R$, $\psi+(R'+T_{\in,B})_{\forall\vee\exists}$ is consistent.
\end{enumerate}

Furthermore for any $\theta$ which is a boolean combination of $\Pi_1$-sentences for $\in_B$ and any
$(V,\in)$ model of $S$, TFAE:
\begin{enumerate}[(A)]
\item $(V,\in_B^V)$ models $\theta$;
\item $(V,\in_B^V)$ models that some forcing notion $P$ forces $\theta$;
\item $(V,\in_B^V)$ models that all forcing notions $P$ force $\theta$.
\end{enumerate}
\end{Theorem}
The second part of the theorem shows that forcing cannot change the $\Pi_1$-fragment of the theory of $V$ in signature $\in_B\supseteq\in_{\Delta_0}$.
Note also that if $(V,\in)$ is a model of $S$ and $R$ is the $\in_B$-theory of its unique extension to a model of 
$T_{\in,B}$, we get that a $\Pi_2$-sentence $\psi$ for $\in_B$ is consistent with the universal and existential fragments of $R$
if and only if $\psi^{H_{\omega_2}}$ is forcible over $V$. 

We will give a detailed definition of $\in_B$ at the proper stage (see Notation \ref{not:keysignaturesforinB}); we anticipate here that it is a recursive set extending 
$\in_{\Delta_0}$ with a predicate symbol for the non-stationary ideal on $\omega_1$, a constant symbol for $\omega_1$, a finite list of function and relation symbols for provably $\Delta_1$-definable operations and relations, and predicate symbols for all sets of reals definable by $\in$-formulae without parameters in the Chang model $L(\Ord^\omega)$ (which by a result of Woodin form an interesting subclass of the universally Baire sets, assuming the large cardinal hypothesis of the Theorem).

%
%

%


\subsection{The AMC-spectrum of set theory and the continuum problem}\label{subsec:AMCCH}

We show that large cardinals place $2^{\aleph_0}=\aleph_2$ in a very special position of the AMC spectrum of set theory.

We formalize $\CH$ and $2^{\aleph_0}>\aleph_2$  in signature $\in_{\Delta_0}$ as follows:
\begin{itemize}
\item
$(x\text{ is a cardinal})$ is the $\Pi_1$-formula 
\[
(x\text{ is an ordinal})\wedge
\forall f\,\qp{(f\text{ is a function}\wedge \dom(f)\in x)\rightarrow\ran(f)\neq x}.
\]
\item
$(x\text{ is }\aleph_1)$ is the boolean combination of $\Sigma_1$-formulae
\begin{align*}
(x\text{ is a cardinal})\wedge (\omega\in x) \wedge\\
\wedge\exists F\,\qp{(F:\omega\times x\to x)\wedge \forall\alpha\in x\,(F\restriction \omega\times\bp{\alpha}\text{ is a surjection on }\alpha)}.
\end{align*}
\item
$\CH$ is the $\Sigma_2$-sentence
\[
\exists f\, \qp{(f\text{ is a function}\wedge\dom(f)\text{ is }\aleph_1) \wedge \forall r\subseteq\omega\, (r\in\ran(f))}.
\]
and $\neg\CH$ is the  boolean combination of $\Pi_2$-sentences\footnote{We let $\neg\CH$ include the $\Sigma_2$-sentence $\exists x\,(x\text{ is }\aleph_1)$, for otherwise its failure could be witnessed by the assertion that there is no uncountable cardinal, a statement which holds true in $H_{\omega_1}$, regardless of whether $\CH$ or its negation is true in the corresponding universe of sets.}
\[
\exists x\,(x\text{ is }\aleph_1)\wedge\forall f\, \qp{(\dom(f)\text{ is }\aleph_1\wedge f\text{ is a function}) \rightarrow \exists r\subseteq\omega\, (r\not\in\ran(f))}.
\]
\item
$(x\text{ is }\aleph_2)$ is the $\Sigma_2$-formula
\begin{align*}
(x\text{ is a cardinal})\wedge \\
\wedge\exists F\exists y\,\qp{ (y\text{ is }\aleph_1)\wedge (y\in x)\wedge(F:y\times x\to x)\wedge \forall\alpha\in x\,(F\restriction y\times\bp{\alpha}\text{ is a surjection on }\alpha)}.
\end{align*}
\item
$2^{\aleph_0}>\aleph_2$ is the boolean combination of $\Pi_2$-sentences
\[
\exists x\,(x\text{ is }\aleph_2)\wedge\forall f\, \qp{(f\text{ is a function}\wedge\dom(f)\text{ is }\aleph_2) \rightarrow \exists r\,(r\subseteq\omega \wedge r\not\in\ran(f))}.
\]
\item
$2^{\aleph_0}\leq\aleph_2$ is the $\Sigma_2$-sentence
\[
\exists f\, \qp{(f\text{ is a function})\wedge\dom(f)\text{ is }\aleph_2 
\wedge \forall r\, (r\subseteq\omega\rightarrow r\in\ran(f))}.
\]
\end{itemize}

\begin{Theorem} \label{mainthm:CH*}
Let $S$ be the $\in$-theory of Thm. \ref{Thm: mainthmforcibility}.
The following holds:
\begin{enumerate} 
\item \label{mainthm:CH*-1}
Let  $R\supseteq S$ be an $\in$-theory. 
Assume $A\in \SpecAMC{R}$ with $\in_A\supseteq\in_{\Delta_0}$ and
$\neg\CH+(R+T_{\in,A})_{\forall\vee\exists}$ is consistent. 
Then $\CH\not\in \bool{AMC}(R,A)$.
\item \label{mainthm:CH*-2}
For the signature $\in_B$ of Thm. \ref{Thm: mainthmforcibility}
$\neg\CH$ is in $\bool{AMC}(R,B)$ for any $\in$-theory $R\supseteq S$.
\end{enumerate}
\end{Theorem}

We can prove exactly the same type of result replacing $\CH$ by $2^{\aleph_0}>\aleph_2$.
Specifically Moore introduced in \cite{MOO06} a $\Pi_2$-sentence $\theta_{\mathrm{Moore}}$ for $\in_{\Delta_0}$ to show the existence of a definable well order of the reals in type $\omega_2$ in models of the bounded proper forcing axiom\footnote{We use here the definable well-ordering of the reals in type $\omega_2$ existing in models of bounded forcing axioms isolated by Moore, but Thm. \ref{mainthm:2omegageqomega2*} could be proved replacing $\theta_{\mathrm{Moore}}$ with any other coding device which produce the same effects, for example those introduced in \cite{CAIVEL06,TOD02}, or the sentence $\psi_{\mathrm{AC}}$ of Woodin as in \cite[Section 6]{HSTLARSON}.}. 
  We can use  $\theta_{\mathrm{Moore}}$
as follows: 

\begin{Theorem}  \label{mainthm:2omegageqomega2*}
There is a $\Pi_2$-sentence $\theta_{\mathrm{Moore}}$  for $\in_{\Delta_0}$ such that
the following holds:
\begin{enumerate}
\item \label{mainthm:2omegageqomega2*-1}
$\theta_{\mathrm{Moore}}$ is independent of $S+T_{\Delta_0}$, where $S$ is the $\in$-theory of Thm. \ref{Thm: mainthmforcibility}. 
\item \label{mainthm:2omegageqomega2*-2}
$\ZFC^-_{\Delta_0}+\exists x\,(x\text{ is }\aleph_1)+\theta_{\mathrm{Moore}}$ proves that there exists a well-ordering of\footnote{More precisely: there is a a  $\ZFC^-_{\Delta_0}$-provably $\Delta_1$-property $\psi(x,y,z)$ such that in any model $\mathcal{M}$ of the mentioned theory there is a parameter $d\in\mathcal{M}$ such that $\psi(x,y,d)$ defines an injection of $\pow{\omega}$ of the model 
with the class of ordinals of size at most $\omega_1$ of the model.} $\pow{\omega}$ in type at most $\omega_2$.
\item \label{mainthm:2omegageqomega2*-3}
$\ZFC^-_{\Delta_0}+\exists x\,(x\text{ is }\aleph_2)+\theta_{\mathrm{Moore}}$ proves that $2^{\aleph_0}\leq\omega_2$.
\item \label{mainthm:2omegageqomega2*-4}
For $S$ and $\in_B$ the theory and signature considered in Thm. \ref{Thm: mainthmforcibility},
$\exists x\,(x\text{ is }\aleph_1),\theta_{\mathrm{Moore}}$ are both in $\bool{AMC}(R,B)$ for any $\in$-theory $R$ extending $S$.
\item \label{mainthm:2omegageqomega2*-5}
If $R$ extends $S$, $A\in\SpecAMC{R}$ is such that $\in_A\supseteq\in_{\Delta_0}$, $\exists x\, (x\text{ is }\aleph_2)\in \bool{AMC}(R,A)$, $\bool{AMC}(R,A)$ implies the Replacement schema, and
\[
\theta_{\mathrm{Moore}}+(R+T_{\in,A})_{\forall\vee\exists}+\ZFC
\]
is consistent, then $2^{\aleph_0}>\aleph_2$ is not in $\bool{AMC}(R,A)$.
\end{enumerate}
\end{Theorem}

The two theorems show an asimmetry between $2^{\aleph_0}=\aleph_2$ and all other solutions of the continuum problem assuming large cardinals: for any $R$ extending $\ZFC+$\emph{large cardinals} there is at least one $B\in\SpecAMC{R}$ with $\neg\CH$ (and a definable version of $2^{\aleph_0}\leq\aleph_2$) in $\bool{AMC}(R,B)$, and this occurs even if $R\models\CH$ or $R\models 2^{\aleph_0}>\aleph_2$.
On the other hand if $\CH$ is independent of $R$, $\CH$ is never in $\bool{AMC}(R,A)$ for any $A\in\SpecAMC{R}$ (with $\in_A\supseteq\in_{\Delta_0}$)
and similarly if $\theta_{\mathrm{Moore}}$ is independent of $R$, $2^{\aleph_0}>\aleph_2$  is never in $\bool{AMC}(R,A)$ for any $A\in\SpecAMC{R}$ (with $\in_A\supseteq\in_{\Delta_0}$ and $\bool{AMC}(R,A)$ satisfying the Replacement schema).

Furthermore the last part of Thm. \ref{Thm: mainthmforcibility} outlines that
$\CH$, $2^{\aleph_0}=\aleph_2$, $2^{\aleph_0}>\aleph_2$, $\theta_{\mathrm{Moore}}$ are all 
boolean combination of $\Pi_2$-sentences in the signature $\in_{\Delta_0}$ which cannot be expressed by
boolean combination of $\Pi_1$-sentences for the signature $\in_B\supseteq\in_{\Delta_0}$ in models of $S$ (with $S$ and $B$ as in Thm. \ref{Thm: mainthmforcibility}), 
as their truth value can be changed by forcing.

\begin{remark}
The above results are in the same spirit of \cite{ASPLARMOO12} where it is shown that one can find two $\in$-sentences $\psi_0,\psi_1$ which are $\Pi_2$ for the signature $\in_{\Delta_0}\cup\bp{\omega_1}$ and such that $\ZFC+\bool{LC}+\CH+\psi_i$ (where $\bool{LC}$ is a shorthand for sufficiently strong large cardinal axioms)  are separately consistent (and both forcible by a semiproper forcing), but $\ZFC+\psi_0+\psi_1\vdash\neg\CH$. This latter result shows that the theory given by the $\Pi_2$-sentences for $\in_B$ which are individually consistent with $\CH$ cannot hold in the $H_{\aleph_2}$ of a model of $\ZFC+T_{\in,B}+\bool{LC}+\CH$, while our results show that the theory of $\Pi_2$-sentences for $\in_B$ which are strongly consistent with the $\Pi_1$-fragment for $\in_B$ of $\ZFC+\bool{LC}+T_{\in,B}$ is not only consistent, but is its AMC (note also that $\psi_0,\psi_1$ are both axioms of this AMC).
\end{remark}

\subsection{Model theoretic formulation of a strong form of Woodin's axiom $(*)$}

Woodin's axiom $(*)$ is a technical statement which provides a model of set theory in which any ``reasonable'' $\Pi_2$-sentence for $H_{\aleph_2}$ which is forcible holds true in the $H_{\aleph_2}$ of the model. In view of our previous considerations the following result is not a suprise and provides a formulation of a strong form of $(*)$ accessible to anyone able to grasp the meaning of the following concepts: Chang model, stationarity on $\omega_1$, saturation of the non-stationary ideal\footnote{$\NS\emph{ is saturated}$ asserts that the boolean algebra $\pow{\omega_1}/\NS$ has only antichains of size at most $\aleph_1$.} $\NS$ on $\omega_1$, model completeness:
\begin{Theorem}\label{Thm:mainthm-1bis0}
Let $S$ and $B$ be as in Thm. \ref{Thm: mainthmforcibility}.
Assume $(V,\in)$ models 
\(
S+\NS\emph{ is saturated}.
\)

TFAE:
\begin{enumerate}
\item\label{thm:char(*)-modcomp-1}
$(V,\in)$ models\footnote{See Def. \ref{def:stA} for details, roughly it states that there exists a generic filter for $\Pmax$ meeting all dense subsets of this forcing which belong to the Chang model $L(\Ord^\omega)$.} $\stA$ for $\mathcal{A}$ being the powerset of $\mathbb{R}$ as computed in its Chang model;
\item\label{thm:char(*)-modcomp-2}
The $\in_{B}\cup\pow{\mathbb{R}}^{L(\Ord^\omega)}$-theory\footnote{E.g. we regard $A\subseteq \mathbb{R}^k$ as a $k$-ary predicate symbol for any $A\subseteq \mathbb{R}^k$ which belongs to $L(\Ord^\omega)$. For more details on $\in_B$ see Notation \ref{not:keysignaturesforinB} or the comments after Thm. \ref{Thm: mainthmforcibility}.} of $H_{\omega_2}$ is model complete.
\end{enumerate}
\end{Theorem}

\section*{Structure of the paper}

The remainder of this paper is organized as follows:
\begin{itemize}
\item
Section \ref{sec-AMC+spectrum} gives an account on the main properties of AMC.
\item
Section \ref{sec-AMC-set-theory} proves Theorem \ref{Thm:AMCsettheory+Repl}, part \ref{mainthm:CH*-1} of 
Thm. \ref{mainthm:CH*}, and Thm. \ref{mainthm:2omegageqomega2*} (with the exception of part
\ref{mainthm:2omegageqomega2*-4}).
 \item
Section \ref{sec:geninv} proves the 
generic invariance of the theory of $V$ in signature $\in_{\Delta_0}\cup\bp{\omega_1,\NS_{\omega_1}}$
(more precisely a stronger version of the second part of Thm. \ref{Thm: mainthmforcibility}). 

\item
Section \ref{sec:Homega2} completes the proof of Thm. \ref{Thm: mainthmforcibility} (and of the missing parts of Theorems \ref{mainthm:CH*}, \ref{mainthm:2omegageqomega2*}). It also provides the proof of Thm. \ref{Thm:mainthm-1bis0}.
\item
We conclude the paper with some comments and open questions.
\end{itemize}
Any reader familiar enough with model companionship to follow Section \ref{subsec:mainres-AMC} can easily grasp the content of Section \ref{sec-AMC+spectrum}. Section \ref{sec-AMC-set-theory} needs Section \ref{sec-AMC+spectrum} and familiarity with set theory at the level of  \cite[Chapters I, III, IV]{KUNEN}.
 (no knowledge of forcing is required).
The proofs in Section \ref{sec:geninv} assume familiarity with Woodin's stationary tower forcing, and 
(in its second part, cfr. Section \ref{subsec:geninvtoa}) also with some bits of Woodin's $\Pmax$-technology. Section \ref{sec:Homega2} can be fully appreciated only by readers familiar with forcing axioms, Woodin's stationary tower forcing, Woodin's $\Pmax$-technology and take advantage of Asper\'o and Schindler's proof that $\MM^{++}$ implies a strong form of Woodin's axiom $(*)$ \cite{ASPSCH(*)}.


\section{Existentially closed structures and absolute model companionship}\label{sec-AMC+spectrum}


This section proves the model theoretic results on model companionship stated in Section \ref{subsec:mainres-AMC}. 
Absolute model companionship (AMC) isolates those (possibly non-complete) theories $T$ whose model companion is axiomatized by the $\Pi_2$-sentences which are consistent with 
the universal and existential fragments of any model of $T$. 
We also show that AMC is strictly stronger than model companionship and does not imply (nor is implied by) model
completion\footnote{A detailed account of model companionship, Kaiser Hulls, existentially closed structures in line with our treatment of this topic can be found in the notes \cite{VIAMODCOMPNOTES}.}. 

We introduce the following terminology:
\begin{notation} \label{not:keynotation}
\emph{}

\begin{itemize}
\item
$\sqsubseteq$ denotes the substructure relation between structures.
\item
$\mathcal{M}\prec_n\mathcal{N}$ indicates that $\mathcal{M}$ is a $\Sigma_n$-elementary substructure of $\mathcal{N}$, we omit the $n$ to denote full-elementarity.
\item Similarly for $\tau$-structures $\mathcal{M},\mathcal{N}$, $\mathcal{M}\equiv_n\mathcal{N}$ if both satisfy exactly the same $\Sigma_n$-sentences,  $\mathcal{M}\equiv\mathcal{N}$ if they are elementarily equivalent.
\item
Given a first order signature $\tau$,
$\tau_\forall$  denotes the universal $\tau$-sentences; likewise we interpret $\tau_\exists, \tau_{\forall\exists},\dots$. $\tau_{\forall\vee\exists}$  denotes the boolean combinations of universal $\tau$-sentences; likewise we interpret $\tau_{\forall\exists\vee\exists\forall},\dots$. 
\item
Given a first order theory $T$, $T_\forall$ denotes the sentences in $\tau_\forall$ 
which are consequences of $T$,
likewise we interpret $T_\exists, T_{\forall\exists},T_{\forall\vee\exists},\dots$.
\item We often denote a $\tau$-structure $(M,R^M:R\in\tau)$ by $(M,\tau^M)$.
\item
We often identify a $\tau$-structure
$\mathcal{M}=(M,\tau^M)$ with its domain $M$
and an ordered tuple $\vec{a}\in M^{<\omega}$
with its set of elements. 
\item
We often
write $\mathcal{M}\models\phi(\vec{a})$ rather than
$\mathcal{M}\models\phi(\vec{x})[\vec{x}/\vec{a}]$ when
$\mathcal{M}$ is $\tau$-structure $\vec{a}\in\mathcal{M}^{<\omega}$, $\phi$ is a $\tau$-formula.
\item
We let the atomic diagram $\Delta_0(\mathcal{M})$ of a $\tau$-model $\mathcal{M}=(M,\tau^M)$ be the family of quantifier free sentences $\phi(\vec{a})$
in signature $\tau\cup M$ such that $\mathcal{M}\models\phi(\vec{a})$.
\end{itemize}
\end{notation}

\subsection{Byembeddability versus absolute byembeddability}


Let us give a proof of the following well known fact, since it will be helpful to outline the subtle difference between model companionship and absolute model companionship.

\begin{lemma}\label{lem:coth}
Let $\tau$ be a signature  and $T$, $S$ be $\tau$-theories. TFAE:
\begin{enumerate}
\item\label{lem:coth1}
 $T_\forall\supseteq S_\forall$.
 \item\label{lem:coth2}
 For any $\mathcal{M}$ model of $T$ there is $\mathcal{N}$ model of $S$ superstructure of 
 $\mathcal{M}$.
\end{enumerate}
\end{lemma}
\begin{proof}
\emph{}

\begin{description}
\item[\ref{lem:coth1} implies \ref{lem:coth2}]
Assume  $\mathcal{M}$ models $T$ and is such that no $\mathcal{N}$ model of $S$ is a superstructure of $\mathcal{M}$.
Then $S\cup\Delta_0(\mathcal{M})$ is not consistent (where $\Delta_0(\mathcal{M})$ is the atomic diagram of $\mathcal{M}$). By compactness find $\psi(\vec{a})\in\Delta_0(\mathcal{M})$ quantifier-free sentence such that $S+\psi(\vec{a})$ is inconsistent.
This gives that 
\[
S\models\forall\vec{x}\,\neg\psi(\vec{x})
\]
since $\vec{a}$ is a string of constant symbols all outside of $\tau$.
Therefore $\forall\vec{x}\,\neg\psi(\vec{x})\in S_\forall\subseteq T_\forall$.
Hence 
\[
\mathcal{M}\models \forall\vec{x}\,\neg\psi(\vec{x})\wedge\psi(\vec{a}),
\]
a contradiction.

\item[\ref{lem:coth2} implies \ref{lem:coth1}] Left to the reader.
\end{description}
\end{proof}

The following is a natural question: assume $S$ and $T$ are $\tau$-theories such that
$T_\forall=S_\forall$, can we extend a model $\mathcal{M}$ of $T$ to a superstructure $\mathcal{M}$ of $S$ so that $\mathcal{M}$ and $\mathcal{N}$ satisfy exactly the same universal sentences?
The answer is no as shown by $\tau=\bp{\cdot,+,0,1}$, $T$ the $\tau$-theory of fields, $S$ the $\tau$-theory of algebraically closed fields: it is easy to see that $T_\forall=S_\forall$ in view of Lemma \ref{lem:coth2}, but $\mathbb{Q}$ cannot be extended to an algebraically closed field without killing the universal $\tau$-sentence stating the non existence of the square root of $-1$.

The clarification of this issue is what has brought our attention to $T_{\forall\vee\exists}$.

Note that any sentence in $T_ {\forall\vee\exists}$ is either logically equivalent to  
$\theta\vee\psi$ or equivalent to  $\theta\wedge\psi$ with
 $\theta$ universal and $\psi$ existential.
 
Note also that 
$T_{\forall\vee\exists}$ may contain more information than $T_{\forall}\cup T_{\exists}$ as there could be a universal 
$\theta\not\in T_\forall$ and an 
existential $\psi\not\in T_\exists$ with $\theta\vee\psi\in T_{\forall\vee\exists}$.

\begin{lemma}\label{lem:abscoth}
Let $\tau$ be a signature  and $T$, $S$ be $\tau$-theories. TFAE:
\begin{enumerate}
\item\label{lem:abscoth1}
 $T_{\forall\vee\exists}\supseteq S_{\forall\vee\exists}$.
 \item\label{lem:abscoth2}
 For any $\mathcal{M}$ model of $T$ there is $\mathcal{N}$ model of $S$ superstructure of 
 $\mathcal{M}$ realizing exactly the same universal sentences.
\item \label{lem:abscoth3}
For every boolean combination of universal sentences $\theta$, $T+\theta$ is consistent only if so is $S+\theta$.
\end{enumerate}
\end{lemma}
\begin{proof}
\emph{}

\begin{description}
\item[\ref{lem:abscoth1} implies \ref{lem:abscoth2}]
Assume  $\mathcal{M}$ models $T$ and is such that no $\mathcal{N}$ model of $S$ which is a superstructure of $\mathcal{M}$ realizes exactly the same universal sentences.

For any such $\mathcal{N}$ with $\mathcal{M}\sqsubseteq\mathcal{N}\models S$ we get that some universal $\tau$-sentence $\theta_{\mathcal{N}}$ true in $\mathcal{M}$ fails
in $\mathcal{N}$. We claim that the $\tau\cup\mathcal{M}$-theory
\[
S^*=\Delta_0(\mathcal{M})\cup S\cup\bp{\theta_{\mathcal{N}}: \mathcal{M}\sqsubseteq\mathcal{N},\, \mathcal{N}\models S }
\]
is inconsistent.
If not let $\mathcal{P}^*$ be a model of $S^*$. Then $\mathcal{P}=(\mathcal{P}^*\restriction\tau)\sqsupseteq\mathcal{M}$ is a model of 
\[
 S\cup\bp{\theta_{\mathcal{N}}: \mathcal{M}\sqsubseteq\mathcal{N},\, \mathcal{N}\models S }.
\]
Hence it models
$\theta_{\mathcal{P}}$ and $\neg\theta_{\mathcal{P}}$ at the same time.

By compactness we can find a universal sentence $\phi_\mathcal{M}$ given by the conjunction of a finite set 
\[
\bp{\theta_{\mathcal{P}_i}:\,i=1,\dots,n, \mathcal{M}\sqsubseteq\mathcal{P}_i\models S}
\]
and a quantifier free sentence
$\psi_{\mathcal{M}}(\vec{a})$ of $\Delta_0(\mathcal{M})$ such that 
\[
S+\psi_{\mathcal{M}}(\vec{a})+\phi_\mathcal{M}
\]
is inconsistent.
Hence
\[
S\models\neg\phi_\mathcal{M}\vee\neg \exists\vec{x}\,\psi_{\mathcal{M}}(\vec{x}).
\]
Now observe that:
\begin{itemize} 
\item
$\neg\phi_\mathcal{M}\vee\neg\exists\vec{x}\,\psi_{\mathcal{M}}(\vec{x})$ 
is a boolean combination of universal sentences, 
\item
$\mathcal{M}\models T+\exists\vec{x}\,\psi_{\mathcal{M}}(\vec{x})\wedge\phi_{\mathcal{M}}$.
\end{itemize}
Therefore we get that $\neg\phi_\mathcal{M}\vee\neg \exists\vec{x}\,\psi_{\mathcal{M}}(\vec{x})$ is in $S_{\forall\vee\exists}\setminus T_{\forall\vee\exists}$.


\item[\ref{lem:abscoth2} implies \ref{lem:abscoth3}] Left to the reader.

\item[\ref{lem:abscoth3} implies \ref{lem:abscoth1}]
If $T_{\forall\vee\exists}\not\supseteq S_{\forall\vee\exists}$ there is
 $\theta\in S_{\forall\vee\exists}\setminus T_{\forall\vee\exists}$. Then $\neg\theta$ is inconsistent with $S$ and consistent with $T$.
\end{description}
\end{proof}

\begin{definition}
Let $\tau$ be a signature and $T,S$ be $\tau$-theories.
\begin{itemize}
\item
$T$ and $S$ are \emph{cotheories} if $T_\forall=S_\forall$.
\item
$T$ and $S$ are \emph{absolute cotheories} if $T_{\forall\vee\exists}=S_{\forall\vee\exists}$.
\end{itemize}
\end{definition}

\begin{remark}
Say that a $\tau$-theory $T$ is \emph{$\Pi_1$-complete} if
$T\vdash\phi$ or $T\vdash\neg\phi$  for any universal $\tau$-sentence $\phi$. 

Now consider the
$\bp{+,\cdot,0,1}$-theories $\bool{ACF}_0$ and $\bool{Fields}_0$ expanding  $\bool{ACF}$ and 
$\bool{Fields}$ with the axioms fixing the characteristic of their models to be $0$. 
Note that $\bool{ACF}_0$ is $\Pi_1$-complete (it is actually complete) while $\bool{Fields}_0$ is not, even if
$(\bool{ACF}_0)_\forall=(\bool{Fields}_0)_\forall$.
In particular $T_\forall=S_\forall$ is well possible with $T$ $\Pi_1$-complete and $S$ not $\Pi_1$-complete.
Absolute cotheories rule out this confusing discrepancy.
In particular we will use the following trivial fact crucially in the proof of Lemma \ref{fac:proofthm1-2}:
if $S$ is a complete theory $S_{\forall\vee\exists}$ is $\Pi_1$-complete, while $S_\forall$ may not.
\end{remark}

%
%



\subsection{Existentially closed models, Kaiser hulls, strong consistency}

\begin{definition}
Given a $\tau$-theory $T$, $\mathcal{M}$ is $T$-existentially closed ($T$-ec) if:
\begin{itemize}
\item
There is some 
$\mathcal{N}\sqsupseteq\mathcal{M}$ which models $T$.
\item
$\mathcal{M}\prec_1\mathcal{N}$ for all superstructures $\mathcal{N}\sqsupseteq\mathcal{M}$ which model $T$.
\end{itemize}
\end{definition}

\begin{remark}
Among the many nice properties of $T$-ec structures, note that $\Pi_2$-sentences (with parameters in $\mathcal{M}$) which hold in
some $T$-model $\mathcal{N}$
superstructure of $\mathcal{M}$ reflect to $\mathcal{M}$.
\end{remark}

In view of the above Lemmas it is not hard to check the following:
\begin{fact}\label{fac:Tecmodels}
TFAE for a $\tau$-theory $T$ and a $\tau$-structure $\mathcal{M}$:
\begin{enumerate}
\item
$\mathcal{M}$ is $T$-ec.
\item
$\mathcal{M}$ is $T_{\forall\vee\exists}$-ec.
\item
$\mathcal{M}$ is $T_\forall$-ec.
\end{enumerate}
\end{fact}

\begin{definition}\label{def:SCHKH}
Let $T$ be a $\tau$-theory. 
\begin{itemize}
\item
A $\tau$-sentence $\psi$ is \emph{strongly $T_{\forall\vee\exists}$-consistent} if
$\psi+R_{\forall\vee\exists}$ is consistent for all $R\supseteq T$.
\item
The \emph{Kaiser hull of $T$} ($\bool{KH}(T)$) consists of the $\Pi_2$-sentences for $\tau$ which hold in all $T$-ec models.
\item
The \emph{strong consistency hull of $T$} ($\bool{SCH}(T)$) consists of the $\Pi_2$-sentences for $\tau$ which are strongly $T_{\forall\vee\exists}$-consistent.
\end{itemize}
\end{definition}
The Kaiser hull of a theory is a well known notion describing an equivalent of model companionship which can be defined also for non-companionable theories (see for example \cite[Lemma 3.2.12, Lemma 3.2.13, Thm. 3.2.14]{TENZIE}); the strong consistency hull is a slight weakening of the Kaiser hull not considered till now (at least to my knowledge) and which does the same with respect to the notion of absolute model companionship (defined below in Def. \ref{def:AMC+MC}).

\begin{fact}\label{rem:SCHKH}
For any $\tau$-theory $T$:
\begin{enumerate}[(i)]
\item \label{rem:SCHKH-0}
$\bool{KH}(T)_{\forall}= T_{\forall}$;
\item \label{rem:SCHKH-1}
$\bool{SCH}(T)_{\forall\vee\exists}= T_{\forall\vee\exists}$;
\item \label{rem:SCHKH-2}
$\bool{SCH}(T)\subseteq\bool{KH}(T)$;
\item \label{rem:SCHKH-3}
$\bool{SCH}(T)=\bool{KH}(T)$ if $T_{\forall\vee\exists}=\bool{KH}(T)_{\forall\vee\exists}$, which is the case if $T$ is complete. 
\item \label{rem:SCHKH-4}
For any $\Pi_2$-sentence $\psi$ such that $T_{\forall\vee\exists}+\psi$ is consistent, there is a model of 
$\bool{SCH}(T)+\psi$.
\end{enumerate}
\end{fact}
\begin{proof}
\emph{}

\begin{enumerate}[(i)]
\item
By definition any model of $\bool{KH}(T)$ is a model of $T_\forall$; conversely any model of $T$ can be extended to a $T$-ec model (see for example \cite[Lemma 3.2.11]{TENZIE}).
We conclude by Lemma \ref{lem:coth}.
\item
Trivial.
\item
Assume a $\Pi_2$-sentence $\psi$ is strongly $T_{\forall\vee\exists}$-consistent.
Let $\mathcal{M}$ be a $T$-ec model. Then $\mathcal{M}$ is $T_{\forall\vee\exists}$-ec.
Let $R$ be the $\tau$-theory of $\mathcal{M}$. Since $\mathcal{M}$ is $T$-ec, any superstructure of $\mathcal{M}$ which models $T$ is also a model of $R_{\forall\vee\exists}$ (by Fact \ref{fac:Tecmodels}).
Since $\psi$ is strongly $T_{\forall\vee\exists}$-consistent, $\psi+R_{\forall\vee\exists}$ is consistent.
By Lemma \ref{lem:abscoth}, $\psi$ holds in some $\mathcal{N}\sqsupseteq\mathcal{M}$ which models 
$R_{\forall\vee\exists}$. Since $\mathcal{M}$ is $T$-ec and $R_{\forall\vee\exists}\supseteq T_\forall$, we get that $\psi$ reflects to $\mathcal{M}$ (being a $\Pi_2$-sentence which holds in $\mathcal{N}$ which is a $\Sigma_1$-superstructure of $\mathcal{M}$).


\item
Assume a $\Pi_2$-sentence $\psi$ is in $\bool{KH}(T)$. Let $R$ be any complete extension of $T$ and
$\mathcal{M}$ be a model of $R$. By Lemma \ref{lem:abscoth} (since $T_{\forall\vee\exists}=\bool{KH}(T)_{\forall\vee\exists}$)
there is $\mathcal{N}$ which is a model of $\bool{KH}(T)$ and of $R_{\forall\vee\exists}$. In particular $\mathcal{N}$ models $\psi+R_{\forall\vee\exists}$.
Since $R$ is arbitrary, $\psi$ is strongly $T_{\forall\vee\exists}$-consistent.

Clearly if $T$ is complete, $T_{\forall\vee\exists}$ is $\Pi_1$-complete, and $\psi$ is strongly $T_{\forall\vee\exists}$-consistent if and only if $\psi+T_{\forall\vee\exists}$ is consistent. 
We conclude also in this case that a $\Pi_2$-sentence $\psi$ holds in some $T$-ec model if and only if it is strongly $T_{\forall\vee\exists}$-consistent. 

\item 
Note that $\bool{SCH}(T)$ is axiomatized by its $\Pi_2$-fragment and $\bool{SCH}(T)_{\forall\vee\exists}= T_{\forall\vee\exists}$.
Therefore we can apply Lemma \ref{lem:Pi2abscoh} below.
\end{enumerate}
\end{proof}

\begin{remark}
There can be $\tau$-theories $T$ whose Kaiser hull strictly contains its strong consistency hull.

Consider the $\bp{0,1,\cdot,+}$-theories $\bool{ACF}_0$ (of algebraically closed fields of characteristic $0$) and $\bool{Fields}_0$ (of fields of characteristic $0$). Note that $\bool{ACF}_0$ is complete while $\bool{Fields}_0$ is not. 
Furthermore $\bool{ACF}_0$ is the Kaiser hull of $\bool{Fields}_0$ 
(note that: $\bool{ACF}_0$ is axiomatized by its $\Pi_2$-fragment; any $\bool{Fields}_0$-ec model is an algebraically closed field; any model of $\bool{ACF}_0$ is $\bool{Fields}_0$-ec).


We get that $\exists x\, (x^2+1=0)$ is a $\Pi_2$-sentence (in fact existential) in the Kaiser hull of  $\bool{Fields}_0$ but not in its strong consistency hull, since it is not consistent with $R_{\forall\vee\exists}$, where $R$ is the 
$\bp{0,1,\cdot,+}$-theory of the rationals.
\end{remark}

\begin{lemma}\label{lem:Pi2abscoh}
Let $S,T$ be $\tau$-theories such that $S_{\forall\vee\exists}\subseteq T_{\forall\vee\exists}$ and $S$ is axiomatized by its $\Pi_2$-fragment.
Then for any $\Pi_2$-sentence $\psi$ consistent with $T$ there is a model of
$S+\psi$.
\end{lemma}

\begin{proof}
We prove a stronger conclusion which is the following:
\begin{quote}
\emph{Let $R$ be a complete theory extending $T+\psi$. Then there is a model of $R_{\forall\vee\exists}+S+\psi$.}
\end{quote}
Let $\bp{\mathcal{M}_n:\,n\in\omega}$ be a sequence of $\tau$-structures such that for all $n\in\omega$:
\begin{itemize}
\item
$\mathcal{M}_n$ is a $\tau$-substructure of $\mathcal{M}_{n+1}$;
\item
$\mathcal{M}_n$ models $R_{\forall\vee\exists}$;
\item
$\mathcal{M}_{2n}$ models $R$;
\item
$\mathcal{M}_{2n+1}$ models $S$.
\end{itemize}
Such a sequence can be defined letting $\mathcal{M}_0$ be a model of $R$, $\mathcal{M}_1$ be a model of $S$ which satisfies $R_{\forall\vee\exists}$ (which is possible in view of Lemma \ref{lem:abscoth})
and defining $\mathcal{M}_n$ as required for all other $n$ appealing to the fact that $S+R_{\forall\vee\exists}$ and $T+\psi+R_{\forall\vee\exists}$ are absolute cotheories with $R_{\forall\vee\exists}$ being the $\Pi_1$-complete fragment shared by both theories.
Then $\mathcal{M}=\bigcup_{n\in\omega}\mathcal{M}_n$ is a model of $R_{\forall\vee\exists}+S+\psi$ since it realizes all $\Pi_2$-sentences which hold in an infinite set of $\mathcal{M}_n$ (see for example \cite[Lemma 3.1.6]{TENZIE}) and satisfies exactly the same $\Pi_1$-sentences of each of the $\mathcal{M}_n$.
\end{proof}

Note that the above cannot be proved if $S,T$ are just cotheories: performing the above construction under this weaker assumption, we may not be able to define $\mathcal{M}_2$ as required if $\mathcal{M}_1$ does not realize exactly the same universal sentences of $\mathcal{M}_0$.

A useful consequence of the above results is the following:
\begin{proposition}\label{prop:sigma2notinSCH}
Assume $T$ is a $\tau$-theory and $\psi$ is a $\Pi_2$-sentence for $\tau$ such that $T+\psi$ is consistent.
Then for all $A\subseteq\bool{Form}_\tau\times 2$ we have that $\neg\psi\not\in \bool{SCH}(T+T_{\tau,A})$.
\end{proposition}
\begin{proof}
By assumption for any $A\subseteq\bool{Form}_\tau\times 2$
$(T+T_{\tau,A})_{\forall\vee\exists}$ is consistent. By Fact \ref{rem:SCHKH}\ref{rem:SCHKH-4}, 
there is a model of  $\bool{SCH}(T+T_{\tau,A})+\psi$, hence $\neg\psi\not\in\bool{SCH}(T+T_{\tau,A})$.
\end{proof}
%


\subsection{Absolute model companionship}

\begin{definition}\label{def:AMC+MC}
A $\tau$-theory $T$ is:
\begin{itemize}
\item \emph{model complete} if $\mathcal{M}\models T$ if and only if it is $T_\forall$-ec;
\item the \emph{model companion} of a $\tau$-theory $S$ if $T$ and $S$ are cotheories and $T$ is model complete;
\item the \emph{absolute model companion} (AMC) of a $\tau$-theory $S$ if $T$ and $S$ are absolute cotheories and $T$ is model complete.
\end{itemize}
\end{definition}
Our definition of model completeness and model companionship takes advantage of \cite[Prop. 3.5.15]{CHAKEI90}.

We will use repeatedly Robinson's test providing different equivalent characterizations of model completeness:

\begin{lemma}\cite[Lemma 3.2.7]{TENZIE}\label{lem:robtest}
(Robinson's test) Let $T$ be a $\tau$-theory. The following are equivalent:
\begin{enumerate}[(a)]
\item \label{lem:robtest-1} $T$ is model complete. 
\item \label{lem:robtest-0} Whenever $\mathcal{M}\sqsubseteq\mathcal{N}$ are models of $T$,
  $\mathcal{M}\prec\mathcal{N}$.
\item \label{lem:robtest-4} Each \emph{existential} $\tau$-formula $\phi(\vec{x})$ in free variables $\vec{x}$ 
is $T$-equivalent to a universal $\tau$-formula $\psi(\vec{x})$ in the same free variables. 
\item \label{lem:robtest-3} Each $\tau$-formula $\phi(\vec{x})$ in free variables $\vec{x}$ 
is $T$-equivalent to a universal $\tau$-formula $\psi(\vec{x})$ in the same free variables. 
\end{enumerate}
\end{lemma}
\begin{remark}\label{rmk:robtest}
\ref{lem:robtest-3} (or \ref{lem:robtest-4}) shows that being a model complete $\tau$-theory $T$ is expressible by
a $\Delta_0$-property in parameters $\tau,T$ in any model of $\ZFC$, hence it is absolute with respect 
to forcing.
They also show that quantifier elimination implies model completeness.
\ref{lem:robtest-4} also shows that model complete theories are axiomatized by their $\Pi_2$-fragment.
\end{remark}

The following characterization of absolute model companionship has brought our attention to this notion.

\begin{lemma}\label{fac:proofthm1-2}
Assume $T,T'$ are $\tau$-theories and $T'$ is model complete. TFAE:
\begin{enumerate}[(i)]
\item\label{fac:proofthm1-2-A}
$T'$ is  the absolute model companion of $T$.
\item\label{fac:proofthm1-2-B}
$T'$ is axiomatized by the strong consistency hull of $T$.
\end{enumerate}
\end{lemma}

\begin{proof}
\emph{}

\begin{description}
\item[\ref{fac:proofthm1-2-A} implies \ref{fac:proofthm1-2-B}]
First of all we note that any model complete theory $S$ is axiomatized by its strong consistency hull in view of 
Robinson's test \ref{lem:robtest-3} 
and Fact \ref{fac:Tecmodels}.

We also note that for absolute cotheories $T,T'$, their strong consistency hull overlap (in view of Lemma 
\ref{lem:abscoth}).

Putting everything together we obtain the desired implication.

%
%
%
%
%

\item[\ref{fac:proofthm1-2-B} implies \ref{fac:proofthm1-2-A}]
%
Note that for $\theta$ a boolean combination of universal $\tau$-sentences, we have that $\theta$ is in the strong consistency hull of some $\tau$-theory $S$ if and only if $\theta\in S_{\forall\vee\exists}$.
Combined with \ref{fac:proofthm1-2-B}, this gives that  $T'_{\forall\vee\exists}=T_{\forall\vee\exists}$.
\end{description}
\end{proof}

Finally the following Lemma motivates our terminology for AMC:

\begin{lemma}
Assume $T,T'$ are $\tau$-structures such that $T'$ is the AMC of $T$.
Then any $S$ extending $T$ has as AMC $T'+S_{\forall}$.
\end{lemma}
Note that this fails for the standard notion of model companionship:
$\bool{ACF}$ is the model companion of $\bool{Fields}$ in signature $\tau=\bp{0,1,\cdot,+}$, 
but if $S$ is the theory of the rationals in signature $\tau$, $S_\forall+\bool{ACF}$ is inconsistent, hence it cannot be the model companion of $S$.
\begin{proof}
Assume $S\supseteq T$ is consistent.
If $\mathcal{M}\models S$, $\mathcal{M}$ has a superstructure which models $T'+S_{\forall\vee\exists}$, since $T$ and $T'$ are absolute cotheories. 
This gives that $S'=T'+S_{\forall}$ is consistent. Since $T'$ is model complete, so is $S'$ by Robinson's test (cfr. Remark \ref{rmk:robtest} and Lemma \ref{lem:robtest}\ref{lem:robtest-4}).
Now observe that $S'$ and $S$ satisfy item \ref{lem:abscoth2} of Lemma
\ref{lem:abscoth} (since $S'\supseteq T'$ and $S\supseteq T$ with $T$ and $T'$ absolute cotheories), yielding easily that $S_{\forall\vee\exists}=S_{\forall\vee\exists}'$. Therefore $S'$ is the AMC of $S$.
\end{proof}

\begin{remark}\label{rmk:keyrmkcharkaihull}
Absolute model companionship is strictly stronger than model companionship:
if $T$ is model complete, $T$ is the model companion of $T_\forall$ and the absolute model companion of $T_{\forall\vee\exists}$;
the two
notions do not coincide whenever $T_{\forall}$ is strictly weaker than $T_{\forall\vee\exists}$.

If $T'$ is the model companion of $T$, $T_{\forall\vee\exists}'\supseteq T_{\forall\vee\exists}$: assume $\mathcal{M}\models T'$, then
there is a superstructure $\mathcal{N}$ of $\mathcal{M}$ which models $T$ (since 
$T'$ is the model companion of $T$).
Now $\mathcal{M}\prec _1\mathcal{N}$, since $\mathcal{M}$ is $T$-ec.
Hence $\mathcal{N}$ has the same $\Pi_1$-theory of $\mathcal{M}$.
The inclusion can be strict as shown by the counterexample given by $\bool{Fields}$ versus 
$\bool{ACF}$.
\end{remark}


Recall that $T$ is the model completion of $S$ if it is its model companion and admits quantifier elimination
(See \cite[Prop. 3.5.19]{CHAKEI90}).


Absolute model companionship does not imply model completion:

\begin{fact}
Let $\in_B$ and $S$ be the signature and theory appearing in Theorem \ref{Thm: mainthmforcibility}. Then $\bool{AMC}(S,B)$ does not admit quantifier elimination,
hence $S+T_{\in,B}$ does not have the amalgamation property  and has no model completion.
\end{fact}

\begin{proof}
We show that $S+T_{\in,B}$ does not have the amalgamation property. This suffices to prove the Fact by \cite[Prop. 3.5.19]{CHAKEI90}.
Given $(V,\in)$ model of $S+\MM^{++}$, let $G$ be $V$-generic for Namba forcing at $\aleph_2$ and $H$ be $V$-generic for 
$\Coll(\omega_1,\omega_2)$. Note that (by Thm. \ref{thm:genabshomega1}, since all predicate symbols of $\in_B$ are either universally Baire sets or $\Delta_0$-definable formulae or the non-stationary ideal on $\omega_1$)
\[
(V,\in_B^V)\sqsubseteq (V[G],\in_B^{V[G]}),  (V[H],\in_B^{V[H]}).
\]
If $W$ is an $\in_B$-amalgamation of $V[G]$ and $V[H]$ over $V$, and $W$ models $S+T_{\in,B}$, we get that in $W$ $\omega_2^V$ has cofinality $\omega_1^V$ and $\omega$ at the same time (both properties are expressible by $\Delta_0$-formulae in parameters $\omega,\omega_1^V,\omega_2^V,f,g$ stating that $f:\omega\to\omega_2^V$ is cofinal and $g:\omega_1\to\omega_2^V$ is cofinal), hence $\omega_1^V$ is countable in $W$. 
This is impossible since $(H_{\omega_2}^V,\in_B^V)\prec_1 (W,\in_A^W)$ and $\omega_1^V$ is uncountable in $(H_{\omega_2}^V,\in_B^V)$.
\end{proof}

\subsection{Preservation of the substructure relation and of $\Sigma_1$-elementarity by expansions via definable Skolem functions}

These technical results will be needed for the enhanced version of Levy absoluteness given by Lemma \ref{lem:levabsgen}. This is crucial for the proofs of
Thm. \ref{Thm:AMCsettheory+Repl}\ref{Thm:AMCsettheory+Repl-1} and in Sections \ref{sec:geninv} and \ref{sec:Homega2} to define the signature $\in_B$ of Thm. \ref{Thm: mainthmforcibility}.
Recall Notation \ref{not:keynotation0}.

\begin{definition}\label{def:Delta1S}
Let $S$ be a $\tau$-theory.
A $\tau$-formula $\phi(\vec{x})$ is $\Delta_1(S)$ if there are quantifier free $\tau$-formulae $\psi_\phi(\vec{x},\vec{y})$ and 
$\theta_\phi(\vec{x},\vec{z})$ such that $S$ proves
\[
\forall\vec{x}\,[\phi(\vec{x})\leftrightarrow \forall\vec{y}\,\psi_\phi(\vec{x},\vec{y})\leftrightarrow \exists\vec{z}\,\theta_\phi(\vec{x},\vec{z})].
\]
\end{definition}

\begin{fact}\label{fac:correctness-expansionDelta1}
Assume $T$ is a $\tau$-theory, $\mathcal{M},\mathcal{N}$ are $\tau$-models of $T$ and $A\subseteq\bool{Form}_\tau\times 2$ is such that:
\begin{itemize}
\item
for each $\ap{\phi,i}$ in $A$, $\phi$ is a provably $\Delta_1(T)$-formula;
\item
for each $(\phi,1)\in A$ 
\[
T\models \forall\vec{x}\exists! y\phi(\vec{x},y).
\]
\end{itemize}
The following holds for $\mathcal{M}_A,\mathcal{N}_A$  the unique expansions of $\mathcal{M},\mathcal{N}$ to $\tau_A$-models of $T_{\tau,A}$ both interpreting as some $a\in\mathcal{M}$ the constant $c_\tau$:
\begin{enumerate}[(A)]
\item If $\mathcal{M}\sqsubseteq\mathcal{N}$ (for $\tau$), then $\mathcal{M}_A\sqsubseteq\mathcal{N}_A$  (for $\tau_A$).
\item if $\mathcal{M}\equiv_1\mathcal{N}$  (for $\tau$), then $\mathcal{M}_A\equiv_1\mathcal{N}_A$
 (for $\tau_A$).
\item if $\mathcal{M}\prec_1\mathcal{N}$  (for $\tau$), then $\mathcal{M}_A\prec_1\mathcal{N}_A$
 (for $\tau_A$).
\end{enumerate}
\end{fact}

\begin{proof}
The proof is an elementary application of basic observations.
First one shows that whenever 
 $\phi(\vec{x},y)$ is a provably $\Delta_1(T)$
$\tau$-formula and $A_0=\bp{\ap{\phi,0}}$,
then:
\begin{enumerate}[(i)]
\item \label{fac:correctness-expansionDelta1-1}
for every quantifier free formula $\gamma(\vec{v})$ of $\tau_{A_0}$
there are a
$\psi_\gamma(\vec{v},\vec{u}),\theta_\gamma(\vec{v},\vec{z})$ quantifier free $\tau$-formulae such that $T+T_{\tau,A_0}$ models
\[
\forall\vec{v}\,[\gamma(\vec{v})\leftrightarrow \forall\vec{u}\,\psi_\gamma(\vec{v},\vec{u})\leftrightarrow \exists\vec{z}\,\theta_\gamma(\vec{v},\vec{z})].
\]
\end{enumerate}
Then one shows that if furthermore
\[
T\models \forall\vec{x}\exists! y\phi(\vec{x},y),
\]
and let $A_1=\bp{\ap{\phi,1}}$.
Then:
\begin{enumerate}[(i)]
\setcounter{enumi}{1}
\item \ref{fac:correctness-expansionDelta1-1} holds replacing $A_0$ with $A_1$ all over.
\end{enumerate}
Using the above inductively one proves that every quantifier free $\tau_A$-formula is provably equivalent over $T+T_{\tau,A}$ both to a $\Sigma_1$-formula for $\tau$ and to a $\Pi_1$-formula for $\tau$. 

The three conclusions for  $\mathcal{M}_A,\mathcal{N}_A$ follow then easily.
\end{proof}

\subsection{Summing up}

\begin{itemize}
\item
We see model completeness, model companionship, AMC as tameness properties of elementary classes 
$\mathcal{E}$ defined by a theory $T$ rather than of the theory $T$ itself: 
these model-theoretic notions outline certain regularity patterns for the substructure relation on
models of $\mathcal{E}$, patterns which may be unfolded only when passing to a signature distinct 
from the one in which
$\mathcal{E}$ is first axiomatized (much the same way as it occurs for Birkhoff's 
characterization of algebraic varieties in terms of universal theories).

\item We will see in the next sections that set theory together with large cardinal axioms  has (until now unexpected) tameness properties when formalized in certain natural signatures (already implicitly considered in most of the prominent set theoretic results of the last decades). These tameness properties 
couple perfectly with well 
known (or at least published) generic absoluteness results.
The notion of AMC-spectrum gives an additional model theoretic criterium for selecting these ``natural'' signatures out of the 
continuum many
signatures which produce definable extensions of $\ZFC$. 

\end{itemize}

\section{The AMC spectrum of set theory}\label{sec-AMC-set-theory}

In this section we prove the basic properties of  the AMC spectrum of set theory.
Specifically we prove Thm.~\ref{Thm:AMCsettheory+Repl}, Thm. \ref{mainthm:CH*}, Thm. \ref{mainthm:2omegageqomega2*} (items \ref{mainthm:CH*-2} of Thm. \ref{mainthm:CH*} and \ref{mainthm:2omegageqomega2*-4} of Thm. \ref{mainthm:2omegageqomega2*} are proved conditionally to the proof of Thm. \ref{Thm: mainthmforcibility}).

In \ref{subsec:basicnotterm}  we define precisely the signatures in which we need to formalize set theory in order to prove all our main results.

\ref{subsec:Levabs} formulates in precise mathematical terms the role Levy absoluteness plays in the proofs of the above theorems.

In \ref{subsec:replAMC} we show that for any $A\subseteq\bool{Form}_\in\times 2$ with $\in_A\supseteq\in_{\Delta_0}$  in the model companionship spectrum of set theory and $R\supseteq\ZFC$, 
$\bool{MC}(R,A)$ is a model of a large chun of the replacement schema, choice, and is closed under Goedel operations.
Typically these axioms characterize the structures of type $H_\kappa$ for $\kappa$ a regular cardinal.

In \ref{sec:Hkappa+} we show how to produce sets $A$ such that $\in_A\supseteq\in_{\Delta_0}$ and the AMC of set theory with respect to $\in_A$ exists and is the theory of $H_{\kappa^+}$ for some infinite cardinal $\kappa$.

In \ref{subsec:AMCnegCH} we show why it is  artificial to put $\CH$ and $2^{\aleph_0}>\aleph_2$ in 
some AMC of set theory. 

\ref{subsec:replAMC}, \ref{sec:Hkappa+}, \ref{subsec:AMCnegCH} can be read independently of one another.

All proofs in this section are elementary (e.g. knowledge of \cite[Chapters $I$, $III$, $IV$]{KUNEN} suffices to follow the arguments); however ---especially in \ref{sec:Hkappa+}--- the notation is heavy. We haven't been able to simplify it.

\subsection{Basic notation and terminology} \label{subsec:basicnotterm}

We introduce the basic signatures and fragments of set theory we will always include in any signature of interest to us.

\begin{notation}\label{not:basicsettheorynot}
We let $\in_{\Delta_0}$ be $\in_D$ for $D\subseteq\bool{Form}_\in\times 2$ extending the set
$\Delta_0\times\bp{0}$ with the pairs $(\phi,1)$ as $\phi$ ranges over the following $\Delta_0$-formulae:
\begin{itemize}
\item
The $\Delta_0$-formulae $\phi_\omega(x)$, $\phi_\emptyset(x)$ defining $\emptyset$ and $\omega$ in any model of $\ZF^-$, where the latter includes all axioms of $\ZF$ with the exception of power-set axiom
(also we denote by $\omega$ and $\emptyset$ the constants $f_{\phi_\emptyset}$, $f_{\phi_\omega}$).
\item
The $\Delta_0$-formulae $\phi_i(\vec{x},y)$
as $G_i$ ranges over the Goedel operations 
$G_1,\dots,G_{10}$ as defined in \cite[Def. 13.6]{JECHST} 
and $\phi_i(\vec{x},y)$ is the $\Delta_0$-formula defining the graph of $G_i$ in any $\in$-model of
\footnote{In models of $\ZF^-$ the Goedel operations $G_1,\dots,G_{10}$ as listed and defined in \cite[Def. 13.6]{JECHST} and their compositions have as graph the extension of a $\Delta_0$-formula  (by \cite[Lemma 13.7]{JECHST}).} 
$\ZF^-$.
\end{itemize}
We let $T_{\Delta_0}$ be given by the axioms:
\begin{equation}\label{eqn:ZFDELTA0-1}
\forall \vec{x} \,(R_{\forall z\in y\phi}(y,z,\vec{x})\leftrightarrow \forall z(z\in y\rightarrow R_\phi(y,z,\vec{x})),
\end{equation}
\begin{equation}\label{eqn:ZFDELTA0-2}
\forall \vec{x} \,[R_{\phi\wedge\psi}(\vec{x})\leftrightarrow (R_{\phi}(\vec{x})\wedge R_{\psi}(\vec{x}))],
\end{equation}
\begin{equation}\label{eqn:ZFDELTA0-3}
\forall \vec{x} \,[R_{\neg\phi}(\vec{x})\leftrightarrow \neg R_{\phi}(\vec{x})]
\end{equation}
\begin{equation}\label{eqn:ZFDELTA0-4}
\forall x\, (x\not\in \emptyset)
\end{equation}
\begin{equation}\label{eqn:ZFDELTA0-4bis}
\omega\emph{ is a non-empty ordinal all whose elements are successor ordinals or $\emptyset$}.
\end{equation}
\begin{equation}\label{eqn:ZFDELTA0-5}
\forall \vec{x}\,\exists! y \,R_{\phi_i}(\vec{x},y))
\end{equation}
\begin{equation}\label{eqn:ZFDELTA0-6}
\forall \vec{x}\,\forall y \,[y=G_i(\vec{x})\leftrightarrow R_{\phi_i}(\vec{x},y)]
\end{equation}
for the Goedel operations 
$G_1,\dots,G_{10}$.

\smallskip

We axiomatize suitable fragments of the $\in$-theory $\ZFC+T_{\Delta_0}$ as follows:
\begin{itemize}
\item
$\bool{Z}^-_{\Delta_0}$ stands for the $\in_{\Delta_0}$-theory given by: 
\begin{enumerate}[(a)]
\item
the Extensionality Axiom
\[
\forall x,y,z\,\qp{(z\in x\leftrightarrow z\in y)\rightarrow x=y},
\]
\item
the Foundation Axiom
\[
\forall x\qp{x=\emptyset\vee\exists y\in x\,\forall z\in x\,(z\not\in y)},
\]
\item $T_{\Delta_0}$.
\end{enumerate}


\item
$\bool{Z}_{\Delta_0}$ enriches $\bool{Z}^-_{\Delta_0}$ adding the power-set axiom
\[
\forall x\exists y\,\qp{\forall z\, (z\subseteq x\leftrightarrow z\in y}.
\]
\item
$\bool{ZC}^-_{\Delta_0}$ enriches $\bool{Z}^-_{\Delta_0}$ adding the axiom of choice $\bool{AC}$
\[
\forall x\exists f\,\qp{(f\text{ is a bijection})\wedge\dom(f)=x\wedge(\ran(f)\text{ is an ordinal})}.
\]
\item
$\bool{ZF}^-_{\Delta_0}$ enriches $\bool{Z}^-_{\Delta_0}$ adding the Strong Replacement Axiom \footnote{Which is a strong form of the Collection Principle and is defined in (\ref{eqn:strrep}) below.} for all $\in_{\Delta_0}$-formulae.
\item
$\bool{ZFC}^-_{\Delta_0}$, $\bool{ZF}_{\Delta_0}$, $\bool{ZFC}_{\Delta_0}$ are defined as expected.
\end{itemize}
\end{notation}

\begin{remark}
We took the pain of giving an explicit axiomatization of $\bool{Z}^-_{\Delta_0}$ using Extensionality, Foundation, and axioms
\ref{eqn:ZFDELTA0-1},\dots,\ref{eqn:ZFDELTA0-6} because this axiomatizion is given by $\Pi_2$-sentences
of $\in_{\Delta_0}$, hence it is preserved by $\Sigma_1$-substructures. Note that $\bool{AC}$ is a $\Pi_2$-axiom of $\in_{\Delta_0}$ while the power-set axiom and the (strong) replacement schema for a quantifier free $\in_{\Delta_0}$-formula are both $\Pi_3$.

A simple inductive argument shows that $\ZF^-+T_{\in,D}$ (where $D$ is the subset of $\bool{Form}_\in\times 2$ used in Not. \ref{not:basicsettheorynot} to define $\in_{\Delta_0}$) is logically equivalent to $\ZF^-$ enriched with axioms 
\ref{eqn:ZFDELTA0-1},\dots,\ref{eqn:ZFDELTA0-6} (with $\emptyset$ taking the place of $c_\in$ and $\omega$ being the constant of $\in_{\Delta_0}$ associated to the $\Delta_0$-formula defining it).
We skip the details.
\end{remark}

We now introduce the terminology to handle set theory formalized in signatures richer than $\in_{\Delta_0}$.

\begin{notation}
Let $\tau\supseteq \in_{\Delta_0}$.
For a $\tau$-formula $\phi(\vec{x},\vec{y},\vec{z})$:

\begin{itemize} 
\item
The \emph{Replacement Axiom} for $\phi$ ($\bool{Rep}(\phi)$) states:
\[
\forall\vec{z}\qp{(\forall x\,\exists! y\,\phi(x,y,\vec{z}))\rightarrow
\forall X\exists F\,(F\text{ is a function}\wedge\dom(F)=X\wedge\forall x\in X\,\phi(x,F(x),\vec{z}))};
\]
$\bool{Rep}_\tau$ holds if  $\bool{Rep}(\phi)$ holds for all $\tau$-formulae $\phi$.
\item
The \emph{Strong Replacement Axiom} for\footnote{This is a strong form of the Collection Principle.  Note that Replacement, the Collection Principle and Strong Replacement are equivalent over the other axioms of $\ZFC$ (when including powerset and choice and excluding replacement), but are not equivalent if one drops either one among choice and powerset and retains the other axioms of $\ZFC$ (but not replacement).} $\phi$ ($\bool{StrRep}(\phi)$) states:
\begin{equation}\label{eqn:strrep}
\forall\vec{z}\qp{(\forall x\,\exists y\,\phi(x,y,\vec{z}))\rightarrow
\forall X\exists F\,(F\text{ is a function}\wedge\dom(F)=X\wedge\forall x\in X\,\phi(x,F(x),\vec{z}))};
\end{equation}
$\bool{StrRep}_\tau$ holds if  $\bool{StrRep}(\phi)$ holds for all $\tau$-formulae $\phi$.
\item
$\ZF^-_\tau$ is
$\bool{Z}^-_{\Delta_0}+\bool{StrRep}_\tau$.
\item
Accordingly we define $\ZFC_{\tau}$, $\ZFC^-_\tau$, $\ZF_\tau$, $\ZFC_\tau$,\dots
\item 
We write $\ZFC_{\Delta_0}$ rather than $\ZFC_\tau$ when $\tau=\in_{\Delta_0}$, etc.
\item
If $A\subseteq\bool{Form}_\in\times 2$ is such that $\in_{\Delta_0}\subseteq \in_A$, we write $\ZFC^-_A$ rather than $\ZFC^-+T_{\in,A}$,\dots
\end{itemize}
\end{notation}

Clearly (the suitable fragment of) $\ZFC+T_{\in,A}$ is logically equivalent to (the suitable fragment of) $\ZFC_{A}$.

\subsubsection{Further notational conventions}
Let us introduce notation we will use to handle the substructure relation over expanded signatures.
 The following supplements Notations \ref{not:keynotation0}, \ref{not:keynotation}.
\begin{notation}\label{not:modthnot2}
Given some signature $\tau\supseteq\in_{\Delta_0}\cup\bp{\kappa}$ and a
$\tau$-structure $(M,\tau^M)$ and some $B\subseteq\bool{Form}_{\tau}\times 2$,
$(M,\tau_B^M)$ is the unique extension of 
$(M,\tau)$ defined in accordance with 
Notations \ref{not:keynotation0}, \ref{not:keynotation} which satisfies $T_{\tau,B}$.
In particular $(M,\tau_B^M)$ is a shorthand for 
$(M,S^M:S\in\tau_B)$.
If $(N,\tau^N)$ is a substructure of $(M,\tau^M)$ we also write
$(N,\tau_B^M)$ as a shorthand for 
$(N,S^M\restriction N:S\in\tau_B)$.
\end{notation}

\begin{remark}
Note that even if $(N,\tau^N)$ is a substructure of $(M,\tau^M)$, $(N,\tau_B^M)$ and $(N,\tau_B^N)$
could be differents structures and $(N,\tau_B^N)$ may not be a substructure of $(M,\tau_B^M)$:
\begin{itemize}
\item
 $(N,\tau_B^M)$  is obtained by restricting to $N$ the intepretation of the new symbols of $\tau_B$ in $M$
 according to how $M$ realizes $T_{\tau,B}$;
 \item
$(N,\tau_B^N)$ is obtained by interpreting the new symbols of $\tau_B$ in $N$ according to $T_{\tau,B}$
as realized in $N$. 
\end{itemize}
We are spending a great deal of attention to isolate those set theoretic concepts which grant that 
 $(N,\tau_B^M)$ and $(N,\tau_B^N)$ are the same (at least when $(N,\tau^N)\prec_1(M,\tau^M)$).
 \end{remark}

\begin{notation}\label{not:inDelta1}
Let $(M,E)$ be an $\in$-structure. For $a\in M$, 
\[
\text{Ext}^M_\in(a)=\bp{b\in M:\,M\models b\in a}.
\]
$N\subseteq M$ is a \emph{transitive subset of $M$} if $\text{Ext}^M_\in(a)=\text{Ext}^N_\in(a)$ for all $a\in N$. 
\end{notation}

There are many basic set theoretic properties which are established in standard textbooks (such as \cite{JECHST,KUNEN})
only for transitive models of fragments of $\ZFC$ and can be established for arbitrary models of these fragments:

\begin{fact}\label{fac:keyfatr0}
Let $(M,E)$ be an $\in$-structure with $N\subseteq M$ a transitive subset.
Then
for all $\Delta_0$-formula $\phi$ and $\vec{a}\in N^{<\omega}$, 
$N\models\phi(\vec{a})$ if and only if $M\models\phi(\vec{a})$.
\end{fact}
\begin{proof}
By the results of \cite[Section IV.3]{KUNEN} (which are there established under the further assumptions that $M,N$ are transitive, but can be proved just assuming the weaker assumption that $(M,E)$, $(N,E)$ are models of $\ZFC^-$ with $N$ a transitive subset of $M$).
We leave the details to the reader.
\end{proof}

\begin{notation}
Let $\in_{\Delta_1}$ be $\in_D$ for $D\subseteq\bool{Form}_\in\times 2$ given by the $\phi_i$ such that:
\begin{itemize}
\item $\phi$ is $\Delta_1(\ZFC^-_{\Delta_0})$
\item if $(\phi,1)\in E$, then $\ZFC^-$ proves $\forall\vec{x}\,\exists!y\phi(\vec{x},y)$.
\end{itemize}
$T_{\Delta_1}$ is $T_{\Delta_0}$ enriched with $\bool{Ax}^i_\phi$ for $(\phi,i)\in D$;
$\ZFC^-_{\Delta_1}$ is $\ZFC^-_{\Delta_0}+T_{\Delta_1}$; accordingly we define $\ZFC_{\Delta_1}$, etc.
\end{notation}
Note that $\in_{\Delta_0}$ and $\ZFC_{\Delta_0}$ are recursive sets while $\in_{\Delta_1}$ and $\ZFC_{\Delta_1}$ are just semi-recursive.

\begin{fact}\label{fac:keyfatr}
Assume $\tau_0\supseteq\in_{\Delta_0}$ and $\mathcal{M}_0,\mathcal{N}_0$ are $\tau_0$-structures
which model $\ZFC^-_{\Delta_0}$. Let $\tau_1$ be $\tau_0\cup\in_{\Delta_1}$ and
$\mathcal{M}_1,\mathcal{N}_1$ be the unique extensions to $\tau_1$-structures which model
$\ZFC^-_{\Delta_1}$. Then
\[
\mathcal{M}_0\mathrel{R}\mathcal{N}_0 \text{ if and only if } \mathcal{M}_1\mathrel{R}\mathcal{N}_1
\] 
holds for any binary relation $R\in\bp{\equiv_1,\sqsubseteq,\prec_1}$.
\end{fact}
\begin{proof}
Use Fact \ref{fac:correctness-expansionDelta1}.
\end{proof}

In view of the above facts if $(N,E)$ models $\ZFC^-$ and $M$ is a transitive subset of $N$ such that 
$(M,E)$ also models $\ZFC^-$ the unique expansions of $M,N$ to $\in_{\Delta_1}$-models of $\ZFC_{\Delta_1}$ are still substructures also according to $\in_{\Delta_1}$.

\subsection{Levy absoluteness} \label{subsec:Levabs}

We need a generalization of Levy's absoluteness in most proofs of the remainder of this paper.
We state and prove the Lemma under the assumption that the model of $\ZFC$ we work in is transitive; but this assumption
is unnecessary. Here and in other places of this paper we just need that the models in question satisfy $\ZFC^-$ or slightly more. 

Recall that $(H_\lambda,\in)$ is a model of $\ZFC^-$ for all regular cardinal $\lambda$.

\begin{Lemma}\label{lem:levabsgen}
Let $(V,\in_{\Delta_0})$ be a model of $\ZFC_{\Delta_i}$ and $\lambda>\kappa$ be infinite cardinals for $V$ with
$\lambda$ regular.
Then for both $i=0,1$  the structure
\[
(H_{\lambda},\in_{\Delta_i}^{H_{\lambda}},\, A: A\subseteq\pow{\kappa}^k,\, k\in\mathbb{N})
\]
is $\Sigma_1$-elementary in
\[
(V,\in_{\Delta_i}^V,
\, A: A\subseteq\pow{\kappa}^k,\,k\in\mathbb{N}).
\]
\end{Lemma}

Its proof is a variant of the classical result of Levy (which is the above theorem stated just for the signature 
$\in_{\Delta_0}$); it is a slight expansion of
\cite[Lemma 5.3]{VIAVENMODCOMP}; we include it here since it is not literally the same:
\begin{proof}
Let $\tau$ be either one of the signature $\in_{\Delta_0},\in_{\Delta_1}$ and
$\phi(\vec{x},y)$ be a quantifier free formula for $\tau\cup\bp{A_1,\dots,A_k}$ with each of the $A_i$ being a subset of some\footnote{Note that 
$\exists x\in y A(y)$ is not a quantifier free formula, and is actually equivalent to the 
$\Sigma_1$-formula
$\exists x(x\in y)\wedge A(y)$.}  $\pow{\kappa}^{n_i}$,
and 
$\vec{a}\in H_{\lambda}$ be such that
\[
(V,\, \tau^{V},\, A_1,\dots,A_k)\models\exists y\phi(\vec{a},y).
\]
Let $\alpha>\kappa$ be large enough  so that for some $b\in V_\alpha$
\[
(V,\, \tau^{V},\, A_1,\dots,A_k)
\models\phi(\vec{a},b).
\]
Then
\[
(V_\alpha,\, \tau^{V_\alpha},\, A_1,\dots,A_k)
\models\phi(\vec{a},b)
\]
(since $(V_\alpha,\tau^{V_\alpha},A_1,\dots,A_k)\sqsubseteq (V,\, \tau^{V},\, A_1,\dots,A_k)$
by  Fact \ref{fac:keyfatr}).
By the downward Lowenheim-Skolem theorem, we  can find
$X\subseteq V_\alpha$ which is the domain of a $\tau\cup\bp{A_1,\dots,A_k}$-elementary substructure of
\[
(V_\alpha,\,\tau^{V_\alpha},\, A_1,\dots,A_k)
\]
such that $X$ is a set of size $\kappa$ containing $\kappa$ and such that
$A_1,\dots,A_k,\kappa,b,\vec{a}\in X$. 
Since $|X|=\kappa\subseteq X$, a standard argument shows that 
$H_{\lambda}\cap X$ is a transitive set, and that $\kappa^+$ is the least ordinal in
$X$ which is not contained in $X$.
Let $M$ be the transitive collapse of $X$ via the Mostowski collapsing map $\pi_X$.

We have that
the first ordinal moved by $\pi_X$ is $\kappa^+$ and
$\pi_X$ is the identity on $H_{\kappa^+}\cap X$. Therefore $\pi_X(a)=a$
for all 
$a \in H_{\kappa^+}\cap X$.
Moreover for $A\subseteq \pow{\kappa}^n$ in $X$
\begin{equation}\label{eqn:piXidonpowkappa}
\pi_X(A)=A\cap M.
\end{equation}
We prove equation (\ref{eqn:piXidonpowkappa}):
\begin{proof}
Since 
$X\cap V_{\kappa+1}\subseteq X\cap H_{\kappa^+}$,
$\pi_X$ is 
the identity on $X\cap H_{\kappa^+}$, and
$A\subseteq \pow{\kappa}\subseteq V_{\kappa+1}$,
we get that
\[
\pi_X(A)=\pi_X[A\cap X]=\pi_X[A\cap X\cap V_{\kappa+1}]=A\cap M\cap V_{\kappa+1}=A\cap M.
\]
\end{proof}
It suffices now to show that
\begin{equation}\label{eqn:keyeqlevabs}
(M,\tau^M,\pi_X(A_1),\dots,\pi_X(A_k))\sqsubseteq (H_{\lambda},\tau^{H_\lambda},A_1,\dots,A_k).
\end{equation}
Assume \ref{eqn:keyeqlevabs} holds; since $\pi_X$ is an isomorphism and $\pi_X(A_j)=\pi_X[A_j\cap X]$, we
get that 
\[
(M,\tau^M,\pi_X(A_1),\dots,\pi_X(A_k))\models\phi(\pi_X(b),\vec{a})
\]
since 
\[
(X,\tau^V,A_1\cap X,\dots,A_k\cap X)\models\phi(b,\vec{a}).
\]
By (\ref{eqn:keyeqlevabs}) we get that 
\[
(H_{\lambda},\tau^{H_\lambda},A_1,\dots,A_k)\models\phi(\pi_X(b),\vec{a})
\]
and we are done.

We prove (\ref{eqn:keyeqlevabs}):
\begin{proof}
since $(M,\in)$ is a transitive model of $\ZFC^-$ with $M\subseteq H_\lambda$, any atomic $\tau$-formula
holds true in $(M,\tau^M)$ if and only if it holds in $(H_{\lambda},\tau^{H_\lambda})$ (again by Fact \ref{fac:keyfatr}).
It remains to argue that the same occurs for the formulae of type $A_j(x)$, i.e. that
$A_j\cap M=\pi_X(A_j)$ for all $j=1,\dots,n$; which is the case
by (\ref{eqn:piXidonpowkappa}).
\end{proof}
\end{proof}

\subsection{What holds in a $T$-ec structure for $T\supseteq\ZFC_{\Delta_0}$}\label{subsec:replAMC}

%
%

\begin{lemma}\label{lem:repl-closure0}
Assume $\mathcal{N}$ is a $\tau$-structure for some $\tau\supseteq\in_{\Delta_i}$ (for some $i=0,1$)
such that $\mathcal{N}\models \ZF^-_{\Delta_i}$. Assume further that
$\mathcal{M}\prec_1\mathcal{N}$.
Then:
\begin{enumerate}
\item $\mathcal{M}$ models $\bool{Z}^-_{\Delta_0}$.
\item
$\mathcal{M}\models\bool{StrRep}(\phi)$ for any existential $\in_{\Delta_i}$-formula $\phi(x,y,\vec{z})$
such that for all $X\in\mathcal{M}$ and parameters $\vec{a}\in\mathcal{M}^{<\omega}$
\[
\mathcal{N}\models\forall x\in X\exists y \phi(x,y,\vec{a}).
\]
\item
Furthermore if 
$\mathcal{N}\models \bool{AC}$, so does $\mathcal{M}$.
\end{enumerate}
\end{lemma}

\begin{proof}
\emph{}

\begin{enumerate}
\item 
All axioms of $\bool{Z}^-_{\Delta_0}$ hold
in $\mathcal{N}$ and are reflected to $\mathcal{M}$,
since $\mathcal{M}\prec_1\mathcal{N}$, and those axioms are at most $\Pi_2$-sentences of 
$\in_{\Delta_0}$. 

\item 
In view of Fact \ref{fac:keyfatr} it suffices to prove the Lemma just for a quantifier free $\in_{\Delta_0}$-formula $\phi(x,y,\vec{z})$.
By $\bool{StrRep}(\phi)$ applied in $\mathcal{N}$ and the assumptions, we get that for any $X\in\mathcal{M}$ and parameters 
$\vec{a}\in\mathcal{M}^{<\omega}$
\[
\mathcal{N}\models\exists F\,[F\text{ is a function}\wedge\dom(F)=X\wedge\forall y\in\ran(F)\phi(x,y,\vec{a})].
\]
Now observe that 
\[
F\text{ is a function}\wedge\dom(F)=X\wedge\forall y\in\ran(F)\phi(x,y,\vec{a})
\] 
is  a $\Delta_0$-formula in parameters $F,X,\vec{a}$, hence logically equivalent to an atomic
$\in_{\Delta_0}$-formula. By $\Sigma_1$-elementarity
\[
\mathcal{M}\models\exists F\,[F\text{ is a function}\wedge\dom(F)=X\wedge\forall y\in\ran(F)\phi(x,y,\vec{a})]
\]
and we are done.

\item It is immediately checked that if $\mathcal{N}\models \bool{AC}$, $\mathcal{M}$
satisfies also the axiom of choice, as this is a $\Pi_2$-sentence for $\in_{\Delta_0}$
which holds in $\mathcal{N}$ and thus reflects to $\mathcal{M}$.
\end{enumerate}
%
\end{proof}

This proves Theorem \ref{Thm:AMCsettheory+Repl}\ref{Thm:AMCsettheory+Repl-1}.

\subsection{The theory of $H_{\kappa^+}$ as the model companion of set theory}\label{sec:Hkappa+}

In this section we prove Thm. \ref{Thm:AMCsettheory+Repl}\ref{Thm:AMCsettheory+Repl-2}.
To motivate the result we first show the following:
\begin{lemma}
Assume $\tau\supseteq\in_{\Delta_i}\cup\bp{\kappa}$ (for some $i=0,1$) is such that
\(
(H_{\kappa^+}^\mathcal{N},\tau^{\mathcal{N}})\prec_1(\mathcal{N},\tau^{\mathcal{N}})
\)
whenever $\mathcal{N}$ models $S\supseteq\ZFC_\tau+\kappa$\emph{ is a cardinal}.
Then every $S$-ec structure satisfies 
\[
\forall x\exists f\,(f:\kappa\to x\text{ is a surjection}).
\]
\end{lemma}

In particular for any $\tau,S$ as in the Lemma, the $S$-ec models satisfy $\bool{ZC}_{\Delta_0}$, Strong Replacement for most existential $\tau$-formulae, and the $\Pi_2$-sentence stating that all sets have size at most $\kappa$, e.g. they provide a reasonable class of models describing a theory of $H_{\kappa^+}$.

\begin{proof}
Let $\mathcal{M}$ be $S$-ec and $\mathcal{N}\sqsupseteq\mathcal{M}$ be a model of 
$S$.
Since 
\(
(H_{\kappa^+}^\mathcal{N},\tau^{\mathcal{N}})\prec_1(\mathcal{N},\tau^{\mathcal{N}}),
\)
the two structures share the same $\Pi_1$-theory (which is clearly $\Pi_1$-complete); hence by Lemma \ref{lem:abscoth}, there is $\mathcal{P}\sqsupseteq\mathcal{N}$ satisfying the theory of 
$(H_{\kappa^+}^\mathcal{N},\tau^{\mathcal{N}})$, and in particular the universal theory of $S$ and the $\Pi_2$-sentence 
\(
\forall x\exists f\,(f:\kappa\to x\emph{ is a surjection}).
\)
Since $\mathcal{M}$ is $S$-ec, the latter sentence reflects to $\mathcal{M}$ and we are done.
\end{proof} 

We now give existence results stating that there are many signatures $\tau\supseteq\in_{\Delta_0}\cup\bp{\kappa}$ so that $\ZFC_\tau+\kappa$\emph{ is a cardinal} admits as AMC the theory 
\[
\ZFC^-_\tau+(\kappa\text{ is a cardinal})+\forall x\exists f\,(f:\kappa\to x\text{ is a surjection}).
\]

\subsubsection{By-interpretability of the first order theory of $H_{\kappa^+}$
with the first order theory of $\pow{\kappa}$}
\label{subsec:secordequiv}

Let's compare the first order theory of the structure
\[
(\pow{\kappa},\in_{\Delta_0}^V)
\]
with 
that of the $\in$-theory of $H_{\kappa^+}$ in models $(V,\in)$ of $\ZFC$. 
We show that they
are $\ZFC_{\Delta_0}$-provably by-interpretable with a by-interpetation translating $H_{\kappa^+}$ in a $\Pi_1$-definable subset of  $\pow{\kappa^2}$ (in signature $\in_{\Delta_0}$) and the $\in$-relation into a
$\Sigma_1$-relation over this set (in signature $\in_{\Delta_0}$). 
This result is the key to the proof of 
Thm.  \ref{Thm:AMCsettheory+Repl}\ref{Thm:AMCsettheory+Repl-1} and is just outlining the model theoretic consequences of the well-known fact that sets can be coded by well-founded extensional graphs.


\begin{definition}\label{def:codkappa}
Given $a\in H_{\kappa^+}$, $R\in \pow{\kappa^2}$ codes $a$, if
$R$ codes a well-founded extensional relation on 
some $\alpha\leq\kappa$ with top element $0$
so that the transitive collapse mapping of $(\alpha,R)$ maps $0$ to $a$.

\begin{itemize}
\item
$\WFE_\kappa$ is the set of $R\in \pow{\kappa^2}$ which 
are a well founded extensional relation with domain 
$\alpha\leq\kappa$ and top element $0$.
\item
 $\Cod_\kappa:\WFE_\kappa\to H_{\kappa^+}$ is the map assigning $a$ to $R$ if and only if 
 $R$ codes $a$.
\end{itemize}
\end{definition}

The following theorem shows that the structure $(H_{\kappa^+},\in)$ is interpreted by means of ``imaginaries'' in the structure
$(\pow{\kappa},\in_{\Delta_0}^V)$ by means of:
\begin{itemize}
    \item a universal $\in_{\Delta_0}\cup\bp{\kappa}$-formula (with quantifiers
    ranging over subsets of $\kappa^{<\omega}$)
    defining a set $\WFE_\kappa\subseteq\pow{\kappa^2}$.
    \item an equivalence relation $\cong_\kappa$ on $\WFE_\kappa$
    defined by an existential $\in_{\Delta_0}\cup\bp{\kappa}$-formula (with quantifiers
    ranging over subsets of $\kappa^{<\omega}$)
    \item A binary relation $E_\kappa$ on $\WFE_\kappa$
    invariant under $\cong_\kappa$ representing the $\in$-relation as the extension of 
an existential $\in_{\Delta_0}\cup\bp{\kappa}$-formula (with quantifiers
    ranging over subsets of $\kappa^{<\omega}$)\footnote{See \cite[Section 25]{JECHST} for proofs of the case $\kappa=\omega$; in particular the statement and proof of Lemma 25.25 and the proof of \cite[Thm. 13.28]{JECHST} contain all ideas on which one can elaborate to draw the conclusions of Thm.~\ref{thm:keypropCod}. Note that the map $x\mapsto x^{<\omega}$ has a $\Delta_0$-graph in models of $\ZF^-$- Therefore quantification over $\kappa$ or over $\kappa^{<\omega}$ are the same modulo an $\in_{\Delta_0}$-term. In the sequel we might be sloppy and identify at our convenience $\kappa$ with $\kappa^{<\omega}$.}.
\end{itemize}

\begin{notation}\label{not:keynotation1}
Recall Notation \ref{not:basicsettheorynot}.
\[
\phi_{\mathrm{Card}}(x)
\]
is the $\Pi_1$-formula for $\in_{\Delta_0}$  
\[
(x\emph{ is a cardinal})\wedge \omega\subseteq x
\]
\begin{itemize}
\item
$\bool{Z}^-_\kappa$ is the $\in_{\Delta_0}\cup\bp{\kappa}$-theory $\bool{Z}^-_{\Delta_0}\cup\bp{\phi_{\mathrm{Card}}(\kappa)}$.
\item
Accordingly we define the $\in_{\Delta_0}\cup\bp{\kappa}$-theories $\ZF^-_\kappa$, $\ZFC^-_\kappa$, $\ZF_\kappa$, 
$\ZFC_\kappa$.
\end{itemize}
\end{notation}

\begin{theorem}\label{thm:keypropCod}
Assume $\ZFC^{-}_{\kappa}$. The following holds\footnote{Many transitive supersets of $H_{\kappa^+}$ are 
$\in_{\Delta_0}\cup\bp{\kappa}$-models of $\ZFC^{-}_{\kappa}$ for $\kappa$ an infinite cardinal (see \cite[Section IV.6]{KUNEN}). 
To simplify notation we assume to have fixed a transitive 
$\in_{\Delta_0}\cup\bp{\kappa}$-model $\mathcal{N}$ of $\ZFC^{-}_\kappa$
with domain $N\supseteq H_{\kappa^+}$. The reader can easily realize that all these statements holds for an arbitrary model $\mathcal{N}$ of $\ZFC^-_\kappa$ replacing $H_{\kappa^+}$ with its version according to $\mathcal{N}$.}:
 \begin{enumerate}
\item
The map $\mathrm{Cod}_\kappa$ and $\WFE_\kappa$ are defined by $\bp{\in,\kappa}$-formulae which are
$\Delta_1(\ZFC^-_\kappa)$. Moreover $\Cod_\kappa:\WFE_\kappa\to H_{\kappa^+}$
is surjective (provably in $\ZFC^{-}_{\kappa}$), and
$\WFE_\kappa$ is defined by a universal 
$\in_{\Delta_0}\cup\bp{\kappa}$-formula with quantifiers
ranging over subsets of $\kappa^{<\omega}$.
\item 
There are existential $\in_{\Delta_0}\cup\bp{\kappa}$-formulae (with quantifiers
    ranging over subsets of $\kappa^{<\omega}$), $\phi_\in,\phi_=$ such that
for all $R,S\in \WFE_\kappa$, $\phi_=(R,S)$ if and only if $\Cod_\kappa(R)=\Cod_\kappa(S)$ and 
$\phi_\in(R,S)$  if and only if $\Cod_\kappa(R)\in\Cod_\kappa(S)$. In particular letting
\[
E_\kappa=\bp{(R,S)\in \WFE_\kappa: \phi_\in(R,S)},
\]
\[
\cong_\kappa=\bp{(R,S)\in \WFE_\kappa: \phi_=(R,S)},
\]
$\cong_\kappa$ is a $\ZFC^-_\kappa$-provably definable equivalence relation, $E_\kappa$ respects it, and
\[
(\WFE_\kappa/_{\cong_\kappa}, E_\kappa/_{\cong_\kappa})
\]
is 
isomorphic to $(H_{\kappa^+},\in)$ via the map $[R]\mapsto \Cod_\kappa(R)$.
\end{enumerate}
\end{theorem}

\begin{proof}
A detailed proof requires a careful examination of the syntactic properties of $\Delta_0$-formulae, in line with the one carried in Kunen's \cite[Chapter IV]{KUNEN}.
We outline the main ideas, 
following (as we already did) Kunen's book terminology for certain 
set theoretic operations on sets, functions and relations (such as $\dom(f),\ran(f)$, $\text{Ext}(R)$, etc). 
To simplify the notation, we prove the results
for a transitive $\ZFC^-$-model $(N,\in)$ which is then extended
to a structure $(N,\in_{\Delta_0}^N,\kappa^N)$ which 
models $\ZFC^-_\kappa$, and whose domain contains $H_{\kappa^+}$. 
The reader can verify by itself that the argument 
is modular and works for any other model of 
$\ZFC^-_\kappa$ 
(transitive or ill-founded, containing the ``true'' $H_{\kappa^+}$ or not). 
\begin{enumerate} 
\item This is proved in details in \cite[Chapter IV]{KUNEN}.
To define $\WFE_\kappa$ by a universal $\in_{\Delta_0}\cup\bp{\kappa}$-property over subsets of $\kappa$
and $\Cod_\kappa$ by a $\Delta_1$-property for $\in_{\Delta_0}\cup\bp{\kappa}$ over $H_{\kappa^+}$, we proceed
as follows:
\begin{itemize}
\item
\emph{$R$ is an extensional relation with top element $0$}
is defined by the $\in_{\Delta_0}$-atomic formula 
$\psi_{\mathrm{EXT}}(R)$ $\ZF^{-}_{\Delta_0}$-provably equivalent to the $\Delta_0$-formula:
\begin{align}\label{eqn:EXTENS}
(R\emph{ is a binary relation})\wedge (0\in \text{Ext}(R))\wedge\\ \nonumber
\wedge\forall z,w\in\text{Ext}(R)\,
[\forall u\in\text{Ext}(R)\,(u\mathrel{R}z\leftrightarrow u\mathrel{R}w)\rightarrow (z=w)]\wedge\\ \nonumber
\wedge \forall z\in\text{Ext}(R)\, (0\neq z\rightarrow \exists y\in\text{Ext}(R)\, z\mathrel{R}y). \nonumber
\end{align}
\item 
$\WFE_\kappa$ is defined by the universal $\in_{\Delta_0}\cup\bp{\kappa}$-formula 
$\phi_{\WFE_\kappa}(R)$ (quantifying only over subsets of $\kappa^{<\omega}$)
\begin{align*}
\psi_{\mathrm{EXT}}(R)\wedge \\
 (\text{Ext}(R)\in\kappa\vee \text{Ext}(R)=\kappa)\wedge\\
\wedge [\forall f\,(f:\omega\to\text{Ext}(R)\text{ is a function }\rightarrow\exists n\in\omega \,
\neg(\ap{f(n+1),f(n)}\in R))].
\end{align*}

\item To define $\Cod_\kappa$,
consider the $\in_{\Delta_0}$-atomic formula 
$\psi_{\Cod}(G,R)$ provably equivalent to the $\in_{\Delta_0}$-formula:
\begin{align*}
\psi_{\mathrm{EXT}}(R)\wedge\\
\wedge (G\text{ is a function})\wedge\\
\wedge (\dom(G)=\text{Ext}(R))\wedge (\ran(G)\text{ is transitive})\wedge\\ \wedge\forall\alpha,\beta\in\text{Ext}(R)\,[\alpha\mathrel{R}\beta\leftrightarrow G(\alpha)\in G(\beta)].
\end{align*}

Then $\Cod_\kappa(R)=a$ can\footnote{Note that $\Cod_\kappa$ can be defined without any reference  to $\kappa$. This reference appears once we decide to restrict the domain of $\Cod_\kappa$ to $\WFE_\kappa$.} be defined either by the existential $\in_{\Delta_0}$-formula\footnote{Given an $R$ such that $\psi_{\mathrm{EXT}}(R)$ holds,
\emph{$R$ is a well founded relation} holds in a model of 
$\ZFC^-_\kappa$
if and only if 
$\Cod_\kappa$ is defined on $R$. In the theory $\ZFC^-_\kappa$, $\WFE_\kappa$ can be defined using a universal property 
by a $\in_{\Delta_0}\cup\bp{\kappa}$-formula quantifying only over subsets of $\kappa$. On the other hand if we allow arbitrary quantification over elements 
of $H_{\kappa^+}$, we can express the well-foundedness of $R$ also using the existential formula 
$\exists G\,\psi_{\Cod_\kappa}(G,R)$. This is why $\WFE_\kappa$ is defined by a universal 
$\in_{\Delta_0}\cup\bp{\kappa}$-property in the structure $(\pow{\kappa},\in_{\Delta_0}^V,\kappa)$,
while the graph of $\Cod_\kappa$ can be defined by a $\Delta_1$-property for $\in_{\Delta_0}\cup\bp{\kappa}$
in the structure $(H_{\kappa^+},\in_{\Delta_0}^V,\kappa^V)$.} 
\[
\exists G\,(\psi_{\Cod}(G,R)\wedge G(0)=a)
\]
or by the universal $\in_{\Delta_0}$-formula 
\[
\forall G\,(\psi_{\Cod}(G,R)\rightarrow G(0)=a).
\]
\end{itemize}

\item The equality relation in $H_{\kappa^+}$ is transferred to the isomorphism relation
between elements of $\WFE_\kappa$: if $X,Y$ are well-founded extensional on $\kappa$ with a top-element,
the Mostowski collapsing theorem entails that $\Cod_\kappa(X)=\Cod_\kappa(Y)$ if and only if 
$(\mathrm{Ext}(R),R)\cong(\mathrm{Ext}(S),S)$. 
Isomorphism of the two structures $(\mathrm{Ext}(X),X)\cong(\mathrm{Ext}(Y),Y)$ is expressed by the $\Sigma_1$-formula
for $\tau_{\kappa}$:
\[
\phi_=(X,Y)\equiv 
\exists f\,(f \text{ is a bijection of $\text{Ext}(X)$ onto $\text{Ext}(Y)$ and $\alpha \mathrel{X}\beta$ if and only if 
$f(\alpha) \mathrel{Y} f(\beta)$}).
\]
In particular we get that $\phi_=(X,Y)$ holds in $H_{\kappa^+}$ for $X,Y\in \WFE_\kappa$
if and only if $\Cod_\kappa(X)=\Cod_\kappa(Y)$.

Similarly one can express $\Cod_\kappa(X)\in\Cod_\kappa(Y)$ by the $\Sigma_1$-property $\phi_\in$
in $\tau_{\kappa}$
stating that
$(\mathrm{Ext}(X),X)$ is isomorphic to $(\mathrm{pred}_Y(\alpha),Y)$ for some $\alpha\in\kappa$ with $\alpha \mathrel{Y}0$, 
where $\mathrm{pred}_Y(\alpha)$ is given by the elements
of $\mathrm{Ext}(Y)$ which are $Y$-connected by a finite path to $\alpha$. 

Moreover letting $\cong_\kappa\subseteq \WFE_\kappa^2$ denote the isomorphism relation between elements of $\WFE_\kappa$
and $E_\kappa\subseteq \WFE_\kappa^2$ denote the relation which translates into the $\in$-relation via $\Cod_\kappa$, 
it is clear that $\cong_\kappa$ is a congruence relation
over $E_\kappa$, i.e.: if $X_0 \cong_\kappa X_1$ and $Y_0\cong_\kappa Y_1$,
$X_0 \mathrel{E_\kappa} Y_0$ if and only if $X_1 \mathrel{E_\kappa}  Y_1$.

This gives that the structure $(\WFE_\kappa/_{\cong_\kappa}, E_\kappa/_{\cong_\kappa})$ is 
isomorphic to $(H_{\kappa^+},\in)$ via the map $[X]\mapsto \Cod_\kappa(X)$ 
(where $\WFE_\kappa/_{\cong_\kappa}$ is the set
of equivalence classes of $\cong_\kappa$ and the quotient relation $[X] \mathrel{E_\kappa/_{\cong_\kappa}} [Y]$ holds 
if and only if $X \mathrel{E_\kappa}  Y$).

This isomorphism is defined via the map $\Cod_\kappa$, which is by itself defined by 
a $\ZFC^-_\kappa$-provably $\Delta_1$-property for $\in_{\Delta_0}\cup\bp{\kappa}$.

The very definition of $\WFE_\kappa,\cong_\kappa,E_\kappa$ show that
\[
\WFE_\kappa=\phi_{\WFE_\kappa}^{N},
\]
\[
\cong_\kappa=(\phi_{\WFE_\kappa}(x)\wedge \phi_{\WFE_\kappa}(y)\wedge \phi_{=}(x,y))^{N},
\]
\[
E_\kappa=(\phi_{\WFE_\kappa}(x)\wedge \phi_{\WFE_\kappa}(y)\wedge \phi_{\in}(x,y))^{N}.
\]
\end{enumerate}
Note that we crucially use the axiom of choice to prove the surjectivity of $\Cod_\kappa$ on
$H_{\kappa^+}$.
\end{proof}

%

\subsubsection{Model completeness for the theory of $H_{\kappa^+}$}

The following definition isolates those signatures extending $\in_{\Delta_0}$ with predicates defining a family of subsets of $\bigcup_{n\in\omega}\pow{\kappa}^n$ closed under finite unions, complementations, and projections.
\begin{definition}\label{def:projclosureaxioms}
Let $\tau$ be a signature extending $\in_{\Delta_0}\cup\bp{\kappa}$.

$\tau$ is $\kappa$-projectively closed for some $\tau$-theory\footnote{Recall Notation \ref{not:basicsettheorynot} for $\ZFC_\tau$, $\ZF^-_\tau$,\dots} $T\supseteq\ZFC^-_\tau$ if the following holds for a certain family of predicate symbols in $\tau_0\subseteq\tau$:

\begin{enumerate}
\item\label{ax:prrojclos0}
\[
T\models(\kappa\emph{ is a cardinal})
\]
\item\label{ax:prrojclos1}
\[
T\models\forall x_0,\dots,x_n\,(R(x_0,\dots,x_n)\rightarrow\bigwedge_{i=0}^n x_i\subseteq\kappa)
\]
for all predicate symbols $R\in \tau_0$.
\item \label{ax:prrojclos2}
For all quantifier free $\tau$-formulae $\phi$ there exists $R_\phi\in\tau_0$ such that
\[
T\models\forall x_0,\dots,x_n,y_0,\dots,y_n\,
\qp{(\bigwedge_{i=0}^n x_i=\Cod_\kappa(y_i)\wedge\phi(x_0,\dots,x_n))\leftrightarrow
R_\phi(y_0,\dots,y_n)}
\]

\item\label{ax:prrojclos3}
For all  $R,S\in\tau_0$ there is some $U$ also in 
$\tau_0$ such that 
\[
T\models\forall \vec{x},\vec{y}\,
\qp{(R(\vec{x})\vee S(\vec{y}))\leftrightarrow
U(\vec{x},\vec{y})}.
\]

\item \label{ax:prrojclos4}
For all  $R,S\in\tau_0$ there is some $U$ also in 
$\tau_0$ such that 
\[
T\models\forall \vec{x},\vec{y}\,
\qp{(R(\vec{x})\wedge S(\vec{y}))\leftrightarrow
U(\vec{x},\vec{y})}.
\]
\item \label{ax:prrojclos5}
For all  $R\in\tau_0$ there is some $U$ also in 
$\tau_0$ such that
\[
T\models\forall x_0,\dots,x_n\,
\qp{\bigwedge_{i=0}^n x_i\subseteq\kappa \rightarrow (R(x_0,\dots,x_n)\leftrightarrow
\neg U(x_0,\dots,x_n)}
\]

\item \label{ax:prrojclos6}
For all  $R\in\tau_0$ and $i=0,\dots,n$ there is some $U$ also in 
$\tau_0$ such that
\[
T\models\forall x_0,\dots,x_{i-1},x_{i+1},\dots,x_n\,
\qp{U(x_0,\dots,x_{i-1},x_{i+1},\dots,x_n)\leftrightarrow
\exists x_i\,(x_i \subseteq\kappa \wedge R(x_0,\dots,x_n))}.
\]
\end{enumerate}
We just write 
\emph{projectively closed} if $\kappa=\omega$ and $\tau_0=\tau\setminus\in_{\Delta_0}$.
\end{definition}

The following is a trivial but fundamental remark:
\begin{fact}
Let $\tau$ be a $\kappa$-projectively closed family.
Axioms \ref{ax:prrojclos0} to \ref{ax:prrojclos6} are all $\Pi_2$-sentences for $\tau$.
\end{fact}

\begin{lemma}\label{lem:existkappaprojsign}
Let $\tau\supseteq\in_{\Delta_0}$,
$T\supseteq\ZFC_\tau$ be a $\tau$-theory, and $\kappa$ be a $T$-definable cardinal.

Assume $\phi$ is a $\tau$-formula which defines in any model $(V,\tau^V)$ 
of $T$ a transitive model $M_\phi\subseteq V$ of $\ZF^-_\tau$ which contains $H_{\kappa^+}^V$.
Let $A=\bar{A}\times\bp{0}$, where
\[
\bar{A}=\bp{\theta^{M_\phi}(x_0,\dots,x_n)\wedge\bigwedge_{i=0}^nx_i\subseteq\kappa:\, \theta\emph{ a $\tau$-formula}}.
\]
Then $\tau_A$ is $\kappa$-projectively closed for $T+T_{\tau,A}$. 
\end{lemma}
\begin{proof}
This is an almost immediate consequence of the fact that whenever $(V,\tau^V)$ models $T$, $M_\phi\supseteq H_{\kappa^+}^V$ is a transitive model of $\ZF^-_\tau$, hence such that 
$\Cod_\kappa^V=\Cod_\kappa^M$ and
$\pow{\kappa}^V=\pow{\kappa}^M$. Using this fact together with the fact that $(M_\phi,\tau^V)$ is a model of $\ZF^-_\tau$, we get that the collection of subsets of $\pow{\kappa}^V$ definable in $M_\phi$ without (or even with) parameters defines a family which is closed under projections (e.g. Axiom \ref{ax:prrojclos6}), finite unions and complementations (e.g. Axioms \ref{ax:prrojclos3}, \ref{ax:prrojclos4}, \ref{ax:prrojclos5}), and satisfies also axiom \ref{ax:prrojclos2}.

The conclusion follows.
\end{proof}

\begin{theorem}\label{thm:modcompHkappa+}
Let $\tau$ be a $\kappa$-projectively closed signature for some 
$\tau$-theory $T\supseteq\ZFC_\tau$.

Let $S$ be the $\tau$-theory extending $\ZFC^-_\tau$ 
with the axioms \ref{ax:prrojclos0},\dots,\ref{ax:prrojclos6} which follow from $T$ and the axiom
\[
\forall x\exists f\,\qp{f\text{ is a surjection}\wedge\dom(f)=\kappa\wedge\ran(f)=x}.
\]
Then $S$ is model complete.
\end{theorem}

\begin{proof}
To simplify notation, 
we conform to the assumption of Thm. \ref{thm:keypropCod}, 
i.e. 
we assume that the model $(N,\tau^N)$ of $S$
on which we work is a transitive superstructure of 
$H_{\kappa^+}$.

The statement \emph{every set has size at most $\kappa$} is satisified by a
$\ZF^-$-model $(N,\in^N,\kappa^N)$ with
$N\supseteq H_\kappa^+$ if and only if $N=H_{\kappa^+}$.
From now on we proceed assuming this equality.

It suffices to show that
for all $\tau$-formulae $\phi(\vec{x})$ 
\[
S\models
\forall\vec{x} \,(\phi(\vec{x})\leftrightarrow\psi_\phi(\vec{x})),
\]
for some universal $\tau$-formula $\psi_\phi$.

By Axiom \ref{ax:prrojclos1} applied to the atomic $\tau$-formulae $x=x$, $x\in y$ and $x=y$, we obtain
predicate symbols
$R_{\WFE_\kappa}$, $R_{=}$, $R_{\in}$ in $\tau_0$ such that:
\[
R_{\WFE_\kappa}^{H_{\kappa^+}}=\WFE_\kappa,
\]
\[
R_{x=y}^{H_{\kappa^+}}=\cong_\kappa=\bp{(X,Y)\in (\WFE_\kappa)^2:\, \Cod_\kappa(X)=\Cod_\kappa(Y)},
\]
\[
R_{x\in y}^{H_{\kappa^+}}=E_\kappa=\bp{(X,Y)\in (\WFE_\kappa)^2:\, \Cod_\kappa(X)\in\Cod_\kappa(Y)}.
\]

Now by assumptions on $\tau$, we get that for any quantifier free $\tau$-formula $\phi(\vec{x})$, there is some predicate $R_\phi\in \tau_0$ such that:
\[
S\models
R_\phi(x_0,\dots,x_n)\leftrightarrow
\exists y_0,\dots,y_n\,\qp{\phi(y_0,\dots,y_n)\wedge\bigwedge_{i=0}^n y_i=\Cod_\kappa(x_i)}.
\] 
Now
we proceed to define $R_\psi(\vec{x})$ for any $\tau$-formula $\phi$ letting:
\begin{itemize}
\item $R_{\psi\wedge\psi}(\vec{x})$ be $U(\vec{x})$ for the $U$ given by Axiom \ref{ax:prrojclos4}
applied to $R_\psi,R_\phi$,
\item $R_{\neg\psi}(\vec{x})$ be $U(\vec{x})$ for the $U$ given by Axiom \ref{ax:prrojclos5}
applied to $R_\psi,R_\phi$,
\item $R_{\exists y\psi(y,\vec{x})}(\vec{x})$ be $U(\vec{x})$ for the $U$ given by Axiom \ref{ax:prrojclos6}
applied to $R_\psi$.
\end{itemize}
An easy induction on the complexity of the $\tau_0$-formulae 
$R_\phi(\vec{x})$
gives that for any $\tau$-definable subset
$X$ of $(H_{\kappa^+})^n$ which is the extension of some $\tau$-formula
$\phi(x_1,\dots,x_n)$ 
\[
\bp{(Y_1,\dots,Y_n)\in (\WFE_\kappa)^n:\, (\Cod_\kappa(Y_1),\dots,\Cod_\kappa(Y_n))\in X}=
R_{\phi}^{H_{\kappa^+}},
\] 
with the further property that $R_{\phi}^{H_{\kappa^+}}\subseteq (\WFE_{\kappa})^n$ respects the
$\cong_\kappa$-relation.

Then for any $\tau$-formula $\phi(x_1,\dots,x_n)$
$(H_{\kappa^+},\tau^{H_{\kappa^+}})\models \phi(a_1,\dots,a_n)$ if and only if 
\[
(\WFE_\kappa/_{\cong_\kappa}, E_\kappa/_{\cong_\kappa}, R_{\phi}^{H_{\kappa^+}}:\,\phi\emph{ a $\tau$-formula})\models \phi([X_1],\dots,[X_n])
\]
whenever $\Cod_\kappa(X_i)=a_i$ for $i=1,\dots,n$,
if and only if 
\[
(H_{\kappa^+},\tau^{H_{\kappa^+}})\models \forall X_1,\dots,X_n\, [(\bigwedge_{i=1}^n R_{\WFE_\kappa}(X_i)\wedge\Cod_\kappa(X_i)=a_i)\rightarrow R_{\phi}(X_1,\dots,X_n)].
\]

Since this argument can be repeated verbatim for any model of 
$S$,
we have proved the following:
\begin{claim}
For any $\tau$-formula $\phi(x_1,\dots,x_n)$,
$S$ proves
\[
\forall x_1,\dots,x_n\,[\phi(x_1,\dots,x_n)\leftrightarrow \forall y_1,\dots,y_n\,[(\bigwedge_{i=1}^n R_{\WFE_\kappa}(y_i)\wedge\Cod_\kappa(y_i)=x_i)\rightarrow R_{\phi}(y_1,\dots,y_n)]].
\]
\end{claim}
But $\Cod_\kappa(y)=x$ is expressible provably in $S$ by an existential  $\tau$-formula (since $\tau\supseteq\in_{\Delta_0}$).
Therefore
\[
\forall y_1,\dots,y_n\,[(\bigwedge_{i=1}^n R_{\WFE_\kappa}(y)\wedge \Cod_\kappa(y_i)=x_i)\rightarrow R_{\phi}(y_1,\dots,y_n)]
\]
is a universal $\tau$-formula, and we are done. 
\end{proof}

\subsubsection{Proof of Thm.~ \ref{Thm:AMCsettheory+Repl}\ref{Thm:AMCsettheory+Repl-1}}

Conforming to the notation of Thm.~ \ref{Thm:AMCsettheory+Repl}\ref{Thm:AMCsettheory+Repl-1},
it is clear that if $T$ extends $\ZFC$ and $\kappa$ is a $T$-definable cardinal
\[
\bar{A}_\kappa=\bp{\phi^{\pow{\kappa}}(x_1,\dots,x_n)\wedge \bigwedge_{i=1}^n x_i\subseteq\kappa:\, \phi\emph{ an $\in$-formula}},
\] 
we get that
\[
A_\kappa=\bp{(\phi,0): \,\phi\in\bar{A}_\kappa}\cup D
\]
(where $D$ is the set used in Notation \ref{not:basicsettheorynot})
is such that $\in_{A_\kappa}$ is $\kappa$-projectively closed for $\ZFC$.

Therefore the following result completes the proof of Thm.~\ref{Thm:AMCsettheory+Repl}\ref{Thm:AMCsettheory+Repl-1}.

\begin{theorem} \label{Thm:mainthm-1}
Let $\tau$ be $\kappa$-projectively closed for some
$\tau$-theory $T\supseteq \ZFC_\tau$ such that 
\[
H_{\kappa^+}^\mathcal{M}\prec_1 \mathcal{M}
\]
whenever $\mathcal{M}$ models $T$.

Then $T$ has an AMC $T'$ in signature $\tau$.
\end{theorem}

Observe that the assumptions of the theorem apply whenever $\tau\setminus\tau_0$ is $\in_{\Delta_0}\cup\bp{\kappa}$ or any signature satisfying the assumptions of Lemma
\ref{lem:levabsgen}.

\begin{proof}
By Thm. \ref{thm:modcompHkappa+}, any $\tau$-theory extending
\[
\ZFC^{-}_{\tau}+\emph{every set has size at most $\kappa$}
\] 
with the axioms \ref{ax:prrojclos0},\dots,\ref{ax:prrojclos6} which follow from $T$
is model complete.

Therefore:
\begin{itemize}
\item
$H_{\kappa^+}^\mathcal{M}$ models $\ZFC^{-}_{\tau}$, by the standard properties of $H_{\kappa^+}$ and the regularity of $\kappa^+$ (any function $F$ with domain in $H_{\kappa^+}$ and range contained in $H_{\kappa^+}$ is in $H_{\kappa^+}$).
\item
Clearly $H_{\kappa^+}^\mathcal{M}$ models \emph{every set has size at most $\kappa$}.
\item
 $H_{\kappa^+}^\mathcal{M}$ models 
the axioms \ref{ax:prrojclos0},\dots,\ref{ax:prrojclos6} which follow from $T$, as those are all expressible by $\Pi_2$-sentences for $\tau$ holding in $\mathcal{M}$.
\end{itemize}
Therefore
\[
T'=\bp{\phi: H_{\kappa^+}^\mathcal{M}\models \phi,\,
\mathcal{M}\models T}
\] 
is model complete.

Clearly  $T_{\forall\vee\exists}'=T_{\forall\vee\exists}$, since 
$H_{\kappa^+}^\mathcal{M}\prec_1 \mathcal{M}$ whenever $\mathcal{M}\models T$.

The theorem is proved.
\end{proof}

\begin{remark}
Note that there is no reason to expect that the family of models 
$\bp{H_{\kappa^+}^\mathcal{M}:\,\mathcal{M}\models T}$ we used to define 
$T'$ is an elementary class for $\tau$.

By Lemma \ref{lem:existkappaprojsign} for any $\in$-theory $S\supseteq\ZFC$ and any $S$-definable cardinal $\kappa$ 
there is plenty of $A\supseteq\bool{Form}_\in\times 2$ such that $\in_A$ is $\kappa$-projectively closed for 
$\ZFC^-_{\tau_A}$.

However linking model companionship results for set theory to forcibility as we do in 
Theorem \ref{Thm: mainthmforcibility}
requires much more care in the definition of the signature.
We will pursue this matter in more details in the next sections.  
\end{remark}

%
%


\subsection{Absolute codings of the continuum in type $\omega_2$}\label{subsec:AMCnegCH}

We prove variants of Theorems \ref{mainthm:CH*} and \ref{mainthm:2omegageqomega2*}.
Here we  focus on the strong consistency hull of set theory  as formalized in a given signature (which -when model complete- is also its AMC)  rather than on its possible Kaiser hull (which -when model model complete- is only its MC) (recall Def. \ref{def:SCHKH}, Fact \ref{rem:SCHKH}, Lemma \ref{fac:proofthm1-2}, \cite[Lemma 3.2.12, Lemma 3.2.13, Thm. 3.2.14]{TENZIE}). 

We start with Thm. \ref{mainthm:CH*}. We need the following:

\begin{proposition}
Assume $\tau\supseteq\in_{\Delta_0}$ and $S\supseteq\ZFC_\tau$ is a $\tau$-theory such that $S+\neg\CH$ is consistent. Then $\CH$ is not in $\bool{SCH}(S)$.
\end{proposition}
\begin{proof}
Note that $\neg\CH$ 
as formalized in \ref{subsec:AMCCH} is expressible in signature $\tau$ as the conjunction of the purely $\Pi_2$-sentence for $\in_{\Delta_0}$ \emph{``every function with domain $\aleph_1$ is not surjective on the reals''} with the the purely $\Sigma_2$-sentence for $\in_{\Delta_0}$ \emph{``$\aleph_1$ exists''}.

Either $\bool{SCH}(S)$ does not model \emph{``$\aleph_1$ exists''} in which case there is a model of 
$\bool{SCH}(S)+\neg\CH$ (e.g. a model of $\bool{SCH}(S)+$\emph{``all sets are countable''}), or $\bool{SCH}(S)$ models \emph{``$\aleph_1$ exists''}. In this case we note that $\CH$ holds in $\bool{SCH}(S)$ if and only if it $\bool{SCH}(S)$ models the $\Sigma_2$-sentence $\psi$ for $\in_{\Delta_0}$
stating that \emph{``there is a bijection $f$ with domain $\omega_1$ and whose range contains all subsets of $\omega$''}, and also that $S+\neg\CH$ is consistent and logically equivalent to $S+\neg\psi$.
By Fact \ref{rem:SCHKH}\ref{rem:SCHKH-4} we can find a model of $\bool{SCH}(S)+\neg\psi$, hence also in this case 
$\bool{SCH}(S)$ does not model $\CH$.
\end{proof}

\begin{remark}
The above Proposition is the instantiation to $\neg\CH$ of a result which can be obtained for any
$\tau\supseteq\in_{\Delta_0}$, any $\psi$ which is a $\Pi_2$-sentences for $\tau$, any $S\supseteq \ZFC_\tau$ such that $\psi+S$ is consistent.
We can use Proposition \ref{prop:sigma2notinSCH} to infer that $\neg\psi\not\in\bool{SCH}(S)$.
\end{remark}

%

\begin{proposition}\label{prop:geninv}
Let $C\subseteq \bool{Form}\in\times 2$ be such that 
$\ap{(x\text{ is }\aleph_1),1}\in C$ and $\in_C\supseteq\in_{\Delta_0}$; let
$S\supseteq\ZFC$.
Assume that $(S+{T}_{\in,C})_{\forall\vee\exists}$ is invariant across forcing extensions\footnote{E.g. for any universal $\in_C$-sentence $\psi$ 
$(V,\in_C^V)\models\psi$ if and only if so does  $(V[G],\in_C^{V[G]})$ in any forcing extension of $V$.}
of any model of $S$.
Then $\neg\CH\in \bool{SCH}(S+T_{\in,C})$.
\end{proposition}
\begin{proof}
We have that the $\Pi_2$-sentence for $\in_C$ 
\[
\forall f (\dom(f) \text{ is }\aleph_1\wedge f\text{ is a function})\rightarrow \exists r\,(r\subseteq\omega\wedge r\not\in\ran(f))]
\] 
is strongly consistent with
$S_{\forall\vee\exists}$ as it can be forced by Cohen's original forcing over any model of $S$, and (by assumption) such forcings do not change  the universal $\in_C$-theory of the model.
\end{proof}

The two Propositions combine with Thm. \ref{thm:PI1invomega2} of the next Section to give  the following variant of Thm. \ref{mainthm:CH*}:
\begin{corollary}\label{cor:weakCH*}
Assume $S_0$ is any extension of $\ZFC+$\emph{there are class many Woodin cardinals}. Then:
\begin{itemize}
\item
$\CH\in\bool{SCH}(S_0+T_{\in,A})$ for no $A\subseteq \bool{Form}\in\times 2$ such that $\in_A\supseteq\in_{\Delta_0}$;
\item
$\neg\CH\in\bool{SCH}(S_0+T_{\in,C})$ for $C\subseteq \bool{Form}\in\times 2$ such that 
$\in_C=\in_{\Delta_0}\cup\bp{\omega_1}$ and $\bp{\ap{(x\text{ is }\aleph_1),1}\in C}$.
\end{itemize}
\end{corollary}

\begin{proof}
By Thm. \ref{thm:PI1invomega2} the truth value of universal $\in_C$-sentences is invariant across forcing extensions of any model of $S_0$. Hence we can apply the previous propositions to get the desired conclusions.
\end{proof}

The first item of Theorem \ref{mainthm:CH*} follows from the above Corollary. To get the second item it suffices to find $B\subseteq \bool{Form}\in\times 2$ and $S\supseteq S_0$ such that $\in_B$ is invariant across forcing extensions of models of $S$ and $B\in\SpecAMC{S}$. Then again the second Proposition above yields
the second part of Theorem \ref{mainthm:CH*}. We will define such a $B$ and $S$ in Section \ref{sec:Homega2}.

\begin{remark} 
Corollary \ref{cor:weakCH*} is the instantiation for $\neg\CH$ of a results which can be obtained for any $\Pi_2$-sentence $\psi$ formalized in a signature
$\in_C$ and a set theory $T$ formalized in $\in$ so that:
\begin{itemize}
\item
The universal $\in_C$-theory of $(V,\in)$ is invariant across forcing extensions of $V$ in any model $(V,\in)$ of $T$;
\item
$(H_\kappa^V,\in_C^V)\prec_1(V,\in_C^V)$ holds in all models $(V,\in)$ of $T$;
\item
$\psi^{H_\kappa}$ is forcible over any model of $T$. 
\end{itemize}
\end{remark}
We use below this same pattern to analyze what can be said with regard to $2^{\aleph_0}>\aleph_2$ by looking at the AMC-spectrum of set theory.


The rest of this section is dedicated to the proof of Thm. \ref{mainthm:2omegageqomega2*} (or of its variant  in line with Corollary \ref{cor:weakCH*}).
\begin{definition}
Let $\psi_{\aleph_1}(x)$ be the boolean combination of universal formulae for $\in_{\Delta_0}$ whose unique solution is the first uncountable cardinal\footnote{Previously we noted this formula by $(x\text{ is }\aleph_1)$ however for the sake of readabillty, we introduce below a constant symbol $\omega_1$ to denote the least uncountable cardinal and we prefer to write  $\psi_{\aleph_1}(\omega_1)$ rather than $(\omega_1\text{ is }\aleph_1)$, other occasions showing that this new convention is convenient will arise along this section.}.
$c:[\omega_1]^2\to \omega$ is a \emph{ladder system on $\omega_1$} if for every $\alpha$ limit point of $\omega_1$
$c\restriction\bp{\alpha}\times\alpha\to \omega$ is surjective and monotone increasing.

$P:\omega_1\to \omega$ is a \emph{partition of $\omega_1$ in countably many stationary sets} if $P^{-1}[{n}]$ is stationary  for all 
$n\in\omega$ and $P$ is surjective.
\end{definition}

\begin{fact}
There is a $\Delta_0$-formula $\psi_{\mathrm{ladder}}(x,y)$ such that any solution of
$\psi_{\mathrm{ladder}}(x,a)$ in any $\in_{\Delta_0}\cup\bp{a}$-model of
\[
\ZF_{\Delta_0}+\psi_{\aleph_1}(a)
\]
defines a ladder system on $\omega_1$ according to the model.

Accordingly there is a $\Pi_1$-formula $\psi_{\mathrm{countpartstat}}(x,y)$ for $\in_{\Delta_0}$ 
such that any solution of
$\psi_{\mathrm{countpartstat}}(x,a)$ in any model of
\[
\ZF_{\Delta_0}+\psi_{\aleph_1}(a)
\]
defines a partition of $\omega_1$ in countably many stationary sets  according to the model.

Furthermore 
\[
\ZFC^-_{\Delta_0}+\psi_{\aleph_1}(a)\models\exists z\, \psi_{\mathrm{ladder}}(z,a)\wedge \exists y\, \psi_{\mathrm{countpartstat}}(y,a).
\]
\end{fact}
%

\begin{theorem}[Moore \cite{MOO06}]\label{thm:moo06}
Let $T$ be the $\in_{\Delta_0}\cup\bp{\omega_1,c,A}$-theory 
\[
\ZFC^-_{\Delta_0}+\psi_{\aleph_1}(\omega_1)+\psi_{\mathrm{countpartstat}}(A,\omega_1)+
\psi_{\mathrm{ladder}}(c,\omega_1).
\]
There is an $\in_{\Delta_0}$-formula\footnote{$\psi_{\mathrm{Moore}}(r,\alpha,f,C,\omega_1,c,A)$ states that 
$f:\omega_1\to \alpha$ is a surjection, $C$ is a club subset of $\omega_1$ and for all $\alpha\in C$, if $A(\alpha)=n$, then 
for eventually all $\beta\in\alpha\cap C$, $c(\beta,\alpha)<c(\otp(f[\beta]),\otp(f[\alpha])$ if and only if $r(n)=1$.} 
$\psi_{\mathrm{Moore}}(x,y,u_0,u_1,z_0,z_1,z_2)$ 
such that
in any  $\in_{\Delta_0}\cup\bp{\omega_1,c,A}$-model $\mathcal{M}$ of $T$:
\begin{enumerate}[(a)]
\item \label{thm:moo06-1}
If 
\[
\mathcal{M}\models\psi_{\mathrm{Moore}}(r,\alpha,f,C,\omega_1,c,A),
\]
then
\[
\mathcal{M}\models (\alpha\text{ is an ordinal})\wedge(f:\alpha\to \omega_1\text{ is an injection})\wedge (C\subseteq\omega_1\text{ is a club})\wedge  (r\subseteq \omega).
\]
\item \label{thm:moo06-2}
If 
\[
\mathcal{M}\models r,s\subseteq \omega\wedge \exists h\,(h:\alpha\to \omega_1\text{ is an injection}),
\]
then
\[
\mathcal{M}\models \forall f,g,C,D\,\qp{(\psi_{\mathrm{Moore}}(s,\alpha,f,C,\omega_1,c,A)\wedge \psi_{\mathrm{Moore}}(r,\alpha,g,D,\omega_1,c,A))\rightarrow r=s}.
\]
\item \label{thm:moo06-3}
$T+\ZFC_{\Delta_0}$ proves that 
\[
\forall r\subseteq\omega\, \exists\alpha,f,C\,\psi_{\mathrm{Moore}}(r,\alpha,f,C,\omega_1,c,A)
\]
is forcible by a proper forcing.
\end{enumerate}
\end{theorem}
By item \ref{thm:moo06-3} of the Theorem  (and Thm. \ref{thm:PI1invomega2}), if $S$ is 
$T+\ZFC_{\Delta_0}+$\emph{there are class many Woodin cardinals}
\[
\forall r\subseteq\omega \exists\alpha\,\psi_{\mathrm{Moore}}(r,\alpha,\omega_1,c,A)
\]
holds in a model of $S+R_\forall$ for any consistent $\in_{\Delta_0}\cup\bp{\omega_1,c,A}$-theory $R$ extending $S$.

\begin{corollary}\label{cor:2omega>omega2}
Let $\theta_{\mathrm{Moore}}$ be the $\Pi_2$-sentence for $\in_{\Delta_0}$
\[
\forall x,y,z\,\qp{(\psi_{\aleph_1}(x)\wedge
\psi_{\mathrm{ladder}}(y,x)\wedge\psi_{\mathrm{countpartstat}}(z,x))\rightarrow \forall r\subseteq\omega\, \exists\alpha,f,C\,\psi_{\mathrm{Moore}}(r,\alpha,f,C,x,y,z)}.
\]

Then the following holds:
\begin{enumerate}
\item \label{cor:2omega>omega2-0}
Assume $S\supseteq\ZFC$ and $A\in\SpecAMC{S}$ are such that
\[
\exists x\, \psi_{\aleph_1}(x), \theta_{\mathrm{Moore}}
\]
are both in $\bool{AMC}(S,A)$. 
Let $\mathcal{M}$ be a model of $\bool{AMC}(S,A)$ which satisfies the Replacement axiom schema; fix 
$c,a,\omega_1^{\mathcal{M}}\in\mathcal{M}$ so that 
$\psi_{\mathrm{countpartstat}}(a,\omega_1^{\mathcal{M}})$,
$\psi_{\mathrm{ladder}}(c,\omega_1^{\mathcal{M}})$ hold in $\mathcal{M}$.
Then 
 $\exists f,C\,\psi_{Moore}(x,y,f,C,\omega_1^{\mathcal{M}},c,a)$ defines  in $\mathcal{M}$ the graph of a surjection of $\omega_2^{\mathcal{M}}$ onto
 $\pow{\omega}^{\mathcal{M}}$.
 
\item \label{cor:2omega>omega2-1}
There is at least one (recursive set) $B\subseteq\bool{Form}_\in\times 2$ with $\in_B\supseteq\in_{\Delta_0}$ 
such that
for any $R$ extending 
\[
\ZFC+\emph{there are class many Woodin cardinals and a supercompact cardinal}, 
\]
$B\in\SpecAMC{R}$
and
\[
\exists x\, \psi_{\aleph_1}(x), \theta_{\mathrm{Moore}}
\]
are both in $\bool{AMC}(R,B)$ with the latter implying the Replacement schema.

\item \label{cor:2omega>omega2-2}
For any $R\supseteq\ZFC$ and $A\in\SpecAMC{R}$ if
\[
\exists x\, \psi_{\aleph_2}(x), \theta_{\mathrm{Moore}}
\]
are both in $\bool{AMC}(R,A)$ and $\bool{AMC}(R,A)$ implies the Replacement schema, then $\bool{AMC}(R,A)\models 2^{\aleph_0}\leq\aleph_2$.

\item \label{cor:2omega>omega2-3}
Assume $R\supseteq\ZFC$ and $A\in\SpecAMC{R}$  are such that $\bool{AMC}(R,A)$ implies the Replacement schema,
$(\psi_{\aleph_2}(x),1)\in A$,
$\in_A\supseteq \in_{\Delta_0}$, and
\[
\ZFC+(R+T_{\in,A})_{\forall\vee\exists}+\theta_{\mathrm{Moore}}
\]
is consistent.

Then $2^{\aleph_0}>\aleph_2\not \in \bool{AMC}(R,A)$.
\end{enumerate}
\end{corollary}

\begin{proof}
\emph{}

\begin{description}
\item[\ref{cor:2omega>omega2-0}]
See the last item and adjust the proof from it.

\item[\ref{cor:2omega>omega2-1}]
By Thm. \ref{Thm: mainthmforcibility}.

\item[\ref{cor:2omega>omega2-2}]
See the proof below and observe that the argument which below is 
given for just one $\mathcal{M}$ which models $\bool{AMC}(R,A)$, in this case can be repeated for all
models of $\bool{AMC}(R,A)$.

\item[\ref{cor:2omega>omega2-3}]
Since $\bool{AMC}(R,A)$ is the AMC of $R+T_{\in,A}$, we can find $\mathcal{M}\prec_1\mathcal{N}$ with
$\mathcal{N}$ an $\in_A$-model of 
\[
T_0=\ZFC+(R+T_{\in,A})_{\forall\vee\exists}+\theta_{\mathrm{Moore}}
\]
and $\mathcal{M}$ an $\in_A$-model of $\bool{AMC}(R,A)$.
Then 
\[
\mathcal{M}\models\theta_{\mathrm{Moore}}+\bool{ZFC}^-_A.
\]
Fix $a,b,c,A\in\mathcal{M}$ such that $\mathcal{M}$ models:
\begin{itemize}
\item $\psi_{\aleph_1}(a)$,
\item $\psi_{\aleph_2}(b)$,
\item $\psi_{\mathrm{ladder}}(c,a)$,
\item $\psi_{\mathrm{countpartstat}}(A,a)$.
\end{itemize}
Note that $\mathcal{N}$ models all of the above formulae with the same parameters, since $\mathcal{M}\prec_1\mathcal{N}$.

 Let $T$ be the $\in_A\cup\bp{a,b}$- theory 
 \[
 T_0+\psi_{\aleph_1}(a)+\psi_{\aleph_2}(b).
 \]
Note that $a,b$ are constant symbols interpreted according to the axioms $\bool{Ax}^1_{\psi_{\aleph_i}}$ for $i=0,1$, that $\mathcal{M}$ and $\mathcal{N}$ admit unique extensions to models $\mathcal{M}^*,\mathcal{N}^*$ of $T$, and that for such extensions it still holds that 
$\mathcal{M}^*\prec_1\mathcal{N}^*$.
 
This gives that
\begin{equation}\label{eqn:reflmoore}
\mathcal{M}^*\models\forall r\subseteq\omega\exists\alpha\in b\exists f,C\,\psi_{\mathrm{Moore}}(r,\alpha,f,C,a,c,A).
\end{equation}
 
By Lemma \ref{lem:repl-closure0}
Replacement for $\in_{A}\cup\bp{a,b}$-formulae holds in $\mathcal{M}^*$.
By Thm.  \ref{thm:moo06}\ref{thm:moo06-2}, $\mathcal{M}$ models
\begin{equation}\label{eqn:reflmoore-1}
\forall\alpha\in b \exists! r\subseteq\omega\qp{\exists f,C\,\psi_{\mathrm{Moore}}(r,\alpha,f,C,a,c,A)\vee 
(\neg \exists s\subseteq\omega,f,C\,\psi_{\mathrm{Moore}}(s,\alpha,f,C,a,c,A)\wedge r=\emptyset)}
\end{equation}
Applying it to the existential $\in_{A}\cup\bp{a,b}$-formula 
$\exists f,C\,\psi_{\mathrm{Moore}}(x,y,f,C,a,v,w)$ with $v,w$ replaced by parameters $c,A$, we get an 
$F\in\mathcal{M}$ which (by \ref{eqn:reflmoore} and \ref{eqn:reflmoore-1}) is a function with domain 
\[
\bp{\alpha\in\mathcal{M}:\,\mathcal{M}\models \alpha\in b}
\]
and range exactly given by 
\[
\bp{r\in\mathcal{M}:\,\mathcal{M}\models r\subseteq\omega}.
\]
Hence $\pow{\omega}$ according to $\mathcal{M}$ is in $\mathcal{M}$ the set $\ran(F)$, which is the surjective image (via $F$) of the second uncountable cardinal according to $\mathcal{M}$. Therefore $2^{\aleph_0}\leq\aleph_2$ holds in $\mathcal{M}$ as witnessed by $\mathcal{F}$.
\end{description}
\end{proof}

\begin{remark}
We expect to be able to prove for $\theta_{\mathrm{Moore}}$ and $2^{\aleph_0}>\aleph_2$ a result of a similar vein of Cor. \ref{cor:2omega>omega2} for any $A\subseteq \bool{Form}_\in\times 2$ such that 
$\in_A\supseteq\in_{\Delta_0}$ and replacing $\bool{AMC}(S,A)$ (which may not exist for many such $A$) with the strong consistency hull of $S+T_{\in,A}$. However the proof becomes rather intricated since we must check that the needed instances of replacement hold in the required models of $\bool{SCH}(S+T_{\in,A})$.
Most instances of replacement for $\Sigma_1$-formulae hold in such structures. One then should argue that all instances of replacement used in the above proof are among those which can follow from $\bool{SCH}(S+T_{\in,A})$.
\end{remark}
\section{Generic invariance results for signatures of second and third order arithmetic}\label{sec:geninv}

We collect here generic absoluteness results results needed to prove
Thm. \ref{Thm: mainthmforcibility}.
We prove all these results working in ``standard''
models of $\ZFC$, i.e. we assume the models are well-founded. This is a practice we already adopted
in Section \ref{sec:Hkappa+}. We leave to the reader to remove this unnecessary assumption.

\subsection{Universally Baire sets and generic absoluteness for second order number theory}\label{sec:genabssecordnumth}

We recall here the properties of universally Baire sets and the generic absoluteness results for second order 
number theory we need to prove Thm. \ref{Thm: mainthmforcibility}.

\subsubsection{Universally Baire sets}\label{subsec:univbaire}
Assuming large cardinals, 
there is a very large sample of projectively closed families of subsets of $\pow{\omega}$ which are are ``simple'', 
hence it is natural to consider elements of these families
as atomic predicates. 

The exact definition of what is meant by a ``simple'' subset of $2^\omega$ 
is captured by the notion of universally Baire set.

Given a topological space $(X,\tau)$, $A\subseteq X$ is nowhere dense if its closure 
has a dense complement,
meager if it is the countable union of nowhere dense sets, with the Baire property if it 
has meager symmetric difference with
an open set.
Recall that $(X,\tau)$ is Polish if $\tau$ is a completely metrizable, separable topology on $X$.

\begin{definition}
(Feng, Magidor, Woodin) 
Given a compact Polish space $(X,\tau)$, $A\subseteq X$ is \emph{universally Baire} 
if for every compact Hausdorff space $(Y,\sigma)$ and
every continuous $f:Y\to X$ we have that $f^{-1}[A]$ has the Baire property in $Y$.

$\bool{UB}$ denotes the family of universally Baire subsets of $X$ for some compact Polish space $X$.
\end{definition}

We adopt the convention that $\mathsf{UB}$ denotes the class of universally Baire sets and of all elements of 
$\bigcup_{n\in\omega+1}(2^{\omega})^n$ (since the singleton of such elements are universally Baire sets).



The list below outlines three simple examples of projectively closed families of universally Baire sets
containing $2^\omega$.
\begin{remark}[Woodin]\label{thm:UBsetsgenabs}
Let:
\begin{itemize}
\item
 $T_0$ be the $\in_{\Delta_0}$-theory $\mathsf{ZFC}_{\Delta_0}+$\emph{there are infinitely many 
Woodin cardinals and a measurable above};
\item
$T_1$ be the $\in_{\Delta_0}$-theory $\mathsf{ZFC}_{\Delta_0}+$\emph{there are class many Woodin cardinals};
\item
$T_2$ be the $\in_{\Delta_0}$-theory $\mathsf{ZFC}_{\Delta_0}+$\emph{there are class many Woodin cardinals which are a limit of Woodin cardinals}.
\end{itemize}
The following holds:
\begin{enumerate}[(A)]
\item \cite[Thm. 3.1.12, Thm. 3.1.19]{STATLARSON}
Assume $(V,\in_{\Delta_0})$ models $T_0$. Then every projective subset of $2^\omega$ is universally Baire.
\item \cite[Thm. 3.3.3, Thm. 3.3.5, Thm. 3.3.6, Thm. 3.3.8, Thm. 3.3.13, Thm. 3.3.14]{STATLARSON}
Assume $(V,\in_{\Delta_0})\models T_1$.
Then $\mathsf{UB}$ is projectively closed.
\item (Woodin, Unpublished)
Assume $(V,\in_{\Delta_0})\models T_2$. Then the family of subsets of $2^\omega$
which are definable in $L(\Ord^\omega)$ consists of universally Baire sets.
\end{enumerate}
\end{remark}

%

We now list some standard facts about universally Baire sets we will need:
 
 \begin{enumerate}[(i)]
 \item\label{itm1:charUBsets} \cite[Thm. 32.22]{JECHST}
 $A\subseteq 2^{\omega}$ is universally Baire if and only if for each forcing notion $P$ there are 
 trees $T_A,S_A$ on $2\times\delta$ for some $\delta> |P|$
 such that $A=p[[T_A]]$ (where $p:(2\times\delta)^\omega\to 2^{\omega}$ denotes the projection on the first component and $[T]$ denotes the body of the tree $T$), and
 \[
P\Vdash T_A\text{ and }S_A\text{ project to complements},
\]
by this meaning that for all $G$ $V$-generic for $P$
\[
V[G]\models (p[[T_A]]\cap p[[S_A]]=\emptyset)\wedge (p[[T_A]]\cup p[[S_A]]=(2^\omega)^{V[G]})
\]
\item 
Any two Polish spaces $X,Y$ of the same cardinality are Borel isomorphic \cite[Thm. 15.6]{kechris:descriptive}. 
\item The universally Baire sets are a pointclass closed under preimages by Borel functions (by \cite[Thm. 13.1]{kechris:descriptive})
and (assuming the existence of class many Woodin cardinals) also under images by Borel functions (for example by
\cite[Thm. 3.3.7]{STATLARSON}).
 \item
Given $\phi:\mathbb{N}\to\mathbb{N}$, $\prod_{n\in\omega}2^{\phi(n)}$ is compact and Polish
(it is actually homemomorphic to the union of $2^\omega$ with a countable Hausdorff space) 
\cite[Thm. 6.4, Thm. 7.4]{kechris:descriptive}.
\end{enumerate}

Hence it is not restrictive (assuming large cardinals) to focus just on universally Baire subsets of 
$2^\omega$ and of its
countable products, which is what we will do in the sequel.

\begin{notation}\label{not:notUBsetsVG}
Given $G$ a $V$-generic filter for some forcing $P\in V$, $A\in \UB^{V[G]}$ and
$H$ $V[G]$-generic filter for some forcing $Q\in V[G]$, 
\[
A^{V[G][H]}=\bp{r\in (2^\omega)^{V[G][H]}: V[G][H]\models r\in p[[T_A]]},
\]
where $(T_A,S_A)\in V[G]$ is any pair of trees as given in item \ref{itm1:charUBsets} above
such that $p[[T_A]]=A$ holds in $V[G]$,
and $(T_A,S_A)$ project to complements in $V[G][H]$.
\end{notation}

\subsection{Generic absoluteness for second order number theory}\label{subsec:genabssecnumth}




The generic absoluteness results we need in this section
follow readily from \cite[Thm. 3.1.2]{STATLARSON} and the assumption that there exists
class many Woodin limits of Woodin. The Theorem below reduces these large cardinal assumptions to the existence of class many Woodin cardinals.

\begin{theorem}\label{thm:genabshomega1}
Assume in $V$ there are class many Woodin cardinals. 
Let $\mathcal{A}\in V$ be a 
family of universally Baire sets of $V$, and $G$ be $V$-generic for some forcing notion $P\in V$.

Then 
\[
(H_{\omega_1},\in, A:A\in\mathcal{A})\prec(H_{\omega_1}^{V[G]},\in,A^{V[G]}:A\in\mathcal{A}).
\]
\end{theorem}

The reader can find a proof in the author's webpage, it is
an improvement of  \cite[Thm. 3.1]{VIAMMREV}.

It is now convenient to reformulate projective closure in a semantic way which is handy when dealing with a fixed complete first order axiomatization of set theory.

\begin{definition}\label{def:Homega1closed}
Let $\mathcal{A}\subseteq\bigcup_{n\in\omega}\pow{\omega}^n$.
$\mathcal{A}$ is $H_{\omega_1}$-closed if 
any definable subset of $\pow{\omega}^n$ for some $n\in\omega$ in the structure
\[
(H_{\omega_1},\in,U:U\in\mathcal{A})
\]
is in $\mathcal{A}$.

Given a family $\mathcal{X}\subseteq \bigcup_{n\in\omega}\pow{\omega}^n$ 
its projective closure $\bool{PC}(\mathcal{X})$ is the smallest $H_{\omega_1}$-closed family of subsets
of $\bigcup_{n\in\omega}\pow{\omega}^n$ containing $\mathcal{X}$.
\end{definition}
It is immediate to check that if $T$ is the theory of $(V,\in)$ and $\mathcal{A}$ is a family of universally Baire subsets of $V$, $\mathcal{A}$ is projectively closed for $T$  for the signature $\in_{\Delta_0}\cup\mathcal{A}$ (according to Def. \ref{def:projclosureaxioms}) if and only if it
is $H_{\omega_1}$-closed.

We get the following:
\begin{corollary}\label{fac:keyfacHomega1clos}
Assume $(V,\in)$ models $\ZFC+$\emph{there are class many Woodin cardinals}.
Let $\mathcal{A}\subseteq \UB^V$ be $H_{\omega_1}$-closed and $\in_\mathcal{A}=\in_{\Delta_0}\cup\mathcal{A}$
 be the signature in which each element of $\mathcal{A}$ contained in $\pow{\omega}^k$ is a predicate symbol of arity $k$.
 Then for any $G$ $V$-generic for some forcing $P\in V$
 the $\in_\mathcal{A}$-theory of $H_{\omega_1}^V$ is the AMC of the  
 $\in_\mathcal{A}$-theory of $V[G]$ and $\bp{A^{V[G]}:A\in\mathcal{A}}$
 is $H_{\omega_1}^{V[G]}$-closed.
\end{corollary}

\begin{proof}
The assumptions grant that
\[
(H_{\omega_1}^{V},\in_{\Delta_0}^{V},A:A\in\mathcal{A})\prec
(H_{\omega_1}^{V[G]},\in_{\Delta_0}^{V[G]},A^{V[G]}:A\in\mathcal{A})
\prec_1 (V[G],\in_{\Delta_0}^{V[G]},A:A^{V[G]}\in\mathcal{A})
\]
(by Thm. \ref{thm:genabshomega1} and by Lemma \ref{lem:levabsgen} applied in $V[G]$).
Now the theory of $H_{\omega_1}^{V}$ in signature $\in_{\mathcal{A}}$ is complete and model complete,
and is also the $\in_{\mathcal{A}}$-theory of $H_{\omega_1}^{V[G]}$.
We conclude that 
it is the AMC of the $\in_{\mathcal{A}}$-theory of $V[G]$.
It is also easy to check that
$\bp{A^{V[G]}:A\in\mathcal{A}}$
 is $H_{\omega_1}^{V[G]}$-closed.
\end{proof}

\subsection{Generic invariance of the universal fragment of the theory of $V$ with predicates for the non-stationary ideal and for universally Baire sets}\label{subsec:geninvtoa}

The results of this section are the key to establish
Thm. \ref{Thm: mainthmforcibility}.
The proofs require some familiarity with the basics of the $\Pmax$-technology and 
with Woodin's stationary tower forcing.

\begin{notation}\label{not:keynotation-genabsthordar}
\emph{}

\begin{itemize}
\item
$\in_{\NS_{\omega_1}}$ is the signature $\in_{\Delta_0}\cup\bp{\omega_1}\cup\bp{\NS_{\omega_1}}$ with $\omega_1$ 
a constant symbol, $\NS_{\omega_1}$ a unary predicate symbol.

%
\item
$T_{\NS_{\omega_1}}$ is the $\in_{\NS_{\omega_1}}$-theory
given by $T_{\Delta_0}$ together with the axioms
\[
\omega_1\text{ is the first uncountable cardinal},
\]
\[
\forall x\;[(x\subseteq\omega_1\text{ is non-stationary})\leftrightarrow\NS_{\omega_1}(x)].
\]

\item
$\ZFC^-_{\NS_{\omega_1}}$ is the $\in_{\NS_{\omega_1}}$-theory 
\[
\ZFC^-_{\Delta_0}+T_{\NS_{\omega_1}}.
\]
\item
Accordingly we define $\ZFC_{\NS_{\omega_1}}$.
\end{itemize}
\end{notation}
Clearly the above axioms are of the form $\text{Ax}^0_\phi$ for $\phi$ the formula defining the non-stationary ideal on $\omega_1$, and $\text{Ax}^1_\phi$ for $\phi$ the formula defining the first uncountable cardinal. Furthermore the above axioms are $\Pi_2$-sentences of the relevant signature.



\begin{Theorem}\label{thm:PI1invomega2}
Assume $(V,\in)$ models $\ZFC+$\emph{ there are class many Woodin cardinals}.
Then the $\Pi_1$-theory of $V$ for the language 
\(
\in_{\Delta_1}\cup\in_{\NS_{\omega_1}}\cup\UB^V
\)
is invariant under set sized forcings\footnote{Here we consider any $A\subseteq (2^\omega)^k$ in $\UB^V$ as a predicate symbol of arity $k$, and we clearly assume that the symbols of $\in_{\Delta_1}$ are interpreted according to $\ZFC^-_{\Delta_1}+T_{\NS_{\omega_1}}$ in $V$ and in its generic extensions.}.

\end{Theorem}

Asper\'o and Veli\v{c}kovi\`c provided the following basic counterexample to the conclusion of the Theorem
if large cardinal assumptions are dropped.
\begin{remark}\label{rem:limbobangeninv}
Let $\phi(y)$ be the $\Sigma_1$-property in $\in_{\NS_{\omega_1}}$
\[
\exists y (y=\omega_1 \wedge L_{y+1}\models y=\omega_1).
\]
Then $L$ models this property, while the property fails in any forcing extension of $L$ which collapses 
$\omega_1^L$ 
to become countable.
\end{remark}

\begin{remark}
Note that for any $T$ extending $\ZFC+$\emph{there are class many Woodin cardinals} and any signature
$\tau$ extending $\in$ by predicates and function symbols as prescribed by Thm. \ref{thm:PI1invomega2}, we obtain that a $\Pi_2$-sentence $\psi$ for $\tau$ such that
$\psi^{H_{\aleph_2}}$ is forcible over any model of $T$ is in the strong consistency hull of $T$, hence it is realized in any $T$-ec model.
In particular the strong consistency hull of $\ZFC_{\NS_{\omega_1}}+$\emph{there are class many Woodin cardinals} in signature
$\in_{\NS_{\omega_1}}$ contains $\neg\CH$ and the definable version of $2^{\aleph_0}=\aleph_2$ implied by $\theta_{\mathrm{Moore}}$. We instantiated in detail this argument in Section \ref{subsec:AMCnegCH} for the specific $\Pi_2$-sentences
for $\in_{\NS_{\omega_1}}$ which express $\neg\CH$ and for $\theta_{\text{Moore}}$. Mutatis mutandis the same argument work for any $\Pi_2$-sentence expressed in a signature to which 
Proposition \ref{prop:geninv} applies.
\end{remark}
%
%

In order to prove Thm. \ref{thm:PI1invomega2} we need to recall some basic terminology and facts about iterations of countable structures.
\subsubsection{Generic iterations of countable structures}

Recall that for a transitive model $(M,\in)$ of a suitable fragment of $\ZFC$, $j:M\to N$ is the ultrapower embedding with respect to the non-stationary ideal on $\omega_1$ if
\begin{itemize}
\item
$N$ is the transitive collapse of the structure 
\[
(\bp{[f]_G:\, f\in M,\, f:\omega_1^M\to M\text{ is a function}},\in_G)
\]
where:
\begin{itemize} 
\item
$G$ is an $M$-generic filter for the forcing $(\pow{\omega_1}\setminus\NS)^M$;
\item
$f\equiv_G h$ if $\bp{\alpha\in\omega_1^M:h(\alpha)=f(\alpha)}\in G$,
$[f]_G$ is the equivalence class of $f$ with respect to $\equiv_G$,
and $[f]_G\in [h]_G$ if $\bp{\alpha\in\omega_1^M:h(\alpha)\in f(\alpha)}\in G$;
\end{itemize}
\item
$j(a)=[c_a]_G$ where $c_a:\omega_1^M\to M$ is constant with value $a$ for any $a\in M$.
\end{itemize}

\begin{definition}\cite[Def. 1.2]{HSTLARSON}
Let $M$ be a transitive countable 
model of $\ZFC$. 
Let $\gamma$ be an ordinal less than or equal to $\omega_1$. 
An iteration $\mathcal{J}$ of $M$ of length $\gamma$ 
consists of transitive sets $\ap{M_\alpha:\,\alpha \leq\gamma}$, sets $\ap{G_\alpha:\,\alpha< \gamma}$ 
and a commuting family of elementary embeddings 
\[
\ap{j_{\alpha\beta}: M_\alpha\to M_\beta:\, \alpha\leq\beta\leq\gamma}
\]
such that:
\begin{itemize}
\item
$M_0 = M$,
\item
each $G_\alpha$ is an $M_\alpha$-generic filter for 
$(\pow{\omega_1}\setminus \NS_{\omega_1})^{M_\alpha}$,
\item
each $j_{\alpha\alpha}$ is the identity mapping,
\item
each $j_{\alpha\alpha+1}$ is the ultrapower embedding with respect to the non-stationary ideal on $\omega_1$ of $(M_\alpha,\in)$ induced by $G_\alpha$,
\item
for each limit ordinal $\beta\leq\gamma$,
$M_\beta$ is the transitive collapse of the direct limit of the system
$\bp{M_\alpha, j_{\alpha\delta} :\, \alpha\leq\delta<\beta}$, and for each $\alpha<\beta$, $j_{\alpha\beta}$ is the induced embedding.
\end{itemize}
\end{definition}

We adopt the convention to denote an iteration $\mathcal{J}$ just
by $\ap{j_{\alpha\beta}:\, \alpha\leq\beta\leq\gamma}$, we also stipulate that
if $X$ denotes the domain of $j_{0\alpha}$, $X_\alpha$ or $j_{0\alpha}(X)$ will denote
the domain of $j_{\alpha\beta}$ for any $\alpha\leq\beta\leq\gamma$.

\begin{definition}
Let $A$ be  a universally Baire sets of reals.
$M$ is $A$-iterable if:
\begin{enumerate}
\item $M$ is transitive and such that $H_{\omega_1}^M$ is countable.
\item 
$M\models\ZFC+\NS_{\omega_1}$\emph{ is precipitous}.
\item
Any iteration 
\[
\bp{j_{\alpha\beta}:\alpha\leq\beta\leq\gamma}
\] 
of $M$ is well founded and such that 
$A\cap M_\beta=j_{\alpha\beta}(A\cap M_0)$ for all $\beta\leq\gamma$.
\end{enumerate}
\end{definition}

\subsubsection{Proof of Theorem \ref{thm:PI1invomega2} }

\begin{proof}
In view of Facts \ref{fac:keyfatr0}, \ref{fac:keyfatr} it suffices to prove the Theorem for the signature
\(
\in_{\NS_{\omega_1}}\cup\UB^V.
\) 
Let $\phi$ be a $\Pi_1$-sentence for this signature.


Assume $\phi$ holds in $V$ but for some forcing notion $P$, $\phi$ fails in $V[h]$ with $h$ $V$-generic for $P$.
By forcing over $V[h]$ with the appropriate stationary set preserving (in $V[h]$) 
forcing notion (using a Woodin cardinal $\gamma$ of $V[h]$), we may assume that $V[h]$ is extended to a
generic extension $V[g]$ such that $V[g]$ models $\NS_{\omega_1}$ is 
saturated\footnote{A result of Shelah whose outline can be found in \cite[Chapter XVI]{SHEPRO}, or \cite{woodinBOOK}, or in an \href{https://ivv5hpp.uni-muenster.de/u/rds/sat_ideal_better_version.pdf}{handout} of Schindler available on his webpage.}.
Since $V[g]$ is an extension of $V[h]$ by a stationary set preserving forcing and there are in $V[h]$ class many Woodin cardinals, we get that
$V[h]\sqsubseteq V[g]$ with respect to the signature $\in_{\NS_{\omega_1}}\cup\UB^V$.

Since $\Sigma_1$-properties are upward absolute and $\neg\phi$ holds in $V[h]$, 
$\phi$ fails in $V[g]$ as well.

Let $\delta$ be inaccessible in $V[g]$ and let $\gamma>\delta$ be a Woodin cardinal.

Let $G$ be $V$-generic for $\tow{T}^{\omega_1}_\gamma$ (the countable tower $\mathbb{Q}_{<\gamma}$ according to \cite[Section 2.7]{STATLARSON})
and such that  $g\in V[G]$.
Let $j_G:V\to\Ult(V,G)$ be the induced ultrapower embedding.

Now remark that $V_\delta[g]\in \Ult(V,G)$ is $B^{V[G]}$-iterable in $\Ult(V,G)$ for all 
$B\in \mathsf{UB}^{V}$ (since 
$V_\eta[g]\in \Ult(V,G)$ for all $\eta<\gamma$, and this suffices to check that $V_\delta[g]$ 
is $B^{V[G]}$-iterable for all $B\in\mathsf{UB}^V$, see \cite[Thm. 4.10]{HSTLARSON}).

By \cite[Lemma 2.8]{HSTLARSON} applied in $\Ult(V,G)$, there exists in $\Ult(V,G)$ an iteration 
$\mathcal{J}=\bp{j_{\alpha\beta}:\alpha\leq\beta\leq\gamma=\omega_1^{\Ult(V,G)}}$ of 
$V_\delta[g]$ such that
$\NS_{\omega_1}^{X_{\gamma}}=\NS_{\omega_1}^{\Ult(V,G)}\cap X_{\gamma}$, where 
$X_\alpha=j_{0\alpha}(V_\delta[g])$ for all $\alpha\leq\gamma=\omega_1^{\Ult(V,G)}$.

This gives that $X_{\gamma}\sqsubseteq \Ult(V,G)$ for 
\(
\in_{\NS_{\omega_1}}\cup\UB^V.
\)
Since $V_\delta[g]\models\neg\phi$, so does $X_{\gamma}$, by elementarity.
But $\neg\phi$ is a $\Sigma_1$-sentence, hence it is upward absolute for superstructures, therefore
$\Ult(V,G)\models\neg\phi$. This is a contradiction, since $\Ult(V,G)$ is elementarily equivalent to $V$ for
$\in_{\NS_{\omega_1}}\cup\UB^V$, and $V\models\phi$.

\smallskip

A similar argument shows that if $V$ models a $\Sigma_1$-sentence $\phi$ for 
$\in_{\NS_{\omega_1}}\cup\UB^V$
this will remain true in all of its generic extensions:

Assume $V[h]\models\neg\phi$
for some $h$ $V$-generic for some forcing notion $P\in V$. 
Let $\gamma>|P|$ be a Woodin cardinal, and let $g$ be $V$-generic for\footnote{$\tow{T}_\gamma$ is the full stationary tower of height $\gamma$ whose conditions are 
stationary sets in $V_\gamma$, denoted as $\mathbb{P}_{<\gamma}$ in \cite{STATLARSON},
see in particular \cite[Section 2.5]{STATLARSON}.}
 $\tow{T}_\gamma$ with $h\in V[g]$ and $\crit(j_g)=\omega_1^V$ (hence there is in $g$ some stationary set of $V_\gamma$ concentrating on countable sets). 
Then $V[g]\models\phi$ since:
\begin{itemize} 
\item
$V_\gamma\models\phi$, since $V_\gamma\prec_{1} V$ for $\in_{\NS_{\omega_1}}\cup\UB^V$
by Lemma \ref{lem:levabsgen};
\item
$V_\gamma^{\Ult(V,g)}=V_\gamma^{V[g]}$, since $V[g]$ models that 
$\Ult(V,g)^{<\gamma}\subseteq \Ult(V,g)$;
\item
$V_\gamma^{\Ult(V,g)}\models\phi$, by elementarity of $j_g$, since 
$j_g(V_\gamma)=V_\gamma^{\Ult(V,g)}$;
\item
$V_\gamma^{V[g]}\prec_{1}V[g]$ with respect to $\in_{\NS_{\omega_1}}\cup\UB^V$
again by Lemma \ref{lem:levabsgen} applied in $V[g]$.
\end{itemize}

Now repeat the same argument as before for the $\Pi_1$-property $\neg\phi$,
with $V[h]$ in the place of $V$ and $V[g]$ in the place of $V[h]$. 
\end{proof}

%
%


\section{Model companionship results for the theory of $H_{\aleph_2}$ and axiom $(*)$}\label{sec:Homega2}

Let $\UB$ denote the family of universally Baire sets; for $\mathcal{A}\subseteq\UB$ 
$L(\mathcal{A})$ denotes 
the smallest transitive model of $\ZF$ which contains $\mathcal{A}$ (see for details 
Section \ref{subsec:univbaire}).

We will be interested in (what we will call \emph{generically tame}) families $\mathcal{A}$ of universally Baire sets which are
constructibly closed (i.e. $\pow{2^\omega}^{L(\mathcal{A})}=\mathcal{A}$), countably closed (e.g.
$\mathcal{A}^\omega\subseteq L(\mathcal{A})$), and generically invariant (e.g. the theory of $L(\mathcal{A})$ cannot be changed by set sized forcing). One example will be given by
$\mathcal{A}=\pow{2^\omega}^{L(\Ord^\omega)}$ assuming large cardinals\footnote{This remarkable result of Woodin is to my knowledge unpublished. There are some handwritten notes of a proof sketch by Larson, and it is mentioned in \cite[Remark 3.3.12]{STATLARSON}.}, another by the class $\UB$ itself as computed in any generic extension of $V$ collapsing a supercompact to countable.

We aim to prove two model companionship results relating the theory of $V$ to that of 
$H_{\omega_2}$. First of all we must include a certain finite and explicit set of properties and definable functions  which are $\Delta_1(\ZFC^-)$ to $\in_{\Delta_0}$ (in order to be able to express by means of quantifier free formulae certain absolute concepts of set theory). We may call this extended signature $\in_{\Delta_1}$.
This is harmless in view of Fact \ref{fac:keyfatr}, since we will only consider $\in$-structures which are models of $\ZFC^-$.
We will then consider signatures extending $\in_{\Delta_1}\cup\bp{\NS_{\omega_1},\omega_1}$ with predicate symbols for certain elements of a generically tame $\mathcal{A}$: 
one with predicate symbols for \emph{all} elements of 
$\mathcal{A}$, the other just for the \emph{lightface definable} elements of $\mathcal{A}$.
\begin{itemize}
\item
The first result establishes the equivalence between the conditional version of Woodin's axiom $(*)$ to the model $L(\mathcal{A})$
and the assertion that the theory of $V$ is the model companion of the theory of $H_{\omega_2}$
in signature\footnote{We consider an element $A\subseteq (2^\omega)^k$ of $\mathcal{A}$ as a $k$-ary predicate symbol, $\NS_{\omega_1}$ as a unary predicate symbol, $\omega_1$ as a constant symbol.} $\in_{\Delta_1}\cup\mathcal{A}\cup\bp{\NS_{\omega_1},\omega_1}$. 

This result however makes sense to a platonist but not to a formalist since it subsumes the existence of $V$ and it is stated for a signature which  is highly non constructive (it  includes predicate symbols for the ---at least continuum many--- elements of $\mathcal{A}$ in $V$).
\item
The second result is a model companionship result for $\ZFC+$\emph{large cardinals} with respect to the signature $\in_{\NS_{\omega_1},\mathcal{A}}$ extending $\in_{\Delta_1}\cup\bp{\NS_{\omega_1},\omega_1}$ only with predicate symbols for those elements of $\mathcal{A}$ which are provably the extension of an $\in$-formula $\phi(x)$ in no parameters. This is a recursive signature and we can give a recursive axiomatization of the model companion theory.
This model companion theory is the common chore of the
$\in_{\NS_{\omega_1},\mathcal{A}}$-theory of $H_{\omega_2}^V$ as $(V,\in)$ ranges over the models of $\MM^{++}+$\emph{suitable large cardinal axioms}; moreover the $\Pi_1$-fragment of this theory is forcing invariant.
\end{itemize}

We will actually need a generic invariance property for $\mathcal{A}$ which is delicate to formulate as $\mathcal{A}^V$ and $\mathcal{A}^{V[G]}$ share their defining $\in$-formula but, in general, have a trival intersection (since for most $V$-generic filters $G$ on the one hand there could be completely new universally Baire existing in $\mathcal{A}^{V[G]}\setminus\mathcal{A}^V$ --- even if $G$ does not adds new reals, on the other hand
 $A^{V[G]}\neq A^V$ for any uncountable universally Baire set $A\in\mathcal{A}^V$ --- in case $G$ adds a new real). These difficulties lead us to formulate the above properties by means of the  somewhat convoluted syntactic definitions given below.

\begin{Notation}
Let $(V,\in)$ be a model of $\ZFC$.

A family $\mathcal{A}$ of subsets of\footnote{For the remainder of the paper to avoid heavy notation we feel free to identify when needed $2^\omega$ with $\pow{\omega}$, $\pow{\omega^{<\omega}}$ 
or any variation of these sets which is clearly a canonical representation of the Cantor space.} $\pow{2^\omega}$ is:
\begin{itemize}
\item
\emph{constructibly closed} if 
\[
\pow{2^\omega}^{L(\mathcal{A})}=\mathcal{A},
\]
and every binary relation $R$ on $2^\omega\times2^\omega$ in $\mathcal{A}$ can be uniformized by a function $f:2^\omega\to2^\omega$ in $\mathcal{A}$.
\item \emph{countably closed} if $\mathcal{A}^\omega\subseteq L(\mathcal{A})$,
\item
\emph{lightface definable} if it is the extension in $V$ of some $\in$-formula $\phi_{\mathcal{A}}(x)$ without parameters.
\end{itemize}
\end{Notation}

%

Woodin calls slight variants of $\maxA$ as defined below ``Sealing'' for $\mathcal{A}$.

\begin{Definition}\label{Keyprop:maxUB}
$\maxA$ is the conjuction of the following three first order $\in$-sentences construed from some
 $\in$-formula $\phi_{\mathcal{A}}(x)$ in one free variable (together with the axioms of set theory needed to make sense of them\footnote{To make sense of $\maxA$ it suffices that $(V,\in)$ satisfies all axioms of $\ZFC$ with the exception of replacement and comprehension, with the latters replaced by the $\Sigma_n$-replacement schema for some large enough $n$. It is well known that this theory is finitely axiomatizable for each $n\in\omega$.}):
\begin{enumerate}
\item\label{Keyprop:maxUB-1}
 The extension of  $\phi_{\mathcal{A}}(x)$ in any generic extension $V[G]$ is 
 a constructibly closed family of universally Baire sets of $V[G]$. 
 \item\label{Keyprop:maxUB-1.5}
 As above but replacing \emph{constructibly closed} with \emph{countably closed}. 
 \item\label{Keyprop:maxUB-2}
Whenever 
$G$ is $V$-generic for some forcing notion $P\in V$ and $H$ is $V[G]$-generic for some forcing notion
$Q\in V[G]$ there are class of ordinals $I_{V[G]},I_{V[G][H]}$ such that any order preserving map of $I_{V[G]}$ into $I_{V[G][H]}$,
combined with the map $A^{V[G]}\mapsto A^{V[G][H]}$ extends uniquely to an elementary embedding of
$(L(\mathcal{A}^{V[G]}),\in)$ into 
$(L(\mathcal{A}^{V[G][H]}),\in)$ which is the identity on $H_{\omega_1}^{V[G]}$. 
\end{enumerate}

\begin{itemize}
\item
An $\in$-formula $\phi_{\mathcal{A}}(x)$ in one free variable 
is \emph{generically tame for $T$} if $T\models\maxA$.
\item
$\mathcal{A}\in V$ is \emph{generically tame for $T$} 
if $\mathcal{A}$ is the extension of some $\in$-formula
$\phi_{\mathcal{A}}(x)$ in $(V,\in)$ such that:
\begin{itemize}
\item  $T\models\maxA$,
\item
$(V,\in)\models T$.
\end{itemize}
\item
$\mathcal{A}\in V$ is \emph{generically tame} if it is generically tame for the theory of $(V,\in)$.
\end{itemize}
\end{Definition}

Note that if $(V,\in)\models T$, $T\models\maxA$ and $G$ is $V$-generic for some forcing $P\in V$, we may have that $V[G]\not\models T$ nonetheless $V[G]\models\maxA$ holds for sure.
 In particular $\maxA$ is a consequence of $T$ which is preserved by any forcing over any model of $T$, while other axioms of $T$ may not (for example $T\supseteq\ZFC$ might have either one of
  $\CH$ or $\neg\CH$ among its axioms, and neither of them is preserved through forcing extensions over models of $T$).

 
%
%

%
By \cite[Remark 3.3.12, Thm. 3.3.14, Thm. 3.3.19, Thm 3.4.18]{STATLARSON} and \cite[Thm. 36.9]{kechris:descriptive} we get the following:

\begin{enumerate}[(i)]
\item
Assume $T_0$ extends $\ZFC+$\emph{there are class many Woodin cardinals which are a limit of Woodin cardinals}. 

Then the $\in$-formula $\phi_0(x)$ defining\footnote{The \emph{Chang model} $L(\Ord^\omega)$ is the smallest transitive model of $\ZF$ containing all the countable sequences of ordinals.} $\pow{2^\omega}^{L(\Ord^\omega)}$ defines a generically tame family for $T_0$; e.g. $T_0\models\maxA$ for $\mathcal{A}=\pow{2^\omega}^{L(\Ord^\omega)}$.
\item
Assume $T_1$ extends $T_0$ with the axiom\footnote{This is first order expressible by \cite{laver}.} 
\begin{quote} 
\emph{Any model $(V,\in)$ of $T_1$ is a generic extension of some inner model $(W,\in)$ where some countable ordinal $\delta$ of $V$ is a supercompact cardinal in
$W$.}
\end{quote}
Then  the universally Baire sets 
are a generically tame family for $T_1$; e.g. $T_1\models\maxUB$ for $\phi_\UB(x)$ an $\in$-formula defining the universally Baire sets.
\end{enumerate}


Let us now be precise in our definition of the relevant signatures of this Section.

\begin{Notation}\label{not:keysignaturesforinB}

Given an $\in$-formula $\phi_\mathcal{A}(x)$:
\begin{itemize}
\item
$\in_{\Delta_0,\mathcal{A}}$ is the extension of $\in_{\Delta_0}$ in which
we add an $n$-ary relation symbol $S_\phi$ for any $\in$-formula $\phi$ of arity $n$;
\item
$T_{\Delta_0,\mathcal{A}}$ is the extension of $T_{\Delta_0}$ by the axioms
\[
\forall\vec{x} \,\qp{S_\phi(\vec{x})\leftrightarrow (\phi^{L(\mathcal{A})}(\vec{x})\wedge \bigwedge_{i=1}^n x_i\in \pow{\omega}) }
\]
as $\phi$ ranges over the $\in$-formulae of some arity $n$;

\item
$\in_{\NS_{\omega_1},\mathcal{A}}=
\in_{\Delta_0,\mathcal{A}}\cup\bp{\NS_{\omega_1},\omega_1}$;

\item
$T_{\NS_{\omega_1},\mathcal{A}}=T_{\Delta_0,\mathcal{A}}+T_{\NS_{\omega_1}}$ (recall Notation \ref{not:keynotation-genabsthordar});
$\in_{\Delta_1,\mathcal{A}}=
\in_{\Delta_0,\mathcal{A}}\cup\in_{\Delta_1}$;
\item
$T_{\Delta_1,\mathcal{A}}=T_{\Delta_0,\mathcal{A}}+T_{\Delta_1}$.
\item
$\ZFC^-_{\Delta_i,\mathcal{A}}$, $\ZFC_{\NS_{\omega_1},\mathcal{A}}$,\dots are all defined as expected.
\end{itemize}
\end{Notation}

$T_{\Delta_0,\mathcal{A}}$ makes the subsets of $\pow{\omega}$ definable in no parameters in 
$(L(\mathcal{A}),\in)$ the extension of an atomic formula of 
$\in_{\Delta_0,\mathcal{A}}$.

\begin{Remark}
It is clear that 
$\in_{\Delta_0,\mathcal{A}},\in_{\NS_{\omega_1},\mathcal{A}}$ 
can all be expressed as $\in_E$ for suitably chosen recursive sets
$E\subseteq\bool{Form}_\in\times 2$, while $\in_{\Delta_1,\mathcal{A}}$ and 
$T_{\Delta_1,\mathcal{A}}$ are semi-recursive sets. Nonetheless in view of Facts \ref{fac:keyfatr0}, \ref{fac:keyfatr} it is actually irrelevant in our set up if we work with an $\in_{\NS_{\omega_1},\mathcal{A}}$ model of  $\ZFC^-_{\NS_{\omega_1},\mathcal{A}}$ or with its unique expansion to a 
$\in_{\Delta_1}\cup\in_{\NS_{\omega_1},\mathcal{A}}$-model of $\ZFC^-_{\Delta_1}$. More precisely, we will feel free to be somewhat sloppy in stating some results about these signatures, by stating them with respect to $\in_{\NS_{\omega_1},\mathcal{A}}$ and proving them freely using $\in_{\Delta_1}\cup\in_{\NS_{\omega_1},\mathcal{A}}$. Facts \ref{fac:keyfatr0}, \ref{fac:keyfatr} grant that 
in all cases where we do it, this is feasible and correct.
\end{Remark}






Key to all results of this section is an analysis of the properties of
generic extensions by $\Pmax$ of $L(\mathcal{A})$ for $\mathcal{A}$ generically tame.
Generic tameness of $\mathcal{A}$ is used to argue (among other things)
that most of the results established 
in \cite{HSTLARSON}
on the properties of $\Pmax$ for $L(\mathbb{R})$ can be
also asserted for $L(\mathcal{A})$.
We  refrain to define the $\Pmax$-forcing rightaway and we will introduce it when needed in our proofs (see Def. \ref{def:Pmax}). 
Meanwhile we assume the reader is familiar with
 $\Pmax$ or can accept as a blackbox its existence as a certain forcing notion; our reference on this topic 
is \cite{HSTLARSON}.

We now give a precise definition of $\stA$ (modulo the definition of $\Pmax$).

\begin{Definition}\label{def:stA}
$\stA$ holds in some $\ZFC$-model $(V,\in)$ if in $V$:

\begin{itemize}
\item
there are class many Woodin cardinals;
\item
$\maxA$ holds for some $\in$-formula $\phi_\mathcal{A}(x)$;
\item
$\NS_{\omega_1}$ is precipitous\footnote{See \cite[Section 1.6, pag. 41]{STATLARSON}  for a definition of precipitousness and a discussion of its properties.};
\item
there exists a filter $G$ on $\Pmax$ meeting all the dense subsets of $\Pmax$ definable in 
$L(\mathcal{A})$ (where $\mathcal{A}\in V$ is the extension of the formula $\phi_\mathcal{A}(x)$).
\end{itemize}

$\stUB$ holds if $\stA$ holds for $\mathcal{A}=\UB$.
\end{Definition}
Woodin's axiom $(*)$ as defined in \cite[Def. 7.5]{HSTLARSON} is $\stA$ for $\mathcal{A}=\pow{2^\omega}^{L(\mathbb{R})}$ with $\maxA$ omitted.

This is the first main result of this section:
\begin{Theorem}\label{Thm:mainthm-1bis}
Assume $(V,\in)$ models
\[
\ZFC+\emph{there is a supercompact cardinal and class many Woodin cardinals}
\]
and $\mathcal{A}\in V$ is generically tame.

TFAE:
\begin{enumerate}
\item\label{thm:char(*)-modcomp-1}
$(V,\in)$ models $\stA$;
\item\label{thm:char(*)-modcomp-2}
$\NS_{\omega_1}$ is precipitous\footnote{A key observation is that $\NS_{\omega_1}$ being precipitous is independent of $\mathsf{CH}$ (see for example \cite[Thm. 1.6.24]{STATLARSON}), while $\stA$ entails $2^{\aleph_0}=\aleph_2$  (for example by the results of \cite[Section 6]{HSTLARSON}).

Another key point is that we stick to the formulation of $\Pmax$ as in \cite{HSTLARSON} so to 
be able in its proof to quote verbatim from \cite{HSTLARSON} all the relevant results on $\Pmax$-preconditions we will use.
It is however possible to  develop $\Pmax$ focusing on Woodin's countable tower rather than 
on the precipitousness of $\NS_{\omega_1}$ to define the notion of $\Pmax$-precondition. 
Following this approach in
all its scopes, one should be able to reformulate Thm. \ref{Thm:mainthm-1bis}(\ref{thm:char(*)-modcomp-2}) 
omitting the request that
$\NS_{\omega_1}$ is precipitous. We do not explore this venue any further.} in $V$ and
the $\in_{\NS_{\omega_1}}\cup\mathcal{A}^V$-theory\footnote{E.g. we regard $A\subseteq (2^\omega)^k$ as a $k$-ary predicate symbol for any $A\in\mathcal{A}$.} of $V$ has as model companion the
$\in_{\NS_{\omega_1}}\cup\mathcal{A}^V$-theory of $H_{\omega_2}$.
\end{enumerate}
\end{Theorem}

Thm. \ref{Thm:mainthm-1bis0} follows immediately from the above Theorem.

Note the difference between $\in_{\NS_{\omega_1},\mathcal{A}}$ which is a countable recursive signature, and $\in_{\NS_{\omega_1}}\cup\mathcal{A}^V$ for $(V,\in)$ a model of $\ZFC$. The latter has predicate symbols for the at least continuum (in the sense of $V$) many elements of $\mathcal{A}$ existing in $V$.


%

An objection to Thm. \ref{Thm:mainthm-1bis} is that it subsumes the Platonist standpoint that there exists a definite universe of sets.
We can prove a version of Thm. \ref{Thm:mainthm-1bis} which makes perfect sense also
to a formalist and from which we immediately derive Thm. \ref{Thm: mainthmforcibility}.
This is one of the reasons we paid attention to give a syntactically meaningful definition of generic tameness and of $\maxA$.

%
%
%

\begin{Theorem} \label{Thm:mainthm-1*}
Assume $T$ is an $\in$-theory extending\footnote{With $\maxA$ predicated for the $\in$-formula $\phi_\mathcal{A}(x)$.} 
\[
\ZFC+\maxA+\text{ there is a supercompact cardinal and  class many Woodin cardinals}.
\]
Let $B\subseteq \bool{Form}\times 2$ be such that $\in_B$ is exactly $\in_{\NS_{\omega_1},\mathcal{A}}$. Then $B\in\SpecAMC{T}$.

Moreover TFAE for any for any $\Pi_2$-sentence $\psi$ for 
$\in_{\NS_{\omega_1},\mathcal{A}}$:
\begin{enumerate}[(A)]
\item \label{Thm:mainthm-1A}
$\bool{AMC}(T,B)\vdash \psi$.
\item \label{Thm:mainthm-1E}
\(
(V[G],\in_{\NS_{\omega_1},\mathcal{A}}^{V[G]})\models\psi^{H_{\omega_2}}
\)
whenever $(V,\in)\models T$, $V[G]$ is a forcing extension of $V$, 
and $(V[G],\in)\models\stA$.
\item \label{Thm:mainthm-1Cbis}
$T$ proves\footnote{$\dot{H}_{\omega_2}$ denotes a canonical $P$-name for $H_{\omega_2}$ as computed in generic extension by $P$. $\Vdash_P\psi^{\dot{H}_{\omega_2}}$ stands for: 
\[
(V[G],\in_{\NS_{\omega_1},\mathcal{A}}^{V[G]})\models\psi^{H_{\omega_2}}
\]
whenever $G$ is $V$-generic for $P$.} 
\[
\exists P \,(P\text{ is a \emph{stationary set preserving} partial order }\wedge 
\Vdash_P\psi^{\dot{H}_{\omega_2}}).
\]
\item \label{Thm:mainthm-1C}
$T$ proves
\[
\exists P \,(P\text{ is a partial order }\wedge 
\Vdash_P\psi^{\dot{H}_{\omega_2}}).
\]
\item \label{Thm:mainthm-1D}
$T$ proves\footnote{$L(\mathcal{A})\models[\Pmax\Vdash \psi^{\dot{H}_{\omega_2}}]$ stands for: 
\[
(L(\mathcal{A})[G],\in_{\NS_{\omega_1},\mathcal{A}}^{L(\mathcal{A})[G]})\models\psi^{H_{\omega_2}}
\]
whenever $G$ is $L(\mathcal{A})$-generic for $P$.}
\[
L(\mathcal{A})\models[\Pmax\Vdash \psi^{\dot{H}_{\omega_2}}].
\]
\item \label{Thm:mainthm-1Dbis}
If $(V,\in)\models T$ 
and $\psi$ is $\forall x\exists y\,\phi(x,y)$ with $\phi$ quantifier free 
$\in_{\NS_{\omega_1},\mathcal{A}}$-formula, then for all\footnote{See \cite[Def. 1.8]{SCHASPBMM*++} 
for the notion of honest consistency. It can be equivalently stated as: \emph{For some $\kappa$ and $G$ $V$-generic for $\Coll(\omega,\kappa)$, there is a transitive set $M\in V[G]$ such that:}
\[
(H_{\omega_2}^V,\in_{\Delta_1}^V,\NS_{\omega_1}^V,\UB^V)\sqsubseteq (M,\in_{\Delta_1}^M,\NS_{\omega_1}^M,B^{V[G]}\cap M: B\in \UB^V),
\]
\[
(M,\in_{\Delta_1}^M,B^{V[G]}\cap M: B\in \UB^V)\sqsubseteq (V[G],\in_{\Delta_1}^{V[G]},B^{V[G]}: B\in \UB^V),
\]
\emph{and} 
\[
(M,\in_{\Delta_1}^M,\NS_{\omega_1}^M,B^{V[G]}\cap M: B\in \UB^V)\models \exists y\,\phi(A,y).
\]}  $A\in H_{\omega_2}^V$
\[
\exists y\,\phi(A,y)\text{ is \emph{honestly consistent} according to $V$}.
\]
\item \label{Thm:mainthm-1B}
For any consistent $\in_{\NS_{\omega_1},\mathcal{A}}$-theory 
\[
S\supseteq T+T_{\in,B},
\] 
$S_{\forall\vee\exists}+\psi$ is consistent.

\end{enumerate}

\end{Theorem}

Note that even if $T\models\CH$, $\neg\CH$ is in $\bool{AMC}(T,B)$ (for example by \ref{Thm:mainthm-1D} above).
In particular the model companion $\bool{AMC}(T,B)$ of $T$ describes a theory of $H_{\aleph_2}$ which can be 
completely unrelated to that given by models of $T$.
Moreover recall again that $\CH$ is not expressible as a boolean combination of $\Pi_1$-sentences in $\in_{\NS_{\omega_1},\mathcal{A}}$ for models of $T$: it is not preserved by forcing, while $T_{\forall\vee\exists}$ is.

The rest of this section is devoted to proof of Theorems~\ref{Thm:mainthm-1bis} and \ref{Thm:mainthm-1*}.
Crucial to their proof is the recent breakthrough of 
Asper\'o and Schindler \cite{ASPSCH(*)} establishing that$\stA$
follows from $\MM^{++}$  for any generically tame $\mathcal{A}$. 
First of all it is convenient to detail more on $\maxA$ and its use in our proofs.

\subsection{Some remarks on $\maxA$}

From now on we will need in several occasions that $\maxA$ holds in $V$ for a generically tame $\mathcal{A}$ (recall Def. \ref{Keyprop:maxUB}).
We will always explicitly state where this assumption is used, hence if a statement
 does not mention it in the hypothesis, 
the assumption is not needed for its thesis. 

 We will use all three properties of 
$\maxA$ crucially: (\ref{Keyprop:maxUB-1}) and (\ref{Keyprop:maxUB-1.5}) are used in the proof of Lemma \ref{lem:UBcorr};
(\ref{Keyprop:maxUB-2}) in the proof of Fact \ref{fac:densityUBcorrect}.
Similarly they are essentially used in Remark \ref{rmk:maxUBec}.
Specifically we will need  (\ref{Keyprop:maxUB-1}) and (\ref{Keyprop:maxUB-1.5}) of $\maxA$ to prove that 
certain countable families of subsets of $H_{\omega_1}$
simply definable using an existential $\in$-formula quantifying over $L(\mathcal{A})$ with parameters in
$H_{\omega_1}\cup\mathcal{A}$
are coded by a universally Baire set in $\mathcal{A}$, 
and (\ref{Keyprop:maxUB-2}) to prove that this coding is absolute between
generic extensions: i.e.  for any family 
$\bp{\phi_n:n\in\mathbb{N}}$ of
$\in$-formulae\footnote{Note that
the structures $(H_{\omega_1}\cup\mathcal{A},\in)$, $(H_{\omega_1}\cup\mathcal{A},\in_{\Delta_0}^V)$,
$(H_{\omega_1}\cup\mathcal{A},\in_{\Delta_1}^V)$ have the same algebra of definable sets, hence we will use one or the other as we deem most convenient,
since any set definable by some formula in one of these structures is also defined by a possibly different formula in the other. The formulation of $\maxA$ is unaffacted if we choose any of the two structures as the one for which we predicate it.} with parameters in $H_{\omega_1}^V\cup\mathcal{A}$, if 
\[
A_n=\bp{x\in H_{\omega_1}^V: (H_{\omega_1}\cup\mathcal{A},\in^V)\models \phi_n(x)}
\] 
and $\bp{A_n:n\in\mathbb{N}}$ is coded by $A\in \mathcal{A}$, letting
\[
A_n^{V[G]}=\bp{x\in H_{\omega_1}^{V[G]}: (H_{\omega_1}^{V[G]}\cup\mathcal{A}^{V[G]},\in^{V[G]})\models \phi_n(x)}
\] 
$\bp{A_n^{V[G]}:n\in\mathbb{N}}$ is coded by $A^{V[G]}\in \mathcal{A}^{V[G]}$.

It is useful to outline what is the different expressive power of the structures\\
 $(H_{\omega_1},\in_{\Delta_0}^V,A: A\in\mathcal{A}^V)$ and 
 $(H_{\omega_1}\cup\mathcal{A}^V,\in_{\Delta_0}^V)$.
 The latter can be seen as a second order extension of $H_{\omega_1}$, where
we also allow formulae to quantify over the family of universally Baire sets given by $\mathcal{A}$;
in the former quantifiers only range over elements of $H_{\omega_1}$, but we can 
use the subsets of $H_{\omega_1}$ whose univerally Baire code is in $\mathcal{A}$ as parameters.
This is in exact analogy between the comprehension scheme for the Morse-Kelley axiomatization of set theory
(where formulae with quantifiers ranging over classes are allowed) and the 
comprehension scheme for G\"odel-Bernays axiomatization of set theory
(where just formulae using classes as parameters and quantifiers ranging only over sets are allowed).
To appreciate the difference between the two set-up, 
note that that the axiom of determinacy for universally Baire sets in $\mathcal{A}$
is expressible in
\[
(H_{\omega_1}\cup\mathcal{A},\in_{\Delta_0}^V)
\]
by the $\Pi_2$-sentence for
$\in_{\Delta_0}$
\begin{quote}
\emph{For all $A\subseteq 2^\omega$ there is a winning strategy for one of the players in the game with payoff 
$A$},
\end{quote}
while in 
\[
(H_{\omega_1},\in_{\Delta_0}^V,A:A\in\mathcal{A}^V)
\]
it is expressed by the axiom schema of $\Sigma_2$-sentences for $\in_{\Delta_0}\cup\bp{A}$
\begin{quote}
\emph{There is a winning strategy for some player in the game with payoff 
$A$}
\end{quote}
as $A$ ranges over the universally Baire sets in $\mathcal{A}$.

We will crucially use the stronger expressive power of the structure 
$(H_{\omega_1}\cup\mathcal{A},\in_{\Delta_0})$
to define 
certain universally Baire sets as the extension in $(H_{\omega_1}\cup\mathcal{A},\in_{\Delta_0}^V)$ of lightface definable properties (according to the Levy hierarchy); properties which 
require an existential quantifier ranging over all universally Baire sets in $\mathcal{A}$.

\subsection{A streamline of the proofs} 
Let us give a general outline of the proofs before getting into details. From now on we assume 
the reader 
is familiar with the basic theory of $\Pmax$ as exposed in \cite{HSTLARSON}.

Much of our efforts will be now devoted to establish the model completeness of the $\in_{\NS_{\omega_1}}\cup\mathcal{A}$-theory of
 $(H_{\omega_2}^V,\in_{\Delta_1}^V,\NS_{\omega_1}^V,A:A\in\mathcal{A}^V)$  (assuming $\stA$ in $V$) and of the  $\in_{\NS_{\omega_1},\mathcal{A}}$-thery given by the family of models 
 $(H_{\omega_2}^V,\in_{\NS_{\omega_1},\mathcal{A}}^V)$ as $V$ ranges over  models of $\stA$.


We will then leverage on Levy absoluteness to infer that  in models of $\stA$ the theory of 
$(H_{\omega_2}^V,\in_{\Delta_1}^V,\NS_{\omega_1}^V,A:A\in\mathcal{A}^V)$ is the absolute model companion of the theory of \\
$(V,\in_{\Delta_1}^V,\NS_{\omega_1}^V,A:A\in\mathcal{A}^V)$.
A similar strategy will work for the AMC-result for the signature $\in_{\NS_{\omega_1},\mathcal{A}}$.

To prove these model completeness results we use Robinson's test and we show the following:

\begin{quote}
Assuming $\maxA$ there is a \emph{special} universally Baire set
$\bar{D}_{\NS_{\omega_1},\mathcal{A}}$ which belongs to $\mathcal{A}$ and  is defined by an $\in$-formula \emph{(in no parameters)} 
relativized to $L(\mathcal{A})$ (hence represented 
by a relation symbol of $\in_{\Delta_1,\mathcal{A}}$)
coding a family of $\Pmax$-preconditions with the following fundamental property:

\emph{
For any \emph{existential}
$\in_{\NS_{\omega_1},\mathcal{A}}\cup\bp{B_1,\dots,B_k}$-formula 
$\psi(x_1,\dots,x_n)$ mentioning the universally Baire predicates\footnote{Some of these $B_j$ may not appear in $\in_{\Delta_1,\mathcal{A}}$ being not definable by a lightface property.} $B_1,\dots,B_k\in\mathcal{A}$,
there is an algorithmic procedure which finds a \emph{universal} $\in_{\NS_{\omega_1},\mathcal{A}}\cup\bp{B_1,\dots,B_k}$-formula 
$\theta_\psi(x_1,\dots,x_n)$ mentioning just the universally Baire predicates 
$B_1,\dots,B_k,\bar{D}_{\NS_{\omega_1},\mathcal{A}}$ such that}
\[
(H_{\omega_2}^{L(\mathcal{A})[G]},\in_{\NS_{\omega_1},\mathcal{A}}^{L(\mathcal{A})[G]},B_1,\dots,B_k)
\models
\forall\vec{x}\,(\psi(x_1,\dots,x_n)\leftrightarrow\theta_\psi(x_1,\dots,x_n))
\]
\emph{whenever $G$ is $L(\mathcal{A})$-generic for $\Pmax$.}
\end{quote}
Moreover the definition and properties of $\bar{D}_{\NS_{\omega_1},\mathcal{A}}$ and the computation
of $\theta_\psi(x_1,\dots,x_n)$ from $\psi(x_1,\dots,x_n)$ are just based on the assumption that
$(V,\in)$ is a model of $\maxA$, hence can be replicated mutatis-mutandis in any model of 
$\ZFC+\maxA$. 
We will need that $(V,\in)$ is a model of $\maxA+\stA$ just to argue that
in $V$ there is an $L(\mathcal{A})$-generic filter $G$ for $\Pmax$ such that\footnote{It is this part of our argument
where the result of Asper\`o and Schindler establishing the consistency of $\stA$ relative to a supercompact is used in an essential way. We will address again the role of Asper\`o and Schindler's result in all our proofs in  some closing remarks.}
$H_{\omega_2}^{L(\mathcal{A})[G]}=H_{\omega_2}^V$. 
Since in all our arguments we will only use that $(V,\in)$ is a model of $\maxA$ and (in some of them 
also of $\stA$),
we will be in the position to conclude easily for the truth of Theorems~\ref{Thm:mainthm-1bis} and~ \ref{Thm:mainthm-1*}.

We condense the above information in the following:

\begin{theorem}\label{thm:keythmmodcompanHomega2}
Let $\phi_\mathcal{A}(x)$ be an $\in$-formula such that in some model $(V,\in)$ of $\ZFC$
\[
(V,\in)\models\maxA+\emph{there are class many Woodin cardinals}
\]
with $\mathcal{A}^V$ being the extension in $V$ of $\phi_\mathcal{A}(x)$.

Then there is an $\in$-formula $\phi_{\NS_{\omega_1},\mathcal{A}}(x)$ in one free variable\footnote{Whose canonical interpretation in models of $\maxA$ will  be the magic set $\bar{D}_{\NS_{\omega_1},\mathcal{A}}$. Note also that  $\phi_{\NS_{\omega_1},\mathcal{A}}(x)$ is computable in terms of $\phi_\mathcal{A}$.}
such that:
\begin{enumerate}
\item
Whenever $(V,\in)\models \ZFC+\maxA+$\emph{there are class many Woodin cardinals},
\[
\bp{z\in V:\, (V,\in)\models\phi_{\NS_{\omega_1},\mathcal{A}}(z)}
\] 
is the extension of some element of 
$\mathcal{A}^V$.
\item
Given predicate symbols $B_1,\dots,B_k$ of arity $n_1,\dots,n_k$, 
and the theory $T_{B_1,\dots,B_k}$ in signature 
$\in_{\NS_{\omega_1},\mathcal{A}}\cup\bp{B_1,\dots,B_k}$
extending
\[
\ZFC+\maxA+\emph{there are class many Woodin cardinals}
\] 
by the axioms\footnote{Axioms stating that $B_j$ is an element of $\mathcal{A}$
for all the new predicate symbols $B_1,\dots,B_k$.}:
\[
\exists ! y\subseteq (2^\omega)^{n_j}\, \qp{\forall z_1,\dots,z_{n_j}\, (\ap{z_1,\dots,z_{n_j}}\in y\leftrightarrow B_j(z_1,\dots,z_{n_j}))\wedge \phi_{\mathcal{A}}(y)},
\]

there is a recursive procedure assigning to any \emph{existential}
formula $\varphi(x_1,\dots,x_n)$ for $\in_{\NS_{\omega_1},\mathcal{A}}\cup\bp{B_1,\dots,B_k}$
a \emph{universal} formula $\theta_\varphi(x_1,\dots,x_n)$ for 
$\in_{\NS_{\omega_1},\mathcal{A}}\cup\bp{B_1,\dots,B_k}$ (mentioning just the predicate symbols occurring in $\varphi$ and $S_{\phi_{\NS_{\omega_1},\mathcal{A}}}$)
such that $T_{B_1,\dots,B_k}$ proves that\footnote{$\dot{G}\in L(\mathcal{A})$ is the canonical $\Pmax$-name for the generic filter.}
\[
\Pmax\Vdash 
[(H_{\omega_2}^{L(\mathcal{A})[\dot{G}]},\in_{\NS_{\omega_1},\mathcal{A}}^{L(\mathcal{A})[\dot{G}]}, \check{B}_1,\dots, \check{B}_k)
\models\forall\vec{x}\;(\varphi(x_1,\dots,x_n)\leftrightarrow\theta_\varphi(x_1,\dots,x_n))]
\]

\end{enumerate}

\end{theorem}

\subsection{Proofs of Thm.~ \ref{Thm:mainthm-1*}, and of 
(\ref{thm:char(*)-modcomp-1})$\to$(\ref{thm:char(*)-modcomp-2})
of Thm.~\ref{Thm:mainthm-1bis}}

\emph{}

Theorem~ \ref{Thm:mainthm-1*}, and (\ref{thm:char(*)-modcomp-1})$\to$(\ref{thm:char(*)-modcomp-2})
of Theorem~\ref{Thm:mainthm-1bis} are immediate corollaries
of the above theorem combined with:
\begin{itemize}
\item 
Asper\`o and Schindler's proof that 
$\MM^{++}+\maxA+$\emph{there are class many Woodin cardinals} implies $\stA$, 
\item 
Theorem \ref{thm:PI1invomega2}.
\end{itemize}

We start with the proof of
(\ref{thm:char(*)-modcomp-1})$\to$(\ref{thm:char(*)-modcomp-2}) of Thm.~\ref{Thm:mainthm-1bis}
assuming Thm. \ref{thm:keythmmodcompanHomega2} and Thm. \ref{thm:PI1invomega2}:

\begin{proof}
Assume $(V,\in)$ models $\stA$. Then there is a $\Pmax$-filter $G\in V$ such that
$H_{\omega_2}^{L(\mathcal{A})[G]}=H_{\omega_2}^V$.
By Thm. \ref{thm:keythmmodcompanHomega2} and Robinson's test, 
we get that the first order $\in_{\NS_{\omega_1}}\cup\mathcal{A}^V$-theory of 
$H_{\omega_2}^{L(\mathcal{A})[G]}$ is model complete.
By Levy's absoluteness (Lemma \ref{lem:levabsgen}), $H_{\omega_2}^{L(\mathcal{A})[G]}$ is a $\Sigma_1$-elementary substructure of $V$ also according to the signature $\in_{\NS_{\omega_1}}\cup\mathcal{A}$.
We conclude since the two theories are $\Pi_1$-complete and share the same $\Pi_1$ and $\Sigma_1$ fragments.
\end{proof}

The proof of the converse implication requires more information on 
$\bar{D}_{\NS_{\omega_1},\mathcal{A}}$ then what is conveyed in Thm. \ref{thm:keythmmodcompanHomega2}.
We defer it to a later stage.

\smallskip

We now prove Thm.~ \ref{Thm:mainthm-1*}:
\begin{proof}
Let $R$ be the theory given by the
$\Pi_2$-sentences $\psi$
for $\in_{\NS_{\omega_1},\mathcal{A}}$ which 
hold in every model of the form $(H_{\omega_2}^{V[G]},\in_{\NS_{\omega_1},\mathcal{A}}^{V[G]})$ obtained by forcing over some model $(V,\in)$ of $T$ with
$V[G]$
a generic extension of $(V,\in)$ such that $(V[G],\in)\models \stA$.
Recall that $B\subseteq\bool{Form}_{\bp{\in}}\times 2$ is such that
$\in_{\NS_{\omega_1},\mathcal{A}}=\in_B$.

We show that $R=\bool{AMC}(T,B)$.
\begin{description}
\item[$R$ is consistent] by Schindler and Asper\`o's result \cite{ASPSCH(*)}
\[
\ZFC+\maxA+\MM^{++}+\emph{there are class many Woodin cardinals} 
\]
implies $\stA$.

$\MM^{++}$ is forcible over a model of 
\[
\ZFC+\emph{there is a supercompact} 
\]
and
\[
\ZFC+\maxA+\emph{there are class many Woodin cardinals}
\] 
is preserved in forcing extensions.

\item[$R$ is model complete]
\emph{}

\begin{itemize} 
\item
for any existential $\in_{\NS_{\omega_1},\mathcal{A}}$-formula $\phi(\vec{x})$, the $\Pi_2$-sentence for $\in_{\NS_{\omega_1},\mathcal{A}}$
\[
\forall\vec{x}\,(\phi(\vec{x})\leftrightarrow\theta_\phi(\vec{x}))
\]
is in $R$: it holds in all the structures used to define $R$ (by 
 Thm. \ref{thm:keythmmodcompanHomega2}). 
\item
By Robinson's test those axioms suffice to establish the model completeness of $R$.
\end{itemize}


\item[$(T+T_{\in,B})_{\forall\vee\exists}$ and $R_{\forall\vee\exists}$ are the same]
%
By the very definition of $R$, we get that $R_{\forall\vee\exists}$ is equal to
\begin{equation}\label{eqn:T*Bforallexists}
\bp{\psi\in(\in_B)_{\forall\vee\exists}: \, 
\forall\,(V,\in),\bool{B}\in V\, \qp{\qp{(V,\in)\models T+\bool{B}\text{ is a cba}+\Qp{\stA}_{\bool{B}}=1_{\bool{B}}}\rightarrow \Qp{\psi^{H_{\omega_2}}}_{\bool{B}}=1_{\bool{B}}}}.
\end{equation}

In view of Thm.  \ref{thm:PI1invomega2}
and Lemma  \ref{lem:levabsgen}, the formulae $\psi$ in the set displayed in \ref{eqn:T*Bforallexists} 
are exactly the same $\psi$ which belong to
$(T+T_{\in,B})_{\forall\vee\exists}$:
by Thm. \ref{thm:PI1invomega2} the $\Pi_1$-theory of  $(V,\in_B^V)$ for $(V,\in)$ a model of $T$ is exactly the same $\Pi_1$-theory of $(V[G],\in_B^{V[G]})$  for $G$ $V$-generic for some cba forcing $\stA$ over 
$(V,\in)$; by Lemma \ref{lem:levabsgen} applied in $V[G]$ this $\Pi_1$-theory is exactly equal to that 
of the structure of
$(H_{\omega_2}^{V[G]},\in_B^{V[G]})$.

\end{description}
This immediately gives \ref{Thm:mainthm-1A}$\Longleftrightarrow$\ref{Thm:mainthm-1B}
for $(T+T_{\in,B})$ and $R$.

We are left with the proof of the remaining equivalences between \ref{Thm:mainthm-1A},
\ref{Thm:mainthm-1E}, \ref{Thm:mainthm-1Cbis}, \ref{Thm:mainthm-1C}, 
\ref{Thm:mainthm-1D}, \ref{Thm:mainthm-1Dbis}, \ref{Thm:mainthm-1B}.

\begin{description}

\item[\ref{Thm:mainthm-1A}$\Longrightarrow$\ref{Thm:mainthm-1E}]
By definition of $R$.

\item[\ref{Thm:mainthm-1E}$\Longrightarrow$\ref{Thm:mainthm-1Cbis}]
Given an $\in$-model 
$(V,\in)$ of $T$,
by the results of \cite{FORMAGSHE}, we can find a stationary set preserving 
forcing extension $V[G]$ of $V$ which models $\MM^{++}$.
By the key result of Asper\'o and Schindler \cite{ASPSCH(*)}, 
$V[G]\models\stA$. 
By 
\ref{Thm:mainthm-1E} $(V[G],\in_{\NS_{\omega_1},\mathcal{A}}^{V[G]})$ models $\psi^{H_{\omega_2}^{V[G]}}$,
and we are done.

\item[\ref{Thm:mainthm-1Cbis}$\Longrightarrow$\ref{Thm:mainthm-1C}]
Trivial.

\item[\ref{Thm:mainthm-1C}$\Longrightarrow$\ref{Thm:mainthm-1D}]
By\footnote{$\maxA$ implies that the same assumption used in the cited theorem for $L(\mathbb{R})$ holds 
for $L(\mathcal{A})$.}
 \cite[Thm. 7.3]{HSTLARSON}, if some $P$ forces $\psi^{\dot{H}_{\omega_2}}$,
 we get that $L(\mathcal{A})\models \qp{\Pmax\Vdash \psi^{\dot{H}_{\omega_2}}}$.

\item[\ref{Thm:mainthm-1D}$\Longleftrightarrow$\ref{Thm:mainthm-1Dbis}]
By \cite[Thm. 2.7, Thm. 2.8]{SCHASPBMM*++}.

\item[\ref{Thm:mainthm-1D}$\Longrightarrow$\ref{Thm:mainthm-1B}]
Given some complete $S\supseteq T+T_{\in,B}$, and a model $\mathcal{M}$ of $S$, 
find $\mathcal{N}$ forcing extension of $\mathcal{M}$ which models $\psi^{H_{\omega_2}^{\mathcal{N}}}$.
By Thm. \ref{thm:PI1invomega2} and Levy's absoluteness Lemma \ref{lem:levabsgen}, 
$H_{\omega_2}^{\mathcal{N}}$ models $\psi+S_{\forall\vee\exists}$, and we are done.
\end{description}
\end{proof}


\subsection{Proof of Thm.~\ref{thm:keythmmodcompanHomega2}}
The rest of this section is devoted to the proof of Thm.~\ref{thm:keythmmodcompanHomega2}.

What we will do first is to sketch 
the key intuition on how to define $\bar{D}_{\NS_{\omega_1},\mathcal{A}}$.

\subsubsection{More ideas on the proof of Thm.~\ref{thm:keythmmodcompanHomega2}.}


Recall the notion of generic tameness introduced in Def. \ref{Keyprop:maxUB}.
Let $M$ be a countable transitive model of 
$\ZFC+\maxA+$\emph{there are class many Woodin cardinals} for some generically tame 
$\mathcal{A}$.
Then it will  model that $\mathcal{A}^M$ is generically tame for $M$.


Now assume that
 there is a countable family $\mathcal{A}_M$ of universally Baire sets in $L(\mathcal{A})$
which is such that
$\mathcal{A}^M=\bp{B\cap M: B\in\mathcal{A}_M}$.
Furthermore assume that the map $B\cap M\mapsto B$ extends the identity on $H_{\omega_1}^M$ to an elementary embedding of
\[
(H_{\omega_1}^M\cup\mathcal{A}^M,\in_{\Delta_0}^M)
\]
into 
\[
(H_{\omega_1}^V\cup\mathcal{A}^V,\in_{\Delta_0}^V).
\]

The setup described above is quite easy to realize (for example $M$ could the transitive collapse of
some countable $X\prec V_\theta$ for some large enough $\theta$); in particular for any
$a\in H_{\omega_1}$ and $B_1,\dots,B_k\in \mathcal{A}$, we can find $M$ countable transitive model of a suitable fragment of $\ZFC$ with $a\in H_{\omega_1}^M$
and $\mathcal{A}_M\supseteq \bp{B_1,\dots,B_k}$ countable and $H_{\omega_1}$-closed family of 
sets in $\mathcal{A}$ such that  $\mathcal{A}^M=\bp{B\cap M: B\in\mathcal{A}_M}$.

Letting $B_M=\prod\mathcal{A}_M$, $B_M\in\mathcal{A}$ since $\mathcal{A}$ is countably closed; hence
$(L(\mathcal{A}),\in_{\Delta_0})$ is able to compute correctly whether 
$B_M$ encodes a set $\mathcal{A}_M$ such that the pair $(\mathcal{A}_M,M)$ satisfies the above list of requirements, e.g.:
\begin{itemize}
\item $M$ is a countable transitive model of 
$\ZFC+\maxA+$\emph{there are class many Woodin cardinals}.
\item
the map $B\cap M\mapsto B$ extends the identity on $H_{\omega_1}^M$ to an elementary embedding of
\[
(H_{\omega_1}^M\cup\mathcal{A}^M,\in_{\Delta_0}^M)
\]
into 
\[
(H_{\omega_1}^V\cup\mathcal{A}^V,\in_{\Delta_0}^V).
\]
\end{itemize}


In particular $(L(\mathcal{A}),\in_{\Delta_0})$ correctly computes\footnote{Note that 
\(
(H_{\omega_1}^V\cup\mathcal{A}^V,\in_{\Delta_0}^V)
\)
cannot define $D_{\mathcal{A}}$ since the notion that 
\(
(H_{\omega_1}^M\cup\mathcal{A}^M,\in_{\Delta_0}^M)
\) 
is an elementary substructure of 
\(
(H_{\omega_1}^V\cup\mathcal{A}^V,\in_{\Delta_0}^V)
\) 
cannot be defined in 
\(
(H_{\omega_1}^V\cup\mathcal{A}^V,\in_{\Delta_0}^V).
\)
} 
the set $D_{\mathcal{A}}$ of $M\in H_{\omega_1}$ such that there exists a universally Baire set
$B_M=\prod\mathcal{A}_M$ with the property that the pair $(M,\mathcal{A}_M)$ realizes the above set of requirements.
By $\maxA$, $\bar{D}_{\mathcal{A}}=\Cod_\omega^{-1}[D_{\mathcal{A}}]$ is a universally Baire set 
in $\mathcal{A}$.

Note moreover that $\bar{D}_{\mathcal{A}}$ is defined by a $\in$-formula $\phi^*_\mathcal{A}(x)$ 
in no extra parameters;
in particular for any model $\mathcal{W}=(W,E)$ of $\ZFC+\maxA$, we can define 
$\bar{D}_{\mathcal{A}}$ in
$\mathcal{W}$ and all its properties outlined above will hold relativized to $\mathcal{W}$.

We will consider the set $D_{\NS_{\omega_1},\mathcal{A}}$ 
of $M\in D_{\mathcal{A}}$ such that:
\begin{itemize} 
\item
$(M,\NS_{\omega_1}^M)$ is a $\Pmax$-precondition 
which is $B$-iterable for all $B\in \mathcal{A}_M$ (according to \cite[Def. 4.1]{HSTLARSON});
\item
$j_{0\omega_1}$ is a $\Sigma_1$-elementary embedding of 
$(H_{\omega_2}^M\cup\mathcal{A}^M,\in_{\Delta_1},\NS_{\omega_1}^M)$ into \\
$(H_{\omega_2}^V\cup\mathcal{A},\in_{\Delta_1},\NS_{\omega_1}^V)$
whenever
$\mathcal{J}=\bp{j_{\alpha\beta}:\alpha\leq\beta\leq\omega_1}$ is an iteration of $M$ with 
$j_{0\omega_1}(\NS_{\omega_1}^M)=\NS_{\omega_1}^V\cap j_{0\omega_1}(H_{\omega_2}^M)$.
\end{itemize}

It will take a certain effort to prove that  (assuming $\stA$):
\begin{itemize}
\item for any $A\in H_{\omega_2}$ and
$B\in\mathcal{A}$, we can find $M\in D_{\NS_{\omega_1},\mathcal{A}}$ with $B\in \mathcal{A}_M$, some $a\in H_{\omega_2}^M$,
and an iteration $\mathcal{J}=\bp{j_{\alpha\beta}:\alpha\leq\beta\leq\omega_1}$ of $M$ with 
$j_{0\omega_1}(\NS_{\omega_1}^M)=\NS_{\omega_1}^V\cap j_{0\omega_1}(H_{\omega_2}^M)$  such that
$j_{0\omega_1}(a)=A$.
\item
$D_{\NS_{\omega_1},\mathcal{A}}$ is correctly\footnote{The key point is that we will be able to define by a formula in parameter $D_{\mathcal{A}}$ an element of $\mathcal{A}$ which is exactly $D_{\NS_{\omega_1},\mathcal{A}}$ if $\stA$ holds.} computable in $(L(\mathcal{A}),\in_{\Delta_0})$
using $D_{\mathcal{A}}$.
\end{itemize}
This effort will pay off, since we will then be able to prove the model completeness of the 
$\in_{\NS_{\omega_1},\mathcal{A}}$-theory\footnote{Note that aiming for the model completeness result for the theory of $H_{\omega_2}$ we are loosing a certain flavour of second order logic: 
$\in_{\NS_{\omega_1},\mathcal{A}}$ is a signature considered in a structure 
where universally Baire sets are considered only as parameters, but not as objects over which we can quantify.}
\[
(H_{\omega_2},\in_{\NS_{\omega_1},\mathcal{A}}^V)
\]
using Robinson's test with $\bar{D}_{\NS_{\omega_1},\mathcal{A}}=\Cod_\omega^{-1}[D_{\NS_{\omega_1},\mathcal{A}}]$ being the extension of the formula $\phi_{\NS_{\omega_1},\mathcal{A}}(x)$ mentioned in Thm. \ref{thm:keythmmodcompanHomega2},
and showing that for a quantifier free $\in_{\NS_{\omega_1},\mathcal{A}}\cup\mathcal{A}^V$-formula $\phi(\vec{x},y)$, $A\in H_{\omega_2}$, $B_1,\dots,B_k\in\mathcal{A}$ we have that
\[
(H_{\omega_2}^V,\in_{\NS_{\omega_1},\mathcal{A}}^V,B_1^V,\dots,B_k^V)\models
\exists\vec{x}\,\phi(\vec{x},A)
\]
if and only if 
\begin{quote}
\emph{For all $M\in D_{\NS_{\omega_1},\mathcal{A}}$ and for all $\mathcal{J}$ iteration of $M$ of length $\omega_1$ mapping correctly the nonstationary ideal, if there is $a\in M$ mapped via $\mathcal{J}$ to $A$, then}
\[
(H_{\omega_2}^M,\in_{\NS_{\omega_1},\mathcal{A}}^M,,B_1^V\cap M,\dots,B_k^V\cap M)\models
\exists\vec{x}\,\phi(\vec{x},a).
\]
\end{quote}
A bit of work shows that the latter quoted statement is formalized by a $\Pi_1$-formula for $\in_{\NS_{\omega_1},\mathcal{A}}\cup\bp{B_1,\dots,B_k}$ using the predicate symbol $\Cod_\omega^{-1}[D_{\NS_{\omega_1},\mathcal{A}}]$ together with those appearing in 
$\phi$.

We now get into the details.

\subsubsection{$\mathcal{A}$-correct models}

\begin{definition}
Given $M,N$ iterable structures, 
$M\geq N$ if $M\in (H_{\omega_1})^N$ and there is  an iteration
\[
\mathcal{J}=\bp{j_{\alpha\beta}:\,\alpha\leq\beta\leq\gamma=(\omega_1)^N}
\]
of $M$ with $\mathcal{J}\in N$
such that
\[
\NS_{\gamma}^{M_\gamma}=
\NS_{\gamma}^N\cap M_{\gamma}.
\]
\end{definition}

\begin{notation}
Given a countable family 
$\mathcal{X}=\bp{B_n:n\in\omega}$ of universally Baire sets with each $B_n\subseteq (2^{\omega})^{k_n}$,
we say that $B_\mathcal{X}=\prod_{n\in\omega}B_n\subseteq \prod_n(2^{\omega})^{k_n}$ is a code for 
$\bp{B_n:n\in\omega}$.

Clearly $B_\mathcal{X}$ is a universally Baire subset of the compact Polish space $\prod_n(2^{\omega})^{k_n}$.
\end{notation}

\begin{definition}
Let $\mathcal{A}$ be a generically tame family of $V$.

A transitive model of $\ZFC$ $(M,\in)$ is $\mathcal{A}$-correct if 
there is
$\mathcal{A}_M$ countable family of universally Baire sets in $V$
such that:
\begin{itemize}
\item The map 
\begin{align*}
\Theta_M:&\mathcal{A}_M\to M\\
&A\mapsto A\cap M
\end{align*}
is injective.
\item
$(M,\in)$ models that $\bp{A\cap M: A\in \mathcal{A}_M}$ is the family of sets in $M$ satisfying 
$\phi_{\mathcal{A}}(x)$ in $(M,\in)$.

\item 
\(
(H_{\omega_1}^M\cup\bp{A\cap M: A\in \mathcal{A}_M},\in)\prec
(H_{\omega_1}^V\cup \mathcal{A}^V,\in)
\)\\
via the identity on $H_{\omega_1}$ and $A\cap M\mapsto A$ on $\bp{A\cap M: A\in \mathcal{A}_M}$.
\item If $M$ is countable,
$M$ is $A$-iterable for all $A\in \mathcal{A}_M$.
\end{itemize}

$M$ is absolutely $\mathcal{A}$-correct if 
for all $N\geq M$:
\begin{quote}
\emph{
$N$ is $\mathcal{A}$-correct in $V$ if and only if}
\(
(M,\in)\models N\text{ is $\mathcal{A}^M$-correct}.
\)
\end{quote}
\end{definition}

%

\begin{notation}
$D_{\mathcal{A}}$ denotes the set of countable absolutely $\mathcal{A}$-correct $M$; 
$\bar{D}_\mathcal{A}=\Cod_\omega^{-1}[D_\mathcal{A}]$.

For each $M\in D_\mathcal{A}$, $\mathcal{A}_M$ is a witness that $M$ is $\mathcal{A}$-correct and $B_{\mathcal{A}_M}=\prod\mathcal{A}_M$ is a universally Baire set in $\mathcal{A}$ coding this 
witness.

For universally Baire sets $B_1,\dots,B_k\in\mathcal{A}$, $E_{\mathcal{A},B_1,\dots,B_k}$
denotes the set of $M\in D_\mathcal{A}$ with $B_1,\dots,B_k\in \mathcal{A}_M$ for some witness $\mathcal{A}_M$ that $M\in D_\mathcal{A}$;
$\bar{E}_{\mathcal{A},B_1,\dots,B_k}=\Cod_\omega^{-1}[E_{\mathcal{A},B_1,\dots,B_k}]$.
\end{notation}

\begin{fact}  \label{fac:maxAstabdef}
$(V,\in)$ models $M$\emph{ is countable and (absolutely) $\mathcal{A}$-correct as witnessed by\footnote{Note that the map $M\mapsto\mathcal{A}_M$ can be defined in $V$ but possibly not 
in $L(\mathcal{A})$, however the binary relation $\bp{(M,B): B\in\mathcal{A}\text{ is a witness that $M$ is $\mathcal{A}$-correct}}$ is in $L(\mathcal{A})$. This suffices for all our proofs.} $\mathcal{A}_M$}
if and only if so does 
$(L(\mathcal{A}),\in)$.

Consequently 
the set $D_{\mathcal{A}}$ of countable absolutely $\mathcal{A}$-correct $M$ is properly computed in
$(L(\mathcal{A}),\in)$.

Therefore assuming $\maxA$ 
\[
\bar{D}_\mathcal{A}=\mathrm{Cod}^{-1}[D_{\mathcal{A}}]
\] 
is universally Baire. 


The same holds for $\bar{E}_{\mathcal{A},B_1,\dots,B_k}$ for given universally Baire sets $B_1,\dots,B_k$.
\end{fact}
\begin{proof}
The first part follows almost immediately by the definitions, since the assertion in parameters $B,M$:
\begin{quote}

$B=\prod_{n\in\omega}B_n$ codes an $H_{\omega_1}$-closed family 
$\mathcal{A}_M=\bp{B_n:n\in\omega}$ of sets such that 
\begin{itemize}
\item $M$ is $A$-iterable for all $A\in \mathcal{A}_M$,
\item $M$ models that $\bp{A\cap M: A\in \mathcal{A}_M}$ is the family of sets realizing $\phi_\mathcal{A}(x)$ in $M$ and is $H_{\omega_1}$-closed,
\item
$(H_{\omega_1}^M\cup \bp{A\cap M: A\in \mathcal{A}_M},\in_{\Delta_0}^M)$ embeds elementarily into \\
$(H_{\omega_1}^V\cup \mathcal{A}^V,\in_{\Delta_0}^V)$ via the map extending the identity on 
$H_{\omega_1}^M$ by $A\cap M\mapsto A$,

\end{itemize}

\end{quote}
gets the same truth value in $(V,\in)$ and in $(L(\mathcal{A}),\in)$.
Note that $B=\prod\mathcal{A}_M\in\mathcal{A}$ by countable closure of $\mathcal{A}$.

We conclude that $D_{\mathcal{A}}$ has the same extension in 
$(V,\in)$ and in $(L(\mathcal{A}),\in)$.
By $\maxA$ $\bar{D}_\mathcal{A}$ is universally Baire.


The same argument can be replicated for  $\bar{E}_{\mathcal{A},B_1,\dots,B_k}$.
\end{proof}

Note that while the notion of being absolutely $\mathcal{A}$-correct is a priori not computable in
$(H_{\omega_1}^V\cup \mathcal{A}^V,\in_{\Delta_0}^V)$, if $M\leq N$ are both absolutely $\mathcal{A}$-correct in $V$, there is an elementary embedding of 
$(H_{\omega_1}^N\cup \mathcal{A}^N,\in_{\Delta_0}^N)$ into 
$(H_{\omega_1}^M\cup \mathcal{A}^M,\in_{\Delta_0}^M)$ which belongs to $M$ and is coded by a universally set of $V$ which belongs to $\mathcal{A}_M$ whenever the latter is a witness that $M$ is 
$\mathcal{A}$-correct.

\begin{lemma}\label{lem:UBcorr}
Assume $\NS_{\omega_1}$ is precipitous and there are class many 
Woodin cardinals in $V$.
Let $\delta$ be an inaccessible cardinal in $V$ and $G$ be 
$V$-generic for $\Coll(\omega,\delta)$.
Then $V_\delta$ is absolutely $\mathcal{A}^{V[G]}$-correct in $V[G]$ as witnessed by $\prod\bp{B^{V[G]}:B\in \mathcal{A}^V}$. Furthermore $V_\delta$ is $B^{V[G]}$-iterable for all $B\in \mathcal{A}^V$.
\end{lemma}
\begin{proof}
$\bp{B^{V[G]}:B\in \mathcal{A}^V}$ is a countable family of universally Baire sets in $\mathcal{A}^{V[G]}$ hence its product belongs to $\mathcal{A}^{V[G]}$ by countable closure of $\mathcal{A}$ in $V[G]$. 

By $\maxA$ there is an elementary $j:L(\mathcal{A}^V)\to L(\mathcal{A}^{V[G]})$ which is the identity on 
$H_{\omega_1}$ and maps $B\in\mathcal{A}$ to $B^{V[G]}$ in $\mathcal{A}^{V[G]}$.

By Fact \ref{fac:maxAstabdef} $V$ and $L(\mathcal{A}^V)$ (respectively $V[G]$ and
 $L(\mathcal{A}^{V[G]})$) agree on which elements of $H_{\omega_1}^V$ are $\mathcal{A}$-correct. 

The above grants that the unique missing condition to get the $\mathcal{A}^{V[G]}$-correctness of $V_\delta$ is to check that it is $B^{V[G]}$-iterable for all $B\in\mathcal{A}^V$.
This is a standard argument which can be reconstrued looking for example  at the proof of
\cite[Thm. 4.10]{HSTLARSON}.
\end{proof}

\begin{fact}\label{fac:densityUBcorrect}
$(\maxA)$
Assume $\NS_{\omega_1}$ is precipitous and $\maxA$ holds.
Then for any iterable $M$, $B_1,\dots,B_k\in \mathcal{A}$, there is 
$N\leq M$ such that:
\begin{itemize}
\item
$N$ is absolutely $\mathcal{A}$-correct as witnessed by $\mathcal{A}_N$; 
\item
$B_1,\dots,B_k\in \mathcal{A}_N$.
\end{itemize}
\end{fact}
\begin{proof}
The assumptions grant that whenever $G$ is $\Coll(\omega,\delta)$-generic for $V$,
$V_\delta$ is absolutely $\mathcal{A}^{V[G]}$-correct in $V[G]$ and $B^{V[G]}$ iterable for all
$B\in\mathcal{A}^V$ (i.e. Lemma \ref{lem:UBcorr}).

By \cite[Lemma 2.8]{HSTLARSON}, for any iterable $M\in H_{\omega_1}^V$ there is in $V$ an iteration
 $\mathcal{J}=\bp{j_{\alpha\beta}:\alpha\leq\beta\leq \omega_1^V}$ of $M$ such that
 $\NS_{\omega_1}^V\cap M_{\omega_1}=\NS_{\omega_1}^{M_{\omega_1}}$.
 
By $\maxA$
\[
(L(\mathcal{A}^{V}),\in_{\Delta_0}) \prec 
(L(\mathcal{A}^{V[G]}),\in_{\Delta_0})
\]
via a map which is the identity on $H_{\omega_1}$ and maps $B$ to $B^{V[G]}$ for $B\in\mathcal{A}$.
 Therefore  we have that in $V[G]$ $\bar{E}_{\mathcal{A},B_1,\dots,B_k}^{V[G]}$ is exactly 
 $\bar{E}_{\mathcal{A},B_1^{V[G]},\dots,B_k^{V[G]}}$.
 
 Also for any $B\in\mathcal{A}^V$,
$V_\delta$ is $B^{V[G]}$-iterable, and
$B^{V[G]}\in \mathcal{A}_{V_\delta}=\bp{A^{V[G]}:A\in\mathcal{A}^V}$;
 
Hence for each iterable $M\in H_{\omega_1}^V$ and $B_1,\dots,B_k\in \mathcal{A}^V$,
 $N=V_\delta\leq M$ 
 belongs to $E_{\mathcal{A},B_1,\dots,B_k}^{V[G]}$ as witnessed by 
$\mathcal{A}_{V_\delta}=\bp{U^{V[G]}:U\in \mathcal{A}}$ and
$B_1^{V[G]},\dots,B_k^{V[G]}\in \mathcal{A}^{V[G]}$. 


Note that the $\Sigma_1$-formulae for $\in_{\Delta_0}$ we established to be true in $H_{\omega_1}^{V[G]}\cup\mathcal{A}^{V[G]}$ as witnessed by $V_\delta$ are 
 in parameters $M\in H_{\omega_1}^V$, $B_1^{V[G]},\dots,B_k^{V[G]}, D_{\mathcal{A}}^{V[G]}$ as
$B_1,\dots,B_k$ vary in $\mathcal{A}^V$. 
Hence these formulae reflect to $H_{\omega_1}^V\cup\mathcal{A}^V$.

The Lemma is proved.

\end{proof}

\subsection{Three characterizations of $(*)$-$\mathcal{A}$}


\begin{definition}
An absolutely $\mathcal{A}$-correct $M$ is \emph{$(\NS_{\omega_1},\mathcal{A})$-ec}
if $(M,\in)$ models that $\NS_{\omega_1}$ is precipitous and 
there is a witness $\mathcal{A}_M$ that $M$ is $\mathcal{A}$-correct with the following property:   

\begin{quote}
Assume an iterable $N\leq M$ is absolutely $\mathcal{A}$-correct with a witness $\mathcal{A}_N$ such that $\prod\mathcal{A}_M\in\mathcal{A}_N$.

Then for all iterations
\[
\mathcal{J}=\bp{j_{\alpha\beta}:\alpha\leq\beta\leq\gamma=\omega_1^N}
\] 
in $N$ witnessing
$M\geq N$, we have that $j_{0\gamma}$ defines a $\Sigma_1$-elementary embedding of
\[
(H_{\omega_2}^{M}\cup\mathcal{A}^M,\in_{\Delta_1}^{M},\NS_{\omega_1}^{M})
\]
into 
\[
(H_{\omega_2}^N\cup\mathcal{A}^N,\in_{\Delta_1}^N,\NS_{\omega_1}^N).
\]
\end{quote}

\end{definition}

\begin{remark}\label{rmk:maxUBec}
A crucial observation is that
\emph{``$x$ is $(\NS_{\omega_1},\mathcal{A})$-ec''}
is a property correctly definable in $(H_{\omega_1}\cup\mathcal{A},\in)$ using as parameter the set\footnote{Note however that $D_{\mathcal{A}}$ may not be lightface definable in  $(H_{\omega_1}\cup\mathcal{A},\in)$.} $D_{\mathcal{A}}$.
Therefore (assuming $\maxA$)
\[
D_{\NS_{\omega_1},\mathcal{A}}=\bp{M\in H_{\omega_1}: \, M\text{ is $(\NS_{\omega_1},\mathcal{A})$-ec}}
\]
is such that $\bar{D}_{\NS_{\omega_1},\mathcal{A}}=\Cod_\omega^{-1}[D_{\NS_{\omega_1},\mathcal{A}}]$ 
is a universally Baire set in $\mathcal{A}$.
Moreover letting for $V[G]$ a generic extension of $V$
\[
D_{\NS_{\omega_1},\mathcal{A}^{V[G]}}=\bp{M\in H_{\omega_1}^{V[G]}: \, M\text{ is $(\NS_{\omega_1},\mathcal{A}^{V[G]})$-ec}},
\]
we have that 
\[
\bar{D}_{\NS_{\omega_1},\mathcal{A}}^{V[G]}=\Cod_\omega^{-1}[D_{\NS_{\omega_1},\mathcal{A}^{V[G]}}].
\]
\end{remark}

\begin{theorem}\label{thm:char(*)}
Assume $V$ models $\maxA$.
The following are equivalent:
\begin{enumerate}
\item\label{thm:char(*)-1}
Woodin's axiom $(*)$-$\mathcal{A}$ holds
(i.e. 
$\NS_{\omega_1}$ is precipitous,
and there is an $L(\mathcal{A})$-generic filter $G$ for $\mathbb{P}_{\mathrm{max}}$
such that $L(\mathcal{A})[G]\supseteq\pow{\omega_1}^V$).
\item\label{thm:char(*)-2}
Let $\delta$ be inaccessible.
Whenever $G$ is $V$-generic for $\Coll(\omega,\delta)$,
$V_\delta$ is $(\NS_{\omega_1},\mathcal{A}^{V[G]})$-ec in $V[G]$.

\item \label{thm:char(*)-3}
$\NS_{\omega_1}$ is precipitous and
for all $A\in H_{\omega_2}$, $B\in\mathcal{A}$, there is an $(\NS_{\omega_1},\mathcal{A})$-ec $M$
with witness $\mathcal{A}_M$,
and an iteration $\mathcal{J}=\bp{j_{\alpha\beta}:\,\alpha\leq\beta\leq\omega_1}$ of $M$ 
such that:
\begin{itemize}
\item $A\in M_{\omega_1}$,
\item $B\in\mathcal{A}_M$,
\item $\NS_{\omega_1}^{M_{\omega_1}}=\NS_{\omega_1}\cap M_{\omega_1}$.
\end{itemize}
\end{enumerate}
\end{theorem}

Theorem \ref{thm:char(*)} is the key to the proofs of Theorem~\ref{thm:keythmmodcompanHomega2}
and to the missing implication in the proof of Theorem~\ref{Thm:mainthm-1bis}.

\subsubsection{Proof of Theorem~\ref{thm:keythmmodcompanHomega2}}

The theorem is an immediate corollary of the following:

\begin{lemma}\label{lem:keylemmodcomp(*)}

Let $B_1,\dots,B_k$ be new predicate symbols and 
$S_{\NS_{\omega_1},\mathcal{A},B_1,\dots,B_k}$ 
be the $\in_{\NS_{\omega_1},\mathcal{A}}\cup\bp{B_1,\dots,B_k}$-theory 
$\ZFC_{\NS_{\omega_1},\mathcal{A}}+\stA$ enriched with
the sentences asserting that $B_1,\dots,B_k$ have as extension elements of $\mathcal{A}$.

Let $E_{B_1,\dots,B_k}$ consists of the set of
$M\in D_{\NS_{\omega_1},\mathcal{A}}$ such that:
\begin{itemize}
\item
$M$ is $B_j$-iterable for all $j=1,\dots,k$;
\item 
there is $\mathcal{A}_M$ witnessing $M\in D_{\NS_{\omega_1},\mathcal{A}}$ with 
$B_j\in\mathcal{A}_M$ for all $j$.
\end{itemize}
Let also $\bar{E}_{B_1,\dots,B_k}=\Cod_\omega^{-1}[E_{B_1,\dots,B_k}]$. 

Then 
$S_{\NS_{\omega_1},\mathcal{A},B_1,\dots,B_k}$ 
proves that
$\bar{E}_{B_1,\dots,B_k}$ is in $\mathcal{A}$. 

Moreover: 
\begin{itemize}
\item
Let $S_{\NS_{\omega_1},\mathcal{A},B_1,\dots,B_k,\bar{E}_{B_1,\dots,B_k}}$ be the  natural extension of  
$S_{\NS_{\omega_1},\mathcal{A},B_1,\dots,B_k}$
adding a predicate symbol for $\bar{E}_{B_1,\dots,B_k}$ to  
$\in_{\NS_{\omega_1},\mathcal{A}}\cup\bp{B_1,\dots,B_k}$ and the axioms stating that the interpretation of $\bar{E}_{B_1,\dots,B_k}$ is given by its definition.
\item
Let $T_{\NS_{\omega_1},\mathcal{A},B_1,\dots,B_k,\bar{E}_{B_1,\dots,B_k}}$ be the family
of $\Pi_2$-sentences $\psi$  for $\in_{\NS_{\omega_1},\mathcal{A}}\cup\bp{B_1,\dots,B_k,\bar{E}_{B_1,\dots,B_k}}$ such that $S_{\NS_{\omega_1},\mathcal{A},B_1,\dots,B_k,\bar{E}_{B_1,\dots,B_k}}$ proves $\psi^{H_{\omega_2}}$.
\end{itemize}
Then $T_{\NS_{\omega_1},\mathcal{A},B_1,\dots,B_k,\bar{E}_{B_1,\dots,B_k}}$ proves  that every existential formula for  $\in_{\NS_{\omega_1},\mathcal{A}}\cup\bp{B_1,\dots,B_k}$ is equivalent to a universal formula for  $\in_{\NS_{\omega_1},\mathcal{A}}\cup\bp{B_1,\dots,B_k,\bar{E}_{B_1,\dots,B_k}}$.
\end{lemma}

%
%
%


\begin{proof}
$\bar{E}_{B_1,\dots,B_k}$ is universally Baire and in $\mathcal{A}$ by $\maxA$,
since $E_{B_1,\dots,B_k}$
is definable in $(H_{\omega_1}\cup\mathcal{A},\in)$ with parameters the universally Baire sets
$B_1,\dots,B_k,\bar{D}_{\NS_{\omega_1},\mathcal{A}}$.

Given any $\Sigma_1$-formula $\phi(\vec{x})$ for $\in_{\NS_{\omega_1},\mathcal{A}}\cup\bp{B_1,\dots,B_k}$
mentioning the universally Baire predicates $B_1,\dots,B_k$, we want
to find a universal formula $\psi(\vec{x})$  for  $\in_{\NS_{\omega_1},\mathcal{A}}\cup\bp{B_1,\dots,B_k,\bar{E}_{B_1,\dots,B_k}}$
such that
\[
T_{\bp{B_1,\dots,B_k, \bar{E}_{B_1,\dots,B_k}},\NS_{\omega_1}}\models \forall\vec{x}(\phi(\vec{x})\leftrightarrow \psi(\vec{x})).
\]

Let $\theta_\phi(\vec{x})$ be the formula
asserting:
 
\begin{quote}
\emph{For all $M\in E_{B_1,\dots,B_k}$, for all iterations 
$\mathcal{J}=\bp{j_\alpha\beta:\alpha\leq\beta\leq\omega_1}$ of $M$ such that:}
\begin{itemize}
\item
$\vec{x}=j_{0\omega_1}(\vec{a})$\emph{ for some }$\vec{a}\in M$,
\item
$\NS_{\omega_1}^{j_{0\omega_1}(M)}=\NS_{\omega_1}\cap j_{0\omega_1}(M)$,
\end{itemize}
we have that
\[
(H_{\omega_2}^{M},\in_{\NS_{\omega_1}}^{M},B_j\cap \Cod(r): j=1,\dots,k)\models\phi(\vec{a}).
\]
\end{quote}

More formally:
\begin{align*}
\forall r\, \forall \mathcal{J}&\{&\\
&[&\\
&(r\in \bar{E}_{B_1,\dots,B_k})\wedge&\\
&\wedge\mathcal{J}=\bp{j_\alpha\beta:\alpha\leq\beta\leq\omega_1} \text{ is an iteration of }\Cod(r)\wedge&\\
&\wedge\NS_{\omega_1}^{j_{0\omega_1}(\Cod(r))}=\NS_{\omega_1}\cap j_{0\omega_1}(\Cod(r))\wedge&\\
&\wedge \exists\vec{a}\in\Cod(r)\,(\vec{x}=j_{0\omega_1}(\vec{a}))&\\
&]&\\
&\rightarrow&\\
&(H_{\omega_2}^{\Cod(r)},\in_{\NS_{\omega_1}}^{\Cod(r)},B_j\cap \Cod(r): j=1,\dots,k)\models\phi(\vec{a})&\\
&\}.&
\end{align*}
The above is a $\Pi_1$-formula  for 
$\in_{\Delta_1}\cup\bp{\omega_1,\NS_{\omega_1}}\cup\bp{B_1,\dots,B_k, \bar{E}_{B_1,\dots,B_k}}$.

(We leave to the reader to check that the property 
\begin{quote}
\emph{$\mathcal{J}=\bp{j_\alpha\beta:\alpha\leq\beta\leq\omega_1}$ is an iteration of $M$ such that 
$\NS_{\omega_1}^{j_{0\omega_1}(M)}=\NS_{\omega_1}\cap j_{0\omega_1}(M)$}
\end{quote}
is definable by a $\Delta_1$-property in parameters $M,\mathcal{J}$ in the signature
$\in_{\Delta_0}\cup\bp{\omega_1,\NS_{\omega_1}}$).

Now it is not hard to check that:
\begin{claim}
For all $\vec{A}\in H_{\omega_2}$
\[
(H_{\omega_2}^V,\in_{\NS_{\omega_1}}^V,B_1,\dots,B_k)\models\phi(\vec{A})
\]
if and only if 
\[
(H_{\omega_2},\in_{\NS_{\omega_1}}^V,B_1,\dots,B_k, \bar{E}_{B_1,\dots,B_k})
\models\theta_\phi(\vec{A}).
\]
\end{claim}
\begin{proof}
\emph{}

\begin{description}
\item[$\theta_\phi(\vec{A})\rightarrow \phi(\vec{A})$]
Take any
$M$ and $\mathcal{J}$ satisfying the premises of the implication\footnote{At least one such $M$ exists by $\stA$.} in $\theta_\phi(\vec{A})$, 
Then $(H_{\omega_2}^M,\in_{\NS_{\omega_1},\mathcal{A}^M}^M)\models\phi(\vec{a})$
for some $\vec{a}$ such that $j_{0,\omega_1}(\vec{a})=\vec{A}$ and
$B_j\cap M_{\omega_1}=j_{0\omega_1}(B_j\cap M)$ for all $j=1,\dots,k$.

Since $\Sigma_1$-properties are upward absolute and
$(M_{\omega_1},\in_{\NS_{\omega_1}}^{M_{\omega_1}},B_j\cap M_{\omega_1}:j=1,\dots,k)$
is a $\in_{\NS_{\omega_1}}\cup\bp{B_1,\dots,B_k}$-substructure 
of  $(H_{\omega_2},\in_{\NS_{\omega_1}}^V,B_j:j=1,\dots,k)$ which models $\phi(\vec{A})$, we get that
$\phi(\vec{A})$ holds for $(H_{\omega_2},\in_{\NS_{\omega_1}}^V,B_1,\dots,B_k)$.

\item[$\phi(\vec{A})\rightarrow \theta_\phi(\vec{A})$]
Assume
\[
(H_{\omega_2},\in_{\NS_{\omega_1}}^V,B_1,\dots,B_k)\models\phi(\vec{A}).
\]
Take any $(\NS_{\omega_1},\mathcal{A})$-ec $M\in V$ and any iteration 
$\mathcal{J}=\bp{j_\alpha\beta:\alpha\leq\beta\leq\omega_1}$  of  $M$ witnessing the premises
of the implication\footnote{Note that such $M$ and $\mathcal{J}$ exists by Thm.~\ref{thm:char(*)}(\ref{thm:char(*)-3}) applied to 
$B_1\times\dots\times B_k$ and $\vec{A}$.} in 
$\theta_\phi(\vec{A})$, in particular such that:
\begin{itemize}
\item
$\vec{A}=j_{0\omega_1}(\vec{a})\in M_{\omega_1}$ for some $\vec{a}\in M$, 
\item
$\NS_{\omega_1}^{M_{\omega_1}}=\NS_{\omega_1}\cap M_{\omega_1}$,
\item
$M$ is $B_j$-iterable for $j=1,\dots,k$.
\end{itemize}

Let $G$ be $V$-generic for $\Coll(\omega,\delta)$ with $\delta$ inaccessible.
Then in $V[G]$, $V_\delta$ is $\mathcal{A}^{V[G]}$-correct, by Lemma \ref{lem:UBcorr}.
 
Therefore (since $M$ is $(\NS_{\omega_1},\mathcal{A}^{V[G]})$-ec also in $V[G]$ by $\maxA$),
$V[G]$ models that
$j_{0\omega_1^V}$ is a $\Sigma_1$-elementary embedding of\footnote{Actually is $\Sigma_1$-elementary between the structures $(H_{\omega_2}^M\cup\mathcal{A}^M,\in_{\Delta_0},\NS_{\omega_1}^M)$ and  
$(H_{\omega_2}^V\cup\mathcal{A}^V,\in_{\Delta_0},\NS_{\omega_1}^V)$, but we only need the weaker form of $\Sigma_1$-elementarity described in the proof.}
\[
(H_{\omega_2}^{M},\in_{\NS_{\omega_1}}^{M},B\cap M:B\in\mathcal{A}_M)
\]
into 
\[
(H_{\omega_2}^V,\in_{\NS_{\omega_1}}^V,B:B\in\mathcal{A}_M).
\]
This grants that
\[
(H_{\omega_2}^M,\in_{\NS_{\omega_1}}^M,B\cap M:B\in\mathcal{A}_M)\models\phi(\vec{a}),
\]
as was to be shown.
\end{description}
\end{proof}

The Lemma is proved.

\end{proof}

\begin{remark}
Note that by essentially the same proof we can argue that the $\in_{\NS_{\omega_1}}$-theory of $H_{\aleph_2}\cup\mathcal{A}$ in models of $\ZFC_{\NS_{\omega_1}}+\maxA$ is model complete.
However this result is not relevant for the AMC-spectrum results we are aiming for, since 
$(H_{\aleph_2}\cup\mathcal{A},\in_{\Delta_1},\NS_{\omega_1})$ is not a $\Sigma_1$-substructure
of $(V,\in_{\Delta_1},\NS_{\omega_1})$.
\end{remark}

\subsubsection{Proof of (\ref{thm:char(*)-modcomp-2})$\to$(\ref{thm:char(*)-modcomp-1})
of Theorem~\ref{Thm:mainthm-1bis}}

\begin{proof}
Assume $\delta$ is supercompact, $P$ is a standard forcing notion to force $\MM^{++}$ of size $\delta$ (such as the one introduced in 
\cite{FORMAGSHE} to prove the consistency of Martin's maximum), and $G$ is $V$-generic for $P$; then
$(*)$-$\mathcal{A}$ holds in $V[G]$ by Asper\'o and Schindler's recent breakthrough \cite{ASPSCH(*)}.
By Thm. \ref{thm:PI1invomega2}
$V$ and $V[G]$ agree on the $\Pi_1$-fragment of the $\in_{\NS_{\omega_1}}\cup\mathcal{A}^V$-theory
$\bar{T}$ of $V$, therefore so do 
$H_{\omega_2}^V$ and $H_{\omega_2}^{V[G]}$ (by Lemma \ref{lem:levabsgen}
applied in $V$ and $V[G]$ respectively).

Since $P\in\SSP$
\[
(H_{\omega_2}^V,\in_{\NS_{\omega_1}}^V,A:A\in \mathcal{A}^V)\sqsubseteq
(H_{\omega_2}^{V[G]},\in_{\NS_{\omega_1}}^{V[G]},A^{V[G]}: A\in\mathcal{A}^V).
\]

Now the model completeness of the  $\in_{\NS_{\omega_1}}\cup\mathcal{A}^V$-theory $\bar{S}$ of 
 $H_{\omega_2}^V$  grants that $H_{\omega_2}^V$
is $\bar{T}_\forall$-ec. This gives that:
\[
(H_{\omega_2}^V,\in_{\NS_{\omega_1}}^V,\mathcal{A}^V)\prec_{\Sigma_1}
(H_{\omega_2}^{V[G]},\in_{\NS_{\omega_1}}^{V[G]},A^{V[G]}: A\in\mathcal{A}^V).
\]

Therefore any $\Pi_2$-property for $\in_{\NS_{\omega_1}}\cup\mathcal{A}^V$ with parameters in 
$H_{\omega_2}^V$ which holds in
\[
(H_{\omega_2}^{V[G]},\in_{\NS_{\omega_1}}^{V[G]},A^{V[G]}: A\in\mathcal{A})
\]
also holds in $(H_{\omega_2}^V,\in_{\NS_{\omega_1}}^V,\mathcal{A}^V)$.

Hence in $H_{\omega_2}^V$ it holds characterization (\ref{thm:char(*)-3}) of $(*)$-$\mathcal{A}$  given by Thm.~\ref{thm:char(*)} and we are done.
\end{proof}

\subsubsection{Proof of Theorem \ref{thm:char(*)}}

\begin{definition}\cite[Def. 2.1]{HSTLARSON} \label{def:Pmax}
$\Pmax$ is the subset of $H_{\omega_1}$ given by the pairs $(M,a)$ such that
\begin{itemize}
\item $M$ is iterable, countable, and models Martin's axiom. 
\item $a\in\pow{\omega_1}^M\setminus L(\mathbb{R})^M$, and there exists $r\in \pow{\omega}\cap M$ such that $\omega_1^M=\omega_1^{L[a,r]}$.
\end{itemize}

$(M,a)\leq (N,b)$ if there exists $\mathcal{J}=\ap{j_{\alpha\beta}:\, \alpha\leq\beta\leq\omega_1^M}$ in 
$M$ iteration of $N$ of length $\omega_1^M$ such that
$j_{0\omega_1^M}(b)=a$ and 
$(M,\in)$ models that $\mathcal{J}$ is correct.
\end{definition}

Note that $\Pmax$ is a definable class in $(H_{\omega_1},\in)$; in particular it belongs to any transitive model of 
$\ZFC^-$
containing $\pow{\omega}$.

Our definition of $\Pmax$ is slightly different  than the one given in \cite{HSTLARSON}, but it defines a dense subset of
the poset defined in \cite[Def. 2.1]{HSTLARSON} in view of the 
following\footnote{Much weaker large  cardinals 
assumptions are needed, we don't spell the optimal hypothesis.}:

\begin{fact}
Let $\Pmax^0$ be the forcing defined\footnote{E.g. the forcing $\Pmax$ according to the terminology of \cite{HSTLARSON}.} in \cite[Def. 2.1]{HSTLARSON}.
Assume there are class many Woodin cardinals. 

Then for every condition $(M,I,a)$ in $\Pmax^0$ there is
a condition $(N,b)$ in $\Pmax$ and an iteration $\ap{j_{\alpha\beta}:\, \alpha\leq\beta\leq\omega_1^N}\in N$ of $(M,I)$ 
according to \cite[Def. 1.2]{HSTLARSON}
such that $j_{0\omega_1^N}(I)=\NS_{\omega_1}^N\cap j_{0\omega_1^N}[M]$ and $j_0(a)=b$.
Hence $(N,\NS_{\omega_1}^N,b)$ refines $(M,I,a)$ in $\Pmax^0$.
\end{fact}
\begin{proof}
Let $\gamma>\delta$ be two Woodin cardinals.
Let $X\prec V_\gamma$ be countable with $\delta,(M,I,a)\in X$.
Let $N_0$ be the transitive collapse of $X$, and $N$ be a generic extension of $N_0$ by a forcing collapsing $\delta$
to become $\omega_2$ and forcing $\NS_{\omega_1}$ is precipitous and Martin's axiom. 
Since $\gamma$ is Woodin, there are class many measurables in $V_\gamma$, hence $N_0$ is iterable 
and so is $N$
(by \cite[Thm. 4.10]{HSTLARSON}). 

By \cite[Lemma 2.8]{HSTLARSON} there is in $N$  the required iteration $\ap{j_{\alpha\beta}:\, \alpha\leq\beta\leq\omega_1^N}$
of $(M,I)$ and we can set $b=j_{0\omega_1^N}(a)$.
\end{proof}

In particular (at the prize of assuming the right large cardinal assumptions) the forcings 
$\Pmax$ and $\Pmax^0$ are equivalent as the former sits inside the latter as a dense subset.

This is a key property of $\Pmax$ we will need, and is based on Asper\`o and Schindler result that
$\MM^{++}+\maxA$ implies $\stA$:

\begin{lemma}\label{fac:keyfacdensPmax}
Assume $\maxA$ and there is a supercompact cardinal.
Let $\dot{A}\in L(\mathbb{R})$ and $\dot{\mathcal{N}}\in L(\mathcal{A})$ be the $\Pmax$-canonical names respectively for:
\begin{itemize}
\item
$\bigcup\bp{a: (N,a)\in G}$,
\item
$H_{\omega_2}^{L(\mathcal{A})[G]}\cup\mathcal{A}$, 
\end{itemize}
 whenever $G$ is a 
$\Pmax$-generic filter for $L(\mathcal{A})$.

For any quantifier free formula $\phi(x,y,z)$ for $\in_{\Delta_1}\cup\bp{\NS_{\omega_1}}$ 
and $B\in\mathcal{A}^V$
\[
L(\mathcal{A}^V)\models \qp{(N,a)\Vdash\exists y\in \dot{\mathcal{N}}\phi(\dot{A},y,\check{B})}
\] 
if and only if the set $D_\phi$ of 
$(M,e)$ such that 
\begin{itemize}
\item
$M$ is $B$-iterable,
\item
$(H_{\omega_2}^M\cup\mathcal{A}^M,\in_{\Delta_1}^M,\NS_{\omega_1}^M)
\models\exists y\, \phi(e,y,B\cap M)$,
\end{itemize}
is dense below $(N,a)$.
\end{lemma}
\begin{proof}
Let $G$ be a $\Pmax$-generic filter for $L(\mathcal{A}^V)$ and $A=\dot{A}_G$.
Then (by \cite[Lemma 2.7]{HSTLARSON})
\[
G=\bp{(N,a): \, \exists \mathcal{J}_N\,\NS\text{-correct iteration of $N$ mapping $a$ to $A$}}.
\]
If some $(M,a)\in G$ is such that:
\begin{itemize}
\item
$M$ is $B$-iterable,
\item 
\(
(H_{\omega_2}^M\cup\mathcal{A}^M,\in_{\Delta_1}^M,\NS_{\omega_1}^M)\models\exists y\phi(a,y,B\cap M).
\)
\end{itemize}
Let
\[
\mathcal{J}_M=\bp{j_{\alpha,\beta}:M_\alpha\to M_\beta:\, \alpha\leq\beta\leq\omega_1}.
\]
Then 
\[
\mathcal{M}=(H_{\omega_2}^{M_{\omega_1}}\cup\mathcal{A}^{M_{\omega_1}},\in_{\Delta_1}^{M_{\omega_1}},\NS_{\omega_1}^{M_{\omega_1}})\models\exists y\,\phi(A,y,B\cap M_{\omega_1}).
\]
This yields that 
\[
\mathcal{N}=(H_{\omega_2}^{L(\mathcal{A})[G]}\cup\mathcal{A},\in_{\Delta_1}^{L(\mathcal{A})[G]},\NS_{\omega_1}^{L(\mathcal{A})[G]})\models\exists y\phi(A,y,B),
\]
since it is not hard to check\footnote{If $\phi(x_1,\dots,x_n,y_1,\dots,y_k)$ is provably $\Delta_1(\ZFC^-)$, one can prove by an induction on its syintactic complexity that for each $A_1,\dots,A_n\in H_{\omega_2}^{M_{\omega_1}}$ and each $B_1,\dots,B_k\in\mathcal{A}^{M_{\omega_1}}$
$\phi(A_1,\dots,A_n,B_1,\dots,B_k)$ holds in $\mathcal{N}$ if and only if 
$\phi(A_1,\dots,A_n,B_1\cap M_{\omega_1},\dots,B_k\cap M_{\omega_1})$ holds in $\mathcal{M}$.} that $\mathcal{M}\sqsubseteq\mathcal{N}$ via the map which is the identity on $H_{\omega_2}^{L(\mathcal{A})[G]}$ and maps $B\cap M_{\omega_1}$ to $B$ on 
$\mathcal{A}^{M_{\omega_1}}$.

Otherwise note that
\[
D_\phi\cup\bp{(M,b): \forall (N,c)\leq (M,b)\, (N,c)\not\in D_\phi)}
\]
is dense in $\Pmax$ and belongs to $L(\mathcal{A})$.
Hence for some $(M,a)\in G$, 
$L(\mathcal{A})$ models that 
for all $(N,c)\leq (M,a)$ which are $B$-iterable
\[
 (H_{\omega_2}^N\cup\mathcal{A}^N,\in_{\Delta_1}^N,\NS_{\omega_1}^N)\not\models\exists y\,\phi(c,y,B\cap N).
\]
If the Lemma fails we can find $(N,c)\leq (M,a)$ in the above set such that
\[
L(\mathcal{A})\models\qp{(N,c)\Vdash_{\Pmax}\exists y\in\dot{\mathcal{N}}\,\phi(\dot{A},y,\check{B})}.
\]
In particular for any $(P,d)\leq (N,c)$ we have that 
\begin{equation}\label{eqn:densDphi1}
 (H_{\omega_2}^P\cup\mathcal{A}^P,\in_{\Delta_1}^P,\NS_{\omega_1}^P)\not\models\exists y\,\phi(d,y, B\cap P),
\end{equation}
and
\begin{equation}\label{eqn:densDphi2}
L(\mathcal{A})\models\qp{(P,d)\Vdash_{\Pmax}\exists y\in\dot{\mathcal{N}}\,\phi(\dot{A},y,\check{B})}.
\end{equation}

Fix in $V$ 
\[
\mathcal{K}=\bp{k_{\alpha,\beta}:N_\alpha\to N_\beta:\, \alpha\leq\beta\leq\omega_1}.
\]
$\NS$-correct iteration of $N$. Let $A=k_{0\omega_1}(c)$.

Now let $\gamma$ be a supercompact cardinal, $\delta>\gamma$ be inaccessible,
and $H$ be $V$-generic for $\Coll(\omega,\delta)$.
Then in $V[H]$ we can find $K$ $V$-generic for some stationary set preserving forcing of $V$ collapsing $\gamma$ to become $\omega_2$, together with $\MM^{++}$ (and therefore $\stA$ by Asper\`o and Schindler's result).
Then $(V_\delta[K],A)$ is a $\Pmax$-condition in $V[H]$ refining $(N,c)$ (as witnessed by $\mathcal{K}$), hence such that
\[
 (H_{\omega_2}^{V_\delta[K]}\cup\mathcal{A}^{V_\delta[K]},\in_{\Delta_1}^{V_\delta[K]},\NS_{\omega_1}^{V_\delta[K]})
 \not\models\exists y\,\phi(A,y,B^{V_\delta[K]})
\]
by \ref{eqn:densDphi1}.
On the other hand since $(V_\delta[K],A)$ models $\stA$ and $A\in\pow{\omega_1}\setminus L(\mathcal{A}^{V[K]})$, 
\[
G_A=\bp{(P,d):\,\exists \mathcal{J}\text{ $\NS$-correct iteration of $P$ mapping $d$ to $A$}}
\]
is $L(\mathcal{A}^{V[K]})$-generic for $\Pmax$, with $(N,c)$ belonging to $G_A$; therefore
\[
(H_{\omega_2}^{L(\mathcal{A}^{V[K]})[G_A]}\cup\mathcal{A}^{V_\delta[K]},\in_{\Delta_1}^{L(\mathcal{A}^{V[K]})[G_A]},\NS_{\omega_1}^{L(\mathcal{A}^{V[K]})[G_A]})\models\exists y\phi(A,y,B^{V_\delta[K]})
\]
by \ref{eqn:densDphi2}.
Since $H_{\omega_2}^{L(\mathcal{A}^{V[K]})[G_A]}=H_{\omega_2}^{V[K]}$, we have reached a contradiction.
\end{proof}

We can now prove Thm. \ref{thm:char(*)}.
\begin{proof}

\begin{description}

\item[(\ref{thm:char(*)-1})
implies (\ref{thm:char(*)-2})]
Let $G$ be $V$-generic for $\Coll(\omega,\delta)$.
By Lemma \ref{lem:UBcorr},
$V_\delta$ is absolutely $\mathcal{A}^{V[G]}$-correct in $V[G]$ as witnessed by 
$\bp{B^{V[G]}:B\in \mathcal{A}^V}=\mathcal{A}_V=\bp{B_n^{V[G]}:n\in\omega}$
and $B^{V[G]}$-iterable for all $B\in\mathcal{A}^V$.

\begin{claim}
$V_\delta$ is $(\NS_{\omega_1},\mathcal{A}^{V[G]})$-ec in $V[G]$ as witnessed by $\mathcal{A}_V$.
\end{claim}

\begin{proof}
Let in $V[G]$ $B_{\mathcal{A}_V}=\prod_{n\in\omega}B_n^{V[G]}$ be the universally Baire set in $\mathcal{A}^{V[G]}$ coding $\mathcal{A}_V$.

Let (by Fact \ref{fac:densityUBcorrect}) $N\leq V_\delta$ in $V[G]$ be absolutely $\mathcal{A}^{V[G]}$-correct, $B_{\mathcal{A}_V}$-iterable,
with 
$B_{\mathcal{A}_V}\in \mathcal{A}_N$ for some $\mathcal{A}_N$ countable subset of 
$\mathcal{A}^{V[G]}$ witnessing that 
$N$ is $\mathcal{A}^{V[G]}$-correct.

Then it is not hard to check that
\begin{equation}\label{eqn:elem1implies2char*}
(H_{\omega_1}^V\cup\mathcal{A}^V,\in_{\Delta_1}^V)
\prec (H_{\omega_1}^N\cup\mathcal{A}^N,\in_{\Delta_1}^N)
\end{equation}
via the map extending the identity on $H_{\omega_1}^V$ by $B\mapsto B^{V[G]}\cap N$ for $B\in\mathcal{A}^V$.
This holds since $N\in D_\mathcal{A}^{V[G]}$ and $\bp{B^{V[G]}\cap N:\, B\in \mathcal{A}^V}$ in $N$.

Let 
\[
\mathcal{J}=\bp{j_{\alpha,\beta}:\alpha\leq\beta\leq\gamma=(\omega_1)^N}\in N
\] 
be an iteration witnessing
$V_\delta\geq N$ in $V[G]$.

We must show that 
\[
j_{0\gamma}:H_{\omega_2}^V\cup\mathcal{A}^V\to H_{\omega_2}^N\cup\mathcal{A}^N
\]
is $\Sigma_1$-elementary for $\in_{\Delta_1}\cup\bp{\NS_{\omega_1}}$.

By simple coding tricks (e.g. coding any finite tuple of elements of $H_{\omega_2}$ by a subset of $\omega_1$ via the map $\Cod_{\omega_1}$ and any finite tuple of elements of $\mathcal{A}$ by their product), it suffices to check that for any 
$\Sigma_1$-formula  $\phi(x,y)$ for $\in_{\Delta_1}\cup\bp{\NS_{\omega_1}}$ 
$A\in \pow{\omega_1}^V$, $B\in\mathcal{A}^V$
\[
(H_{\omega_2}^N\cup\mathcal{A}^N,\in_{\Delta_1}^N,\NS_{\omega_1}^N)
\models\phi(j_{0\gamma}(A),B^{V[G]}\cap N).
\]
if and only if 
\[
(H_{\omega_2}^V\cup\mathcal{A}^V,\in_{\Delta_1}^V, \NS_{\omega_1}^V)\models
\phi(A,B).
\]

Now by $\stA$ in $V$,
\[
G_A=\bp{(N,a): \exists\mathcal{J}\, \NS\text{-correct iteration of $N$ mapping $a$ to $A$}}
\]
is $L(\mathcal{A})$-generic for $\Pmax$.

Recall the set $D_\phi$ defined in \ref{fac:keyfacdensPmax} and the dense set
\[
E_\phi=D_\phi\cup\bp{(M,b): \forall (N,c)\leq (M,b),\, (N,c)\not\in D_\phi}
\]
which are both elements of $L(\mathcal{A})^V$, since
both sets are the image under $\Cod_\omega$ of universally Baire sets in $\mathcal{A}^V$.
Hence (by \ref{eqn:elem1implies2char*}) 
we can argue in $V[G]$ that 
$(N,j_{0,\gamma}(A))$ 
witnesses that any $(M,b)\in G_A\cap E_\phi\subseteq E_\phi^{V[G]}$ is not in $E_\phi^{V[G]}\setminus D_\phi^{V[G]}$ (since $(M,b)\geq (V_\delta,A)\geq (N,j_{0,\gamma}(A))$.
Therefore any such  $(M,b)\in G_A\cap E_\phi$ is in $D_\phi^{V[G]}\cap V=D_\phi$.

We conclude that 
\[
(H_{\omega_2}^V\cup \mathcal{A}^V,\in_{\Delta_1}^V, \NS_{\omega_1}^V)\models\phi(A,B),
\]
by Lemma \ref{fac:keyfacdensPmax}.
\end{proof}

\item[(\ref{thm:char(*)-2})
implies (\ref{thm:char(*)-3})]
Our assumptions grants that the set 
\[
D_{\mathcal{A}}=
\bp{M\in H_{\omega_1}^V: M\text{ is absolutely $\mathcal{A}^{V}$-correct}}
\]
is coded by a universally Baire set $\bar{D}_\mathcal{A}$ in $V$. 
Moreover we also get that whenever $G$ is $V$-generic for 
$\Coll(\omega,\delta)$,
the lift $\bar{D}_{\mathcal{A}}^{V[G]}$ of $\bar{D}_\mathcal{A}$ to $V[G]$ codes
\[
D_{\mathcal{A}^{V[G]}}^{V[G]}=\bp{M\in H_{\omega_1}^{V[G]}: M\text{ is absolutely $\mathcal{A}^{V[G]}$-correct}}.
\]

By (\ref{thm:char(*)-2}) we get that $V_\delta\in D_{\NS_{\omega_1},\mathcal{A}^{V[G]}}^{V[G]}$.

By Fact \ref{fac:densityUBcorrect}
\[
(H_{\omega_1}^V\cup\mathcal{A}^V,\in_{\Delta_0}^V)
\]
models 
\begin{quote} 
\emph{For all iterable $M$ 
there exists an absolutely $\mathcal{A}$-correct structure 
$\bar{M}\leq M$ with $\prod\mathcal{A}_M\in \mathcal{A}_{\bar{M}}$}.
\end{quote}
Again since 
\[
(H_{\omega_1}^V\cup\mathcal{A}^V,\in_{\Delta_0}^V)\prec(H_{\omega_1}^{V[G]}\cup\mathcal{A}^{V[G]},\in_{\Delta_0}^{V[G]}),
\]
and the latter is first order expressible in the predicate $\bar{D}_\mathcal{A}\in \mathcal{A}^V$, we get that
\[
(H_{\omega_1}^{V[G]}\cup\mathcal{A}^{V[G]},\in_{\Delta_0}^{V[G]})
\]
models that  there is 
an absolutely $\mathcal{A}$-correct structure 
$N\leq V_\delta$ with 
\[
\prod\mathcal{A}_{V_\delta}=\prod\bp{B^{V[G]}:B\in\mathcal{A}^V}\in \mathcal{A}_N.
\]

Let $\mathcal{J}=\bp{j_{\alpha\beta}:\,\alpha\leq\beta\leq\gamma=\omega_1^N}\in H_{\omega_2}^N$ be an iteration witnessing $N\leq V_\delta$.

Now for any $A\in \pow{\omega_1}^V$ and $B\in\mathcal{A}^V$
\[
(H_{\omega_2}^{N}\cup\mathcal{A}^N,\in_{\Delta_1}^{N},\NS_{\gamma}^{N})
\]
models

\begin{quote}
\emph{There exists
an $(\NS_{\omega_1},\mathcal{A}^{V[G]})$-ec structure $M$ with $B^{V[G]}\cap N\in\mathcal{A}_M$ and an $\NS$-correct iteration 
$\bar{\mathcal{J}}=\bp{\bar{j}_{\alpha\beta}:\,\alpha\leq\beta\leq\gamma}$ of $M$ such that
$\bar{j}_{0\gamma}(A)=j_{0\gamma}(A)$}.
\end{quote}
This statement is witnessed exactly by $V_\delta$ in the place of $M$ (since $B=B^{V[G]}\cap V_\delta$ is in $\mathcal{A}^V$ and 
$\mathcal{A}^{V[G]}_{V_\delta}=\bp{B^{V[G]}:\, B\in\mathcal{A}^V}$),
and $\mathcal{J}$ in the place of $\bar{\mathcal{J}}$.

Since $V_\delta$ is $(\NS_{\omega_1},\mathcal{A}^{V[G]})$-ec in $V[G]$ we get that
$j_{0\gamma}\restriction H_{\omega_2}^V\cup\mathcal{A}^V$ is $\Sigma_1$-elementary for $\in_{\Delta_1}\cup\bp{\NS_{\omega_1}}$
 between $H_{\omega_2}^V\cup\mathcal{A}^V$ and $H_{\omega_2}^N\cup\mathcal{A}^N$.

Hence
\[
(H_{\omega_2}^V\cup\mathcal{A}^V,\in_{\Delta_1}^{V},\NS_{\gamma}^{V})
\]
models 
\begin{quote}
\emph{There exists 
an $(\NS_{\omega_1}^V,\mathcal{A}^{V})$-ec structure $M$ with $B\in\mathcal{A}_M$ and an iteration 
$\bar{\mathcal{J}}=\bp{\bar{j}_{\alpha\beta}:\,\alpha\leq\beta\leq(\omega_1)^V}$ of $M$ such that
$\bar{j}_{0\omega_1}(a)=A$ and 
$\NS_{\omega_1}^{\bar{j}_{0\omega_1}(M)}=\NS_{\omega_1}^V\cap \bar{j}_{0\omega_1}(M)$}.
\end{quote}

\item[(\ref{thm:char(*)-3})
implies (\ref{thm:char(*)-1})]

The key point is to prove that if $M$ is $(\NS,\mathcal{A})$-ec, $a\in\pow{\omega_1}^M\setminus L(\mathbb{R})^M$, $D$ is a dense open set of $\Pmax$ such that $D=\Cod_\omega[\bar{D}]$ for some
$\bar{D}\in \mathcal{A}_M$, then there is some $(M_0,a_0)\geq (M,a)$ with $(M_0,a_0)\in D\cap M$.

Once this is achieved (\ref{thm:char(*)-3}) gives immediately the desired conclusion.

So pick $D,a$ as above. 
Find $(P,b)\leq (M,a)\in D$ 
and $(N,c)\leq (P,b)$ with
$N$ $\mathcal{A}$-ec and $\mathcal{A}_M\subseteq \mathcal{A}_N$.

Then $N$ models (as witnessed by $(P,b)$) the $\Sigma_1$-statement for $\in_{\Delta_1}\cup\bp{\NS_{\omega_1}}$ in parameters $\bar{D},c$:
\begin{quote}
There exists a $\NS$-correct iteration of some $(M_0,a_0)\in D^N$
which maps $a_0$ to $c$.
\end{quote}
Since $M$ is $(\NS,\mathcal{A})$-ec and there is a unique $\NS$-correct iteration of $M$ which maps $a$ to $c$, we get that $M$ models
\begin{quote}
There exists a $\NS$-correct iteration of some $(M_0,a_0)\in D^M=D\cap M$
which maps $a_0$ to $a$.
\end{quote}

Now the rest of the argument is routine and is left to the reader.

\end{description}
\end{proof}


\section*{Some comments and open questions}

We believe there is still room to improve the model completeness results one can predicate from Woodin's axiom $(*)$. Specifically we conjecture the following:

\begin{conjecture}
Assume $\maxUB$ and $\stUB$. Let $\Theta$ be the supremum of the ordinals $\alpha$ which are the surjective image of some $\phi:2^\omega\to\Ord$ which exists in $L(\UB)$.
Then the theory of $L_\Theta(\UB^{\omega_1})$ is model complete for the signature $\in_{\Delta_1}\cup\bp{\omega_1,\NS_{\omega_1},\UB}$ where $\UB$ is a predicate symbol which
detects which subsets of $2^\omega$ are universally Baire.
\end{conjecture}

Note that the above conjecture does not say that $\Theta$ (the supremum of Wadge rank of universally Baire sets) is regular in $V$ assuming $V$ models $\maxUB+\stUB$ (it could certainly have cofinality $\omega_2$ in $V$, and we conjecture it could not have cofinality 
$\omega_1$). In fact an argument of Woodin combined with the results of \cite{VIAMM+++} should bring that $\Theta$ cannot be regular in models of 
$\MM^{+++}$.


A subtle question is whether for $\ZFC$ the AMC-spectrum and the model companionship spectrum can be distinct. We conjecture the following:
\begin{conjecture}
Assume $S+T_{\in,A}$ has a model companion for some $\in$-theory $S\supseteq\ZFC$ and some $A$
with $\in_A\supseteq\in_{\Delta_0}$.
Then $A\in\SpecAMC{S}$.
\end{conjecture}

Another set of open questions is whether the (generically invariant for suitable classes of forcings) theories of $H_{\aleph_2}$ under bounded category forcing axioms (or under iterated resurrection axioms) isolated in \cite{VIAASP,VIAAUD14} produce model complete theories for some signature extending $\in_{\Delta_0}$ and for some theory extending $\ZFC+$\emph{large cardinals} with some $\Sigma_2$-sentence not holding 
assuming $\stUB$ (for example the assertion that canonical functions are not dominating modulo clubs, or some other $\Sigma_2$-sentence whose negation can only be forced using a stationary set preserving forcing which cannot be proper).

We believe it can also be interesting to investigate more the notions of model companionship spectrum or AMC-spectrum. For example: given a countable theory, analyze the descriptive set theoretic complexity of the partial order given by its AMC-spectrum under inclusion; can this be a useful measure to compare the complexity of countable theories? Are there other model-theoretic properties of a mathematical theory sensitive to the signature for which the spectrum makes sense? In which case what type of information can we extract from this spectrum?

\bibliographystyle{plain}
	\bibliography{Biblio}

\end{document}